\documentclass[10pt]{amsart}

\usepackage[T1]{fontenc}
\usepackage{amsmath,amssymb,mathtools}
\usepackage{microtype}
\usepackage{tikz-cd}
\usepackage[margin=0.95in]{geometry}
\usepackage[colorlinks=true,citecolor=blue,linkcolor=blue,urlcolor=blue]{hyperref}

\newtheorem{theorem}{Theorem}[section]
\newtheorem{proposition}[theorem]{Proposition}
\newtheorem{lemma}[theorem]{Lemma}
\newtheorem{corollary}[theorem]{Corollary}
\theoremstyle{definition}
\newtheorem{definition}[theorem]{Definition}
\theoremstyle{remark}
\newtheorem{remark}[theorem]{Remark}
\newtheorem{example}[theorem]{Example}
\newtheorem*{theoremA}{Theorem A}
\newtheorem*{theoremB}{Theorem B}
\newtheorem*{theoremC}{Theorem C}

\newcommand{\SH}{\mathrm{SH}}
\newcommand{\THH}{\operatorname{THH}}
\newcommand{\TR}{\operatorname{TR}}
\newcommand{\Sp}{\operatorname{Sp}}
\newcommand{\Map}{\operatorname{Map}}
\newcommand{\Lace}{\operatorname{Lace}}
\newcommand{\Perf}{\operatorname{Perf}}
\newcommand{\Res}{\operatorname{Res}}

\newcommand{\Spec}{\operatorname{Spec}}
\newcommand{\one}{\mathbf 1}
\newcommand{\Gm}{\mathbb G_m}
\newcommand{\Aone}{\mathbb A^1}
\newcommand{\Mcompact}{M^{c}}

\newcommand{\eps}{\varepsilon}
\newcommand{\cC}{\mathcal C}
\newcommand{\cD}{\mathcal D}
\newcommand{\cM}{\mathcal M}

\newcommand{\cF}{\mathcal F}
\newcommand{\Kzero}{K_0}

\title[Intrinsic Restriction Traces and Toric Dynamics]
{Intrinsic Restriction Traces and Toric Dynamics}
\author{Haoyang Liu}
\address{Department of Mathematics, University of California, Santa Barbara,
CA 93106, USA}
\email{haoyangliu@ucsb.edu}
\author{Tianle Liu}
\address{University of Southern California}
\email{tianleli@usc.edu}

\date{}
\makeatletter
\@ifundefined{subjclassname@2020}{%
  \@namedef{subjclassname@2020}{%
    \textup{2020} Mathematics Subject Classification}%
}{}
\makeatother
\hypersetup{
  pdftitle={Intrinsic Restriction Traces and Toric Dynamics},
  pdfauthor={Haoyang Liu, Tianle Liu}
}

\begin{document}

\begin{abstract}
We prove Morita invariance of the
Campbell--Lind--Malkiewich--Ponto--Zakharevich restriction-system trace after
passage to perfect modules.  It therefore defines an intrinsic integral
restriction trace for an exact endofunctor of a small idempotent-complete
stable \(\infty\)-category.  On \(\pi_0\), the \(m\)-th ghost is the laced trace
of the \(m\)-fold iterate, compatibly with Frobenius.  For lattice-graded
algebras, the ghost targets have twisted cocenters in degree zero; over the
open parameter torus, this applies to the cyclic bimodule of
Dinkins--Karpov--Krylov.  For finite monomial endomorphisms of toric varieties,
we construct motivic restriction classes whose ghosts are sums over cones
fixed by the iterates.  In one example, two classes have the same first ghost,
while their second ghosts differ after rational Betti realization.
\end{abstract}

\maketitle
\enlargethispage{6pt}

\section{Introduction and main results}
\label{sec:introduction}

Campbell--Lind--Malkiewich--Ponto--Zakharevich construct a genuine
restriction-system refinement of the Dennis trace from a twisting of spectral
Waldhausen categories \cite{CLMPZ}.  Their construction packages all positive
iterates and the maps between their cyclic levels, but is formulated using a
spectral Waldhausen presentation.  We prove that its perfect-module
restriction system and trace arrow are Morita invariant, and hence depend
only on a small idempotent-complete stable \(\infty\)-category \(\cC\) and an
exact endofunctor \(F\colon\cC\to\cC\). Put
\[
 \cM_F(x,y)=\Map(Fx,y).
\]
This is the graph coefficient of \(F\), and enriched co-Yoneda gives
\(\cM_F^{\circ m}\simeq\cM_{F^m}\)
\cite[Example~4.6 and the proof of Proposition~7.15]{CLMPZ}.
The composite of the intrinsic restriction trace with the \(m\)-th ghost map is
naturally equivalent to the one-fold coefficient trace for \(F^m\),
precomposed with the \(m\)-fold iteration map on laced \(K\)-theory.  On
\(\pi_0\), this coefficient trace agrees with the
Harpaz--Nikolaus--Saunier laced trace \cite{HNS}.  Descending the CLMPZ trace
gives the intrinsic morphism of genuine restriction systems.  Restriction
from \(C_{mn}\)-fixed points to \(C_n\)-fixed points, followed by ordered
block composition, gives the Frobenius maps and the trace-compatibility
squares.

For lattice-graded algebras, this construction yields twisted cocenters at
the ghost levels, and Frobenius acts on the grading parameter by
\(z\mapsto z^m\).  For toric monomial endomorphisms, the \(m\)-th ghost of a
canonical compactly supported class is a sum over the cones fixed by the
\(m\)-th iterate.  In the example below, two toric varieties over
\(\mathbb C\) have equal first ghosts but unequal second ghosts after rational
Betti realization.

\subsection{The intrinsic restriction trace}

Fix a Grothendieck universe \(\mathcal U\), and write
\[
 \operatorname{End}
 (\operatorname{Cat}^{\mathrm{perf}}_{\infty,\mathcal U})
 =\operatorname{Fun}
 \bigl(B\mathbb N,
       \operatorname{Cat}^{\mathrm{perf}}_{\infty,\mathcal U}\bigr).
\]
We work in this \(\mathcal U\)-bounded category; the constructions are
compatible with enlargement of \(\mathcal U\).  A \emph{truncation set} is a
nonempty subset \(S\subset\mathbb N_{>0}\) closed under positive divisors, and
\[
 S/m=\{n\geq1\mid mn\in S\}.
\]
We write \(R_S\) for the limit in
\cite[equation~(8.9) and Definition~8.10]{CLMPZ}, restricted to the indices in
\(S\); the precise convention is given in
\S\ref{subsec:restriction-system-conventions}.  The adjective \emph{integral}
means that all
positive-integer cyclic levels are retained, rather than only a
\(p\)-typical subsystem.

The spectrum \(K^{\mathrm{lace}}(\cC,\cM_F)\) is the algebraic
\(K\)-theory of the stable category of pairs \((x,\alpha)\) with
\(\alpha\colon Fx\to x\).  For \(m\geq1\), the iteration functor
\[
 P_m\colon\operatorname{Lace}(\cC,\cM_F)
 \longrightarrow
 \operatorname{Lace}(\cC,\cM_{F^m})
\]
replaces \(\alpha\) by its ordered \(m\)-fold composite
\(F^mx\to x\); we use the same notation for the induced maps on
\(K\)-theory and on \(K_0\).  Projection to the \(m\)-th fixed-point level,
followed by the unwinding equivalence of
\cite[Proposition~7.6]{CLMPZ}, defines the underlying ghost
\[
 \overline g_{m,S}\colon
 R_S\bigl(\mathbf E(\cC,F)\bigr)
 \longrightarrow
 \THH(\cC,\cM_{F^m}).
\]

\begin{theoremA}
Let \(\cC\) be a \(\mathcal U\)-small idempotent-complete stable
\(\infty\)-category and let \(F\colon\cC\to\cC\) be exact.
\begin{enumerate}
\item[(i)] The CLMPZ trace arrows obtained by applying the perfect-module
construction to strict spectral presentations \((A,f)\) of \((\cC,F)\)
factor through Morita localization.  Their target systems are naturally
equivalent to those built from the graph coefficients
\(X_f(a,b)=A(fa,b)\).  Consequently there is an intrinsic genuine
restriction system \(\mathbf E(\cC,F)\), and for every nonempty truncation set
\(S\) a natural transformation
\begin{equation}\label{eq:intro-theorem-a}
 \operatorname{tr}^{\mathrm{res,int}}_S\colon
 K^{\mathrm{lace}}(\cC,\cM_F)
 \longrightarrow
 R_S\bigl(\mathbf E(\cC,F)\bigr),
\end{equation}
compatible under restriction of truncation sets.

\item[(ii)] For every
\(\xi\in K_0(\operatorname{Lace}(\cC,\cM_F))\) and every \(m\in S\),
\[
 \pi_0\overline g_{m,S}
 \Bigl((\pi_0\operatorname{tr}^{\mathrm{res,int}}_S)(\xi)\Bigr)
 =
 \operatorname{tr}^{\mathrm{lace}}\bigl(P_m(\xi)\bigr)
 \quad\text{in }\pi_0\THH(\cC,\cM_{F^m}).
\]

\item[(iii)] For every \(m\in S\), there is a natural Frobenius map
\[
 \operatorname{Fr}_{m,S}\colon
 R_S\bigl(\mathbf E(\cC,F)\bigr)
 \longrightarrow
 R_{S/m}\bigl(\mathbf E(\cC,F^m)\bigr)
\]
and a homotopy-commutative square
\[
\begin{tikzcd}[column sep=large,row sep=large]
 K^{\mathrm{lace}}(\cC,\cM_F)
 \arrow[r,"\operatorname{tr}^{\mathrm{res,int}}_S"]
 \arrow[d,"K(P_m)"'] &
 R_S\bigl(\mathbf E(\cC,F)\bigr)
 \arrow[d,"\operatorname{Fr}_{m,S}"]\\
 K^{\mathrm{lace}}(\cC,\cM_{F^m})
 \arrow[r,"\operatorname{tr}^{\mathrm{res,int}}_{S/m}"'] &
 R_{S/m}\bigl(\mathbf E(\cC,F^m)\bigr).
\end{tikzcd}
\]
\end{enumerate}
\end{theoremA}

For strict morphisms of twistings of spectral Waldhausen categories, CLMPZ
naturality assembles the trace maps into a functor
\[
 \mathcal T_{\mathrm{CLMPZ}}\colon
 \mathsf{Tw}^{\mathrm{str},\circ}_{\mathrm{SpWald}}
 \longrightarrow
 \operatorname{Fun}(\Delta^1,\operatorname{RSys}^g_\infty)
\]
in the arrow \(\infty\)-category of genuine restriction systems.  Both
vertices are Morita invariant: the target by coefficient Morita invariance,
and the source because its levels are the \(K\)-theories of the cyclic laced
categories \(\mathcal L_r(\cC,F)\).  The arrow therefore descends along
\[
 N\mathsf{Pres}_{\mathcal U}[W_{\mathrm{Mor}}^{-1}]
 \simeq
 \operatorname{Fun}
 \bigl(B\mathbb N,
       \operatorname{Cat}^{\mathrm{perf}}_{\infty,\mathcal U}\bigr)
\]
and applying \(R_S\) gives the intrinsic trace
\eqref{eq:intro-theorem-a}.

Appendix~\ref{app:split-strictification} replaces the
pseudofunctorial extension-of-scalars construction by a strict
zero-preserving model compatible with powers; the required cofibrant models
and coherence are established in Appendix~\ref{app:technical-coherence}.
The descent is Theorem~\ref{thm:arrow-valued-morita-descent}, and its
compatibility with ghosts and Frobenius is recorded in Corollary~
\ref{cor:integral-intrinsic-ghost-formula} and Theorem~
\ref{thm:restriction-trace-frobenius}.

\subsection{Toric periodicity and higher ghosts}

For toric monomial maps, the higher levels have a direct geometric
interpretation.  Throughout the
motivic applications, \(\SH(-)\) denotes the \(\mathbb A^1\)-invariant
six-functor category of Khan--Ravi \cite{KhanRavi}.

Let \(k\) be a field, let \(N\simeq\mathbb Z^d\), and put \(T=T_N\).  For a
cone \(\sigma\) of a finite fan \(\Sigma\), write \(T_\sigma\subset T\) for
the stabilizer of the orbit \(O_\sigma=T/T_\sigma\), and put
\[
 E_T(T_\sigma)=(BT_\sigma\longrightarrow BT)_!\one_{BT_\sigma}
 \quad\text{in }\SH(BT).
\]
Let \(A\colon N\to N\) be injective with finite cokernel and suppose that
\(A_{\mathbb R}\) maps every cone of \(\Sigma\) \emph{onto} a cone of
\(\Sigma\).
Then \(A\) permutes the cones and induces a finite toric endomorphism
\(f_A\) and an exact functor \(F_A=(B\psi_A)^*\) on \(\SH(BT)\), where
\(\psi_A\) is the associated torus isogeny.

\begin{theoremB}
Let \(\cC\subset\SH(BT)\) be a small \(F_A\)-stable thick subcategory
containing the compactly supported motive of \(X_\Sigma\), its orbit motives,
and the filtration-stage motives used in localization.  For every nonempty
truncation set \(S\), the compactly supported semilinear lacing of \(f_A\)
has an intrinsic integral restriction class over \(S\).  For every
\(m\in S\), its \(m\)-th ghost is
\begin{equation}\label{eq:intro-cones-fixed-by-iterates}
 \sum_{\substack{\sigma\in\Sigma\\A^m\sigma=\sigma}}
 \operatorname{tr}^{\mathrm{lace}}
 \bigl(E_T(T_\sigma),\lambda_{\sigma,m}\bigr)
 \quad\text{in }\pi_0\THH(\cC,\cM_{F_A^m}).
\end{equation}
Here \(\lambda_{\sigma,m}\) is proper pullback along
\(BT_\sigma\to B\psi_A^{-m}(T_\sigma)\), after the canonical identification
\(F_A^mE_T(T_\sigma)\simeq E_T(\psi_A^{-m}(T_\sigma))\).
\end{theoremB}

The formula also applies when the characteristic divides the degree of
\(\psi_A\); the scalloped-stack six-functor formalism accommodates the
resulting non-smooth multiplicative-type kernels.

Over \(\mathbb C\), take
\(
 A=\left(\begin{smallmatrix}0&2\\2&0\end{smallmatrix}\right)
\)
and the map \((x,y)\mapsto(y^2,x^2)\) on both
\(\mathbb A^2\setminus\{0\}\) and \(\Gm^2\).  The two resulting restriction
classes have equal first ghosts, because the one-dimensional boundary orbits
are exchanged, whereas rational Betti realization sends the difference of
their second ghosts to \(6\).

A related reconstruction theorem treats \(A=rB\), where \(r>1\) and
\(B\Sigma=\Sigma\).  With the coherent fixed-point realization data of
Section~\ref{sec:finite-reconstruction}, finitely many geometric
fixed-point realizations recover the cycle type on maximal cones.  For scalar
power maps they interpolate
\[
 P_{\Sigma,H}(z)=
 \sum_{\substack{\sigma\in\Sigma\\H\subset T_\sigma}}
 (z-1)^{\dim O_\sigma}.
\]
A separating family of realized first-ghost values instead recovers the
embedded stabilizer multiset for any fixed finite collection of fans, but
neither method recovers cone incidence.  See
Theorems~\ref{thm:finite-ptypical-stabilizer-reconstruction}
and~\ref{thm:toric-fixed-point-cycle-index}, with the additive
limitation recorded in Proposition~\ref{prop:additive-ceiling}.

The toric trace formula and its restriction-class refinement are proved in
Theorem~\ref{thm:cones-fixed-by-iterates-trace} and
Corollary~\ref{cor:intrinsic-cones-fixed-by-iterates-restriction}.
Example~\ref{ex:toric-exchanged-strata-second-ghost} gives the
exchanged-strata example.

\subsection{Graded algebras and the DKK construction}

Dinkins--Karpov--Krylov introduce a cyclic bimodule over a partial
compactification of the grading-parameter torus
\cite[Section~1.6.4]{DKK}.  On the open torus it becomes the graph bimodule
of the universal grading automorphism.  Motivated by this example, we record
the restriction-system construction for an arbitrary graded algebra.

\begin{theoremC}
Let \(R\) be a commutative ring, let \(\Lambda\) be a finite-rank free
abelian group, and let \(A\) be a unital \(\Lambda\)-graded \(R\)-algebra.
For every commutative \(R\)-algebra \(S\) and every grading parameter
\(z\in\operatorname{Hom}(\Lambda,\Gm)(S)\), the twist
\[
 F_z=(-)_{\sigma_z^{-1}}\colon\Perf(A\otimes_RS)\longrightarrow
 \Perf(A\otimes_RS)
\]
 determines an intrinsic genuine restriction system together with a trace map
 from laced \(K\)-theory.  For every \(m\geq1\), the zeroth homotopy group of
 the target of its \(m\)-th ghost is canonically identified with the twisted cocenter
\[
 (A\otimes_RS)\big/\operatorname{span}_S
 \{ab-b\sigma_{z^m}(a)\},
\]
and its intrinsic Frobenius implements \(z\mapsto z^m\).
\end{theoremC}

For the Dinkins--Karpov--Krylov coefficient, Theorem~C applies after
restriction to the open parameter torus.  In this case, the first ghost target
is their localized coefficient Hochschild object, while the \(m\)-th ghost
target uses the \(m\)-th power of the universal grading automorphism.  See
\S\ref{subsec:dkk-open-torus-coefficient}.

\subsection{Motivic trace evaluation and logarithmic zeta functions}

For a smooth proper \(k\)-scheme \(X\) with endomorphism \(\varphi\), the
construction gives an intrinsic restriction class.  If
\(E_X=\Sigma^\infty_+X\) lies in a small thick rigid symmetric-monoidal
subcategory \(\mathcal C_k\subset\SH(k)\) equipped with the trace-evaluation
datum of Definition~\ref{def:trace-evaluation-datum}, then
\[
 \operatorname{gh}^{GW}_m
 \bigl(\mathcal Z^{\TR}_{X,\varphi;\mathcal C_k}\bigr)
 =\operatorname{Tr}_{\SH(k)}(\varphi^m)
\]
for every \(m\geq1\), and hence
\[
 \operatorname{dlog}\zeta^{\mathbb A^1}_{X,\varphi}(t)
 =\sum_{m\geq1}
 \operatorname{gh}^{GW}_m
 \bigl(\mathcal Z^{\TR}_{X,\varphi;\mathcal C_k}\bigr)t^{m-1}.
\]
The class is Frobenius-compatible.  For odd \(q\), let
\(X=\operatorname{Spec}\mathbb F_{q^2}\) over \(\mathbb F_q\), and let
\(\varphi\) be the relative \(q\)-power Frobenius.
Corollary~\ref{cor:no-gw-zeta} gives a realized ghost sequence
outside the image of the standard big-Witt ghost map
\[
 W(GW(\mathbb F_q))\longrightarrow
 GW(\mathbb F_q)^{\mathbb N_{>0}}.
\]

\subsection{Relation to the literature}

Our starting point is the CLMPZ genuine restriction-system trace
\cite[Theorem~8.8, Definition~8.13, and Proposition~8.20]{CLMPZ}.  The Morita
localization results of BGT
\cite[Theorems~3.1 and~4.23 and Proposition~4.28]{BGT} allow us to pass from
its spectral models to stable \(\infty\)-categories.  For the coefficient
terms, we use the HNS coend model
\cite[Definitions~2.2 and~4.4]{HNS} together with the spectral calculation of
\cite[Theorem~4.15 and Corollary~4.20]{CLMPZ}.  Strict functoriality and tensor
powers are handled using the projective diagram model and Eilenberg--Watts
formalism
\cite[Proposition~1.3.4.25]{LurieHA}
\cite[Theorems~4.8.4.1 and~4.8.4.6]{LurieHA};
Appendix~\ref{app:split-strictification} gives the resulting strict
perfect-module model.

The universal additive HNS laced trace
\cite[Section~4.4 and Theorem~4.45, with Theorem~3.11 and
Corollary~4.41 as inputs]{HNS} interprets each
ghost on \(\pi_0\) for the left-representable coefficients used here.  Exact
functors to \(\Perf(R)\) give linear realizations by Morita
invariance and the zero-simplex comparison of
\cite[Theorems~1.2 and~7.8]{CampbellPonto}.  Ramzi constructs a more general
trace map \cite[Theorem~2.1 and Construction~2.8]{Ramzi};
Definition~\ref{def:trace-evaluation-datum} records the comparison with
cyclic-bar zero-simplices used in the motivic applications.

Krause--McCandless--Nikolaus construct shifts and Frobenius maps for polygonic
\(\TR\) \cite[Example~2.28 and Theorem~C']{KMN}.
Agarwal--Campbell--Manco--Ponto--Sun construct the corresponding Frobenius and
Verschiebung operations on reduced \(K_0\) for ordinary rings
\cite[Theorems~1.3 and~1.4]{AgarwalCampbellMancoPontoSun}.  Our Frobenius
square is spectrum-valued, natural for exact endofunctors of small stable
\(\infty\)-categories, and compatible with all restriction levels.  Saunier's
thesis announces, for arbitrary laced
categories \((\cC,M)\), a genuine-polygonic refinement of the HNS trace
\cite[Proposition~7.16 and Theorems~7.18, 7.25, 7.28, and~8.4]{SaunierThesis}.
Its main proofs and cyclotomic comparison are part of work in preparation
\cite{HNRSII,HNRSIII}\cite[p.~113]{SaunierTHHNotes}.  For graph coefficients
of exact endofunctors, which are left-representable, the genuine construction
used here is supplied by CLMPZ.

Semilinear endomorphism \(K\)-theory goes back to Grayson and Braunling
\cite{GraysonSemilinear,Braunling}.  The periodic interpretation is related
to the Fuller trace of Malkiewich--Ponto
\cite[Theorem~1.1 and Corollary~1.3]{MalkiewichPontoPeriodic} and the CLMPZ
zeta applications \cite[Theorems~9.22 and~9.33]{CLMPZ}.  Hilman--Ramzi and
Chan--Gerhardt--Klang give related traces at a fixed finite group
\cite{HilmanRamzi,ChanGerhardtKlang}; the present system treats all cyclic
levels simultaneously, with restriction and Frobenius maps.

\subsection{Acknowledgements}
The authors would like to thank Vasily Krylov for his interest and discussion about this work. 
Generative AI provided valuable feedback during the preparation of this manuscript and assisted in identifying and addressing gaps in the proofs. The authors take full responsibility for the correctness of all mathematical content.

\section{The CLMPZ restriction system}
\label{sec:graph-restriction}

We recall the CLMPZ genuine restriction system and specialize it to the graph
coefficient of an exact endofunctor.

\subsection{Genuine restriction systems and truncation limits}
\label{subsec:restriction-system-conventions}

We use the restriction systems of \cite[Definition~8.1]{CLMPZ}.  Thus a
genuine restriction system is a sequence \(E_\bullet=\{E_n\}\) of orthogonal
\(C_n\)-spectra equipped, for every \(r,s\geq1\), with compatible maps
\[
 \Phi^{C_r}E_{rs}\longrightarrow E_s
\]
that are equivalences when the source is computed by left-derived geometric
fixed points.  CLMPZ construct a stable model structure on pre-restriction
systems whose weak equivalences are the termwise genuine stable equivalences
\cite[Proposition~A.7]{CLMPZ}.  We write \(\operatorname{RSys}^g\) for the
ordinary category of genuine restriction systems and maps of pre-restriction
systems.  We write
\[
 \operatorname{RSys}^g_\infty
\]
for the full \(\infty\)-subcategory of the localization of the model category
of pre-restriction systems spanned by the genuine restriction systems.  In
particular, equivalences in \(\operatorname{RSys}^g_\infty\) are detected at
every genuine cyclic level.  We use \(\mathbf E(A,f)\) for the model attached
to a strict spectral presentation and reserve \(\mathbf E(\cC,F)\) for the
intrinsic
value obtained after Morita descent.

Throughout, in genuine equivariant stable \(\infty\)-categories,
\((-)^H\) denotes categorical fixed points and \(\Phi^H\) denotes geometric
fixed points.  We write
\[
 V_n\colon \operatorname{Sp}^{C_n}\longrightarrow\operatorname{Sp}
\]
for the functor forgetting the \(C_n\)-action.  The canonical inclusion of
categorical fixed points will therefore be written \(E^{C_n}\to V_nE\).
Here \emph{point-set} refers to constructions in orthogonal spectra or spectral
Waldhausen categories before localization.  Raw fixed points are taken on the
indicated cofibrant or fibrant models, where they compute the corresponding
intrinsic fixed-point functors.  We denote homotopy fixed points by
\((-)^{hH}\).

For a nonempty truncation set \(S\), let \(\mathcal I_S\) be the category
whose objects are the integers in \(S\), with one arrow \(n\to m\) when
\(m\mid n\).  If \(n=rs\), the arrow from the term indexed by \(rs\) to the
term indexed by \(s\) is
\begin{equation}\label{eq:derived-divisibility-map}
 E_{rs}^{C_{rs}}
 \simeq
 \bigl(E_{rs}^{C_r}\bigr)^{C_s}
 \longrightarrow
 \bigl(\Phi^{C_r}E_{rs}\bigr)^{C_s}
 \longrightarrow
 E_s^{C_s},
\end{equation}
where the middle arrow is the canonical comparison from categorical to
geometric fixed points.  This is the divisibility diagram of
\cite[equation~(8.9)]{CLMPZ}; the resulting underived and derived limits are
defined in \cite[Definition~8.10]{CLMPZ}.  We define
\begin{equation}\label{eq:derived-truncation-limit}
 R_S(E_\bullet):=\lim_{\mathcal I_S}E_n^{C_n}
\end{equation}
in spectra.  A model before localization is obtained by fibrantly replacing
the restriction system in the model structure of
\cite[Proposition~A.7]{CLMPZ} and then taking raw categorical fixed points
and the homotopy limit.  Thus \(R_{\mathbb N_{>0}}(E_\bullet)\) is the CLMPZ
\(\TR\).  Since categorical fixed points and limits are functors of genuine
equivariant stable \(\infty\)-categories, this construction defines a functor
\[
 R_S\colon\operatorname{RSys}^g_\infty\longrightarrow\Sp.
\]
We reserve \(R_S^{\mathrm{un}}\) for the corresponding underived
limit before fibrant replacement; there is a canonical map
\(R_S^{\mathrm{un}}(E)\to R_S(E)\).

For \(n\geq1\), the principal truncation set generated by \(n\) is
\[
 \langle n\rangle
 =\{d\in\mathbb N_{>0}\mid d\text{ divides }n\}.
\]

\subsection{Twistings and graph coefficients}
\label{subsec:twistings-graph-coefficients}

Following \cite[Definitions~2.16 and~4.5]{CLMPZ}, a twisting of spectral
categories consists of two spectral functors with common source,
\[
 f,g\colon A\longrightarrow C,
\]
and determines the spectral \(A\)-bimodule
\[
 {}_fC_g(a,b)=C(fa,gb).
\]
For a strict spectral endofunctor \(f\colon A\to A\), the twisting
\[
 {}_fA/A_{\operatorname{id}}
\]
has graph coefficient
\begin{equation}\label{eq:clmpz-coefficient}
 X_f(a,b):=A(fa,b).
\end{equation}
When \(A\) carries the structure of a spectral Waldhausen category and \(f\)
is exact, the objects of its twisted-endomorphism category are arrows
\begin{equation}\label{eq:clmpz-left-laced-object}
 \alpha\colon fa\longrightarrow a.
\end{equation}

We call a spectral category--bimodule pair
\((A,M)\) \emph{admissible} if the mapping spectra of \(A\) and the values
of \(M\) are cofibrant.  If the pair is pointed, we also require a unique
chosen zero object \(0\), cofibrations
\(\mathbb S\to A(a,a)\) for \(a\ne0\), and
\(M(0,a)=0=M(a,0)\).  Morphisms are spectral functors together with maps of
bimodules, strictly preserving the chosen zero object in the pointed case.
We write \(\mathsf{Pair}^{\operatorname{adm}}\) for this category.

A compatible spectral Waldhausen presentation
\((\widehat{\cC},\widehat F)\) of \((\cC,F)\) is \emph{admissible} if:

\begin{enumerate}
\item[\textup{(H1)}]
\(\widehat{\cC}\) is pointwise cofibrant and its weak equivalences are exactly
the maps inverted by stable localization.

\item[\textup{(H2)}]
The localization of the endofunctor diagram presents \((\cC,F)\).

\item[\textup{(H3)}]
The twisted-endomorphism category of the chosen enhancement presents the
laced category, and
\[
 K\operatorname{End}
 \bigl({}_{\widehat F}\widehat{\cC}/
       \widehat{\cC}_{\operatorname{id}}\bigr)
 \xrightarrow{\ \simeq\ }
 K^{\operatorname{lace}}(\cC,\cM_F).
\]

\item[\textup{(H4\(_{\mathrm{red}}\))}]
The chosen zero object is unique and has literal zero mapping spectra; the
input spectral category and its graph coefficient form an admissible pair
as above.  The
\(w_\bullet S_\bullet\)- and \(\Sigma_\Delta\)-objects are then the standard
ones of the CLMPZ construction.

\item[\textup{(H5)}]
For each \(n\geq1\), the derived coefficient power
\[
 X_{\widehat F}^{\mathbb L\odot n}
\]
presents \(\cM_{F^n}\).
\end{enumerate}

We call a spectral Waldhausen category \(A\) \emph{zero-reduced} if its chosen
zero object \(0\) is the only zero object of the Waldhausen base and
\(A(0,a)=0=A(a,0)\) as literal equalities of spectra for every \(a\).  A
strict endofunctor twisting \((A,H)\) is zero-reduced when \(A\) is
zero-reduced and \(H\) strictly preserves its chosen zero object.
Lemma~\ref{lem:zero-reduction} supplies a functorial Morita-equivalent
zero-reduced replacement.  Under \textup{(H1)}--\textup{(H3)} and
\textup{(H5)}, the cofibrant replacement \(Q_{\mathrm{red}}\) of
Theorem~\ref{thm:reduced-bar-replacement} supplies
\textup{(H4\(_{\mathrm{red}}\))}.

The \(n\)-th coefficient level of the CLMPZ construction is the cyclic bar
with \(n\) cyclically ordered copies of \(X_f\), equipped with the
\(C_n\)-action rotating the coefficient slots.  We denote it by
\[
 \THH^{(n)}(A;X_f).
\]
For an admissible presentation, the natural norm diagonal of
\cite[Proposition~2.19]{CLMPZ} supplies the restriction arrows and gives the
genuine restriction system
\[
 \mathbf E(A,f)
 =\{\THH^{(n)}(A;X_f)\}_{n\geq1}.
\]
Its topological restriction homology is
\begin{equation}\label{eq:restriction-system-tr}
 \TR_f^{\operatorname{res}}(A)=\TR(A;X_f).
\end{equation}

\subsection{The laced \texorpdfstring{\(K\)}{K}-theory source}

Throughout the intrinsic discussion, \(\Map\) denotes the mapping spectrum;
its underlying mapping space is \(\Omega^\infty\Map\).  Let \(\cC\) be a
small stable \(\infty\)-category.  A categorical bimodule on \(\cC\) is a
functor
\[
 \cM\colon\cC^{\mathrm{op}}\otimes\cC\longrightarrow\Sp
\]
that is exact separately in both variables.  Following
\cite[Definition~2.2]{HNS}, the pair \((\cC,\cM)\) is a laced category.

\begin{definition}
\label{def:laced-category}
The category
\[
 \operatorname{Lace}(\cC,\cM)
\]
is the pullback, along the diagonal \(\cC\to\cC\times\cC\), of the
bifibration obtained by unstraightening \(\Omega^\infty\cM\); see
\cite[Construction~2.5]{HNS}.  Its objects are pairs
\[
 (x,\alpha),
 \qquad
 \alpha\in\Omega^\infty\cM(x,x).
\]
It is stable, and the forgetful functor to \(\cC\) is exact.
\end{definition}

We write
\[
 K^{\operatorname{lace}}(\cC,\cM)
 :=K\bigl(\operatorname{Lace}(\cC,\cM)\bigr).
\]
Coefficient topological Hochschild homology is the coend
\[
 \THH(\cC,\cM):=\int^{x\in\cC}\cM(x,x),
\]
which is equivalently computed by the coefficient cyclic bar construction.
Harpaz--Nikolaus--Saunier construct a natural laced trace
\begin{equation}\label{eq:laced-trace}
 \operatorname{tr}^{\operatorname{lace}}\colon
 K^{\operatorname{lace}}(\cC,\cM)
 \longrightarrow
 \THH(\cC,\cM).
\end{equation}
Now suppose that \(\cC\) is idempotent-complete and that
\(F\colon\cC\to\cC\) is exact.  Its graph bimodule is
\begin{equation}\label{eq:graph-bimodule}
 \cM_F(x,y):=\Map(Fx,y).
\end{equation}
Thus a laced object in \((\cC,\cM_F)\) is a pair
\[
 (x,\alpha),
 \qquad
 \alpha\colon Fx\longrightarrow x.
\]

\subsection{Composition powers and iteration}

Composition of categorical bimodules is the relative tensor product.  We use
the convention
\[
 (\cM\circ\mathcal N)(x,z)
 =\int^{y\in\cC}\cM(x,y)\otimes\mathcal N(y,z).
\]

\begin{proposition}
\label{lem:coefficient-composition}
Let \(F,G\colon\cC\to\cC\) be exact.  Enriched co-Yoneda gives a natural
equivalence
\begin{equation}\label{eq:coefficient-composition}
 \cM_F\circ\cM_G\simeq\cM_{G\circ F}.
\end{equation}
In particular,
\[
 \cM_F^{\circ m}\simeq\cM_{F^m}
\]
for every \(m\geq1\).

Let \(B\mathbb N\) denote the one-object category with endomorphism monoid
\(\mathbb N\), and let
\[
 X_F\colon B\mathbb N\longrightarrow
 \operatorname{Cat}^{\operatorname{perf}}_\infty
\]
be the diagram that carries the generator to \(F\).  We use the right-lax
cone convention in which the structure map of an object is \(Fx\to x\).  The
arrow-fibration description of \cite[Construction~2.5]{HNS} gives a natural
equivalence
\begin{equation}\label{eq:laced-category-right-lax-limit}
 \operatorname{Lace}(\cC,\cM_F)
 \simeq
 \operatorname{RLim}^{\operatorname{lax}}(X_F).
\end{equation}
For \(m\geq1\), let
\[
 \mu_m\colon B\mathbb N\longrightarrow B\mathbb N,
 \qquad 1\longmapsto m.
\]
Restriction of right-lax cones along \(\mu_m\), followed by
\eqref{eq:coefficient-composition}, defines an exact functor
\begin{equation}\label{eq:power-functor}
 P_m\colon
 \operatorname{Lace}(\cC,\cM_F)
 \longrightarrow
 \operatorname{Lace}(\cC,\cM_{F^m}).
\end{equation}
On objects,
\[
 P_m(x,\alpha)=(x,\alpha^{[m]}),
\]
where
\begin{equation}\label{eq:iterated-lacing-arrow}
 \alpha^{[m]}
 =\alpha\circ F(\alpha)\circ\cdots\circ F^{m-1}(\alpha)
 \colon F^mx\longrightarrow x.
\end{equation}
The construction is natural in morphisms of endofunctor diagrams.
\end{proposition}

\begin{proof}
For \(x,z\in\cC\), enriched co-Yoneda gives
\[
 \int^{y\in\cC}
 \Map(Fx,y)\otimes\Map(Gy,z)
 \simeq
 \Map(GFx,z),
\]
which proves \eqref{eq:coefficient-composition}.  The right-lax limit
description identifies an object with a pair
\((x,\alpha\colon Fx\to x)\).  Precomposition with \(\mu_m\) replaces its
structure arrow by the ordered composite
\eqref{eq:iterated-lacing-arrow}.  Since cofiber sequences in the category of
laced objects are detected on underlying objects, \(P_m\) is exact.
\end{proof}

For every nonempty truncation set \(S\) and every \(m\in S\),
Corollary~\ref{cor:integral-intrinsic-ghost-formula} identifies on \(\pi_0\)
the composite of the restriction trace with the \(m\)-th ghost map with
\(\operatorname{tr}^{\operatorname{lace}}\circ K(P_m)\).

\section{Morita descent of the CLMPZ trace}
\label{sec:morita-descent}

We prove that the CLMPZ trace arrow descends from strict spectral
presentations \((A,f)\) to intrinsic pairs \((\cC,F)\).  The required
cofibrant models and strictification results are collected in
Appendices~\ref{app:technical-coherence}
and~\ref{app:split-strictification}.

\subsection{The trace arrow on spectral presentations}
\label{subsec:strict-arrow-construction}

Fix universes \(\mathcal U\in\mathcal V\in\mathcal W\), where
\(\mathcal U\) bounds the input presentations, \(\mathcal V\) the strictified
perfect-module diagrams, and \(\mathcal W\) the CLMPZ constructions and their
resulting restriction systems.  We suppress these bounds from the notation
and form comparisons in a common larger universe when necessary.
Let \(\operatorname{SpCat}^{\Sigma}_{\mathcal U}\) be the BGT model of
\(\mathcal U\)-small categories enriched in symmetric spectra, and let
\(\mathbb O\) be the orthogonal realization functor of
Lemma~\ref{lem:orthogonal-morita-model}.  Let
\[
 \mathsf{Pres}_{\mathcal U}
 =
 \operatorname{Fun}
 (B\mathbb N,\operatorname{SpCat}^{\Sigma}_{\mathcal U})
\]
be its category of strict \(B\mathbb N\)-diagrams.  Its objects are pairs
\(x=(A,f)\) consisting of a symmetric spectral category and a strict spectral
endofunctor, and its strict morphisms satisfy \(uf=gu\) on the nose.  Let
\(W_{\mathrm{Mor}}\) be the objectwise Morita equivalences.  We apply
orthogonal realization to obtain
\(\mathbb O x=(\mathbb O A,\mathbb O f)\), and within point-set arguments
continue to denote this pointwise-cofibrant diagram by \((A,f)\).  Thus \(\mathcal P_A\),
\(A(f-,-)\), and the CLMPZ constructions refer to the orthogonal model.  The
natural equivalence
\eqref{eq:orthogonal-realization-perfect} identifies its intrinsic
perfect-module diagram with that of the original symmetric presentation.

\begin{proposition}
\label{prop:endofunctor-rigidification}
There is a natural equivalence
\begin{equation}\label{eq:endofunctor-diagram-rigidification}
 N\mathsf{Pres}_{\mathcal U}
 [W_{\mathrm{Mor}}^{-1}]
 \simeq
 \operatorname{Fun}
 \bigl(N(B\mathbb N),
       \operatorname{Cat}^{\mathrm{perf}}_{\infty,\mathcal U}\bigr).
\end{equation}
Thus strict spectral endofunctor diagrams present all homotopy-coherent
endomorphism diagrams of \(\mathcal U\)-small idempotent-complete stable
\(\infty\)-categories.
\end{proposition}

\begin{proof}
 Apply Lemma~\ref{lem:orthogonal-morita-model}(3) with
 \(I=B\mathbb N\).  Its proof first uses the projective DK model to rectify
 the diagram and then applies the BGT Morita localization objectwise.  The
 ordinary functor \(\mathbb O\) sends the resulting strict presentation to a
 strict orthogonal diagram without changing the represented object of
 \(\operatorname{Cat}^{\mathrm{perf}}_\infty\).
\end{proof}

We use Lemma~\ref{lem:pseudonatural-pair-morita-localization} to
rectify pseudonatural comparisons before localization.  We also suppress the
nerve and write \(B\mathbb N\) for the indexing \(\infty\)-category.

We describe the source levels using cyclic laced categories.

\begin{definition}
\label{def:cyclic-laced-category}
For \(r\geq1\), let \(\mathcal L_r(\cC,F)\) be the right-lax limit of the
directed \(r\)-cycle whose vertices are copies of \(\cC\) and whose edge
functors are \(F\).  Equivalently, by the arrow-fibration description of
\cite[Construction~2.5]{HNS}, its objects are cyclic strings
\[
 (x_i,\alpha_i)_{i\in\mathbb Z/r},
 \qquad
 \alpha_i\colon F(x_i)\longrightarrow x_{i+1}.
\]
It is a stable \(\infty\)-category, and cofiber sequences are computed
coordinatewise.  Cyclic rotation defines an exact \(C_r\)-action on
\(\mathcal L_r(\cC,F)\), and there is a canonical equivalence
\[
 \mathcal L_1(\cC,F)
 \simeq\operatorname{Lace}(\cC,\cM_F).
\]
An equivalence of endofunctor pairs induces a \(C_r\)-equivariant exact
equivalence of the corresponding cyclic laced categories.
\end{definition}

For \(x=(A,f)\), let \(\mathcal P_x=\mathcal P_A\) be the spectral
Waldhausen category of \(\mathcal U\)-small cofibrant perfect right
\(A\)-modules, with the enrichment of
Lemma~\ref{lem:many-object-perfect-model}, and let \(T_x=f_!\) be extension
of scalars along \(f\).  The
zero reduction of Lemma~\ref{lem:zero-reduction} gives
\((\mathcal P_x^\circ,T_x^\circ)\), which has a unique chosen zero object
preserved on the nose.  The composition and intertwining maps are the coherent
isomorphisms
\[
 v_!u_!\xRightarrow{\ \sim\ }(vu)_!,
 \qquad
 u_!f_!\xRightarrow{\ \sim\ }g_!u_!.
\]
The ordinary functor in the following proposition rectifies these
isomorphisms.  Its Grothendieck-construction model and coherence are given in
Appendix~\ref{app:split-strictification}.

Write
\(\mathsf{Tw}^{\operatorname{str},\circ}_{\mathrm{SpWald}}\) for the
ordinary category of twistings of spectral Waldhausen categories having a
unique chosen zero object, with mapping spectra to or from that object
literally zero, and strict morphisms preserving it on the nose.

\begin{proposition}
\label{thm:split-perfect-strictification}
There is an ordinary functor
\begin{equation}\label{eq:split-perfect-functor}
 \operatorname{StrictPerf}_{\mathcal U}\colon
 \mathsf{Pres}_{\mathcal U}
 \longrightarrow
 \mathsf{Tw}^{\operatorname{str},\circ}_{\mathrm{SpWald}},
 \qquad
 x\longmapsto
 \bigl(\mathcal P_x^{\operatorname{st}},T_x^{\operatorname{st}}\bigr),
\end{equation}
together with a strict morphism of twistings
 \begin{equation}\label{eq:standard-to-split-envelope}
 j_x\colon
 {}_{T_x^\circ}\mathcal P_x^\circ/
   (\mathcal P_x^\circ)_{\operatorname{id}}
 \longrightarrow
 {}_{T_x^{\operatorname{st}}}\mathcal P_x^{\operatorname{st}}/
 (\mathcal P_x^{\operatorname{st}})_{\operatorname{id}}
\end{equation}
for every \(x\), such that:
\begin{enumerate}
\item for all composable morphisms,
\begin{equation}\label{eq:split-perfect-literal-relations}
 (vu)_*^{\operatorname{st}}
 =
 v_*^{\operatorname{st}}u_*^{\operatorname{st}},
 \qquad
 u_*^{\operatorname{st}}T_x^{\operatorname{st}}
 =
 T_y^{\operatorname{st}}u_*^{\operatorname{st}}
\end{equation}
as literal equalities of spectral functors;
\item the underlying spectral functor of \(j_x\) is fully faithful and
      essentially surjective by actual isomorphisms; in particular it is a
      DK equivalence, and a Morita equivalence, of presentations;
\item after stable localization, evaluation gives a natural equivalence of
      endofunctor diagrams
\begin{equation}\label{eq:split-evaluation-localized}
 \bigl(\mathcal P_x^{\operatorname{st}},
       T_x^{\operatorname{st}}\bigr)[W^{-1}]
 \simeq
 \bigl(\Perf(A),\mathbf Lf_!\bigr)
\end{equation}
as \(x\) varies;
\item for every \(r\geq1\), the \(r\)-fold cyclic twisted-endomorphism
      category associated with
      \((\mathcal P_x^{\operatorname{st}},T_x^{\operatorname{st}})\)
      presents the cyclic laced category of
      Definition~\ref{def:cyclic-laced-category}, naturally on
      \(N\mathsf{Pres}_{\mathcal U}\):
\begin{equation}\label{eq:split-source-lace-natural}
 K\!\operatorname{End}^{(r)}
 \bigl({}_{T_x^{\operatorname{st}}}\mathcal P_x^{\operatorname{st}}/
       (\mathcal P_x^{\operatorname{st}})_{\operatorname{id}}\bigr)
 \simeq
 K\bigl(\mathcal L_r(\Perf(A),\mathbf Lf_!)\bigr);
\end{equation}
      for \(r=1\) this is a natural equivalence with
      \(K^{\operatorname{lace}}(\Perf(A),\cM_{\mathbf Lf_!})\).
\end{enumerate}
\end{proposition}

\begin{proof}
Apply the split Grothendieck construction and zero reduction of
Appendix~\ref{app:split-strictification} to the normal perfect-module
pseudofunctor of Proposition~\ref{prop:perfect-module-pseudofunctor}.
The compositor identities give~(1); evaluation and its section give~(2)
and~(3), and evaluation on cyclic strict sections gives~(4).  The appendix
verifies the mapping-spectrum formulas, strict zero preservation, and
coherence.
\end{proof}

The twisting
\((\mathcal P_x^{\operatorname{st}},T_x^{\operatorname{st}})\) now varies
strictly functorially with \(x\), so the CLMPZ naturality theorem applies.

\begin{lemma}
\label{lem:clmpz-strict-functoriality}
On
\(\mathsf{Tw}^{\operatorname{str},\circ}_{\mathrm{SpWald}}\), with strictly
commutative morphisms of twistings in the sense of
\cite[Definition~2.16]{CLMPZ}, the suspended and prolonged CLMPZ
source restriction system, target restriction system, and Dennis-trace
zigzag are natural.  After passage to the \(\infty\)-localization this
gives an arrow-valued functor
\begin{equation}\label{eq:clmpz-strict-arrow-functor}
 \mathcal T_{\mathrm{CLMPZ}}\colon
 N\mathsf{Tw}^{\mathrm{str},\circ}_{\mathrm{SpWald}}
 \longrightarrow
 \operatorname{Fun}\bigl(\Delta^1,\operatorname{RSys}^g_\infty\bigr),
 \qquad
 z\longmapsto
 \bigl[\mathbf K^{\mathrm{res}}(z)
       \xrightarrow{\operatorname{tr}^{\mathrm{CLMPZ}}(z)}
       \mathbf E(z)\bigr].
\end{equation}
Here the right-hand side is formed from \(Q_{\mathrm{red}}z\).  Since the
augmentation leaves the Waldhausen base unchanged, its source is literally
\(\mathbf K^{\mathrm{res}}(z)\).
\end{lemma}

\begin{proof}
Apply \(Q_{\mathrm{red}}\) to the strict diagram of input twistings before
forming the CLMPZ diagram.  By
Theorem~\ref{thm:reduced-bar-replacement}, this is an ordinary functor and
preserves every literal relation among the input endofunctors.  At each
multi-index, the Waldhausen
\(S_\bullet\)-construction is natural because a strict exact morphism
preserves cofibrations, weak equivalences, and the chosen zero object.
The coefficient cyclic-bar maps commute term by term because
\(uf=gu\) holds on the nose, so the induced maps preserve the cyclic order
and every graph-coefficient action.  On the cofibrant wedge summands supplied
by \(Q_{\mathrm{red}}\), the norm-diagonal squares commute by the naturality
of \cite[Proposition~2.19]{CLMPZ}.  Genuine suspension and prolongation are
functorial, the latter by \cite[Proposition~A.7]{CLMPZSW}, and
multisimplicial realization preserves these natural transformations and
their commutative squares.  The CLMPZ Dennis-trace zigzag therefore
defines a functor
\[
 N\mathsf{Tw}^{\mathrm{str},\circ}_{\mathrm{SpWald}}
 \longrightarrow\operatorname{Fun}(J,\operatorname{RSys}^g_\infty),
\]
where \(J\) indexes its terms and arrows.  The replacement satisfies the
pointwise-cofibrancy hypotheses of \cite[Theorem~7.8]{CLMPZ}, so the backwards
twisted-additivity arrows are genuine equivalences at every node.  The
naturality used in the proof of \cite[Theorem~8.8]{CLMPZ} makes them
equivalences in this functor category.
Inverting them and applying the functorial prolongation of
\cite[Proposition~A.7]{CLMPZSW} gives
\eqref{eq:clmpz-strict-arrow-functor}.
\end{proof}

\begin{definition}
\label{def:split-arrow-functor}
Put
\begin{equation}\label{eq:split-clmpz-arrow-functor}
 \mathcal T_{\mathrm{CLMPZ}}^{\mathrm{st}}
 :=\mathcal T_{\mathrm{CLMPZ}}
       \circ N(\operatorname{StrictPerf}_{\mathcal U})\colon
 N\mathsf{Pres}_{\mathcal U}
 \longrightarrow
 \operatorname{Fun}\bigl(\Delta^1,\operatorname{RSys}^g_\infty\bigr).
\end{equation}
For a strict presentation \(x=(A,f)\), write
\(\mathbf K^{\mathrm{res}}(x)\) for its source vertex: the genuine
restriction system obtained by genuinely suspending and prolonging the
naive system
\begin{equation}\label{eq:module-envelope-source-system}
 n\longmapsto
 K\!\operatorname{End}^{(n)}
 \bigl({}_{T_x^{\mathrm{st}}}\mathcal P_x^{\mathrm{st}}/
       (\mathcal P_x^{\mathrm{st}})_{\operatorname{id}}\bigr),
\end{equation}
where the \(C_n\)-action rotates the \(n\) objects and lacing arrows and
the restriction isomorphisms identify a fixed repeated cycle with the
shorter cycle.  Write \(\mathbf E(x)\) for its target vertex, the genuine restriction system
of the graph coefficient of
\((\mathcal P_x^{\operatorname{st}},T_x^{\operatorname{st}})\).  The standard
CLMPZ diagram is formed from the input replacement \(Q_{\mathrm{red}}\) of
Remark~\ref{rem:standing-replacement-convention}, which also determines all
of its \(w_\bullet S_\bullet\)-nodes.
\end{definition}

\begin{theorem}
\label{thm:arrow-valued-morita-descent}
Let \(u\colon(A,f)\to(B,g)\) be a morphism of strict presentations lying
in \(W_{\mathrm{Mor}}\).  Then:
\begin{enumerate}
\item[(i)] the induced source map
\(\mathbf K^{\mathrm{res}}(A,f)\to\mathbf K^{\mathrm{res}}(B,g)\)
is an equivalence of genuine restriction systems;
\item[(ii)] the induced target map
\(\mathbf E(A,f)\to\mathbf E(B,g)\)
is an equivalence of genuine restriction systems;
\item[(iii)] consequently
\(\mathcal T_{\mathrm{CLMPZ}}^{\mathrm{st}}(u)\) is an equivalence in
the arrow category, and
\(\mathcal T_{\mathrm{CLMPZ}}^{\mathrm{st}}\) factors essentially
uniquely through the Morita localization:
\begin{equation}\label{eq:localized-clmpz-arrow-functor}
 \overline{\mathcal T}_{\mathrm{CLMPZ}}^{\mathrm{st}}\colon
 N\mathsf{Pres}_{\mathcal U}[W_{\mathrm{Mor}}^{-1}]
 \longrightarrow
 \operatorname{Fun}\bigl(\Delta^1,\operatorname{RSys}^g_\infty\bigr).
\end{equation}
Under the presentation-localization equivalence
\eqref{eq:endofunctor-diagram-rigidification}, this defines, for
every \(\mathcal U\)-small idempotent-complete stable \(\infty\)-category
\(\cC\) with exact endofunctor \(F\), an intrinsic arrow of genuine
restriction systems
\begin{equation}\label{eq:intrinsic-restriction-system-arrow}
 \mathbf K^{\mathrm{cyc}}(\cC,F)
 \longrightarrow
 \mathbf E(\cC,F),
\end{equation}
natural in \((\cC,F)\) and independent of the chosen presentation.
\end{enumerate}
\end{theorem}

\subsection{Morita invariance of the target}
\label{subsec:target-morita}

We first formulate coefficient Morita invariance for stable
\(\infty\)-categories with bimodules; Appendix~
\ref{app:technical-coherence} compares this formulation with the
orthogonal-spectrum model used by CLMPZ.

For a small idempotent-complete stable \(\infty\)-category \(\cC\), let
\[
 \operatorname{Bimod}(\cC)
 :=\operatorname{Fun}^{\operatorname{ex}}
   (\cC^{\operatorname{op}}\otimes\cC,\Sp)
\]
be the \(\infty\)-category of categorical bimodules, exact separately in
both variables.  By \cite[Definition~2.2 and Proposition~2.3]{HNS}, extension
and restriction of bimodules along exact functors assemble these categories
into a bicartesian fibration over
\(\operatorname{Cat}^{\operatorname{perf}}_\infty\).  We write
\[
 \operatorname{Pair}^{\operatorname{perf}}_\infty
 :=\int_{\cC\in\operatorname{Cat}^{\operatorname{perf}}_\infty}
       \operatorname{Bimod}(\cC)
\]
for its total \(\infty\)-category; this is the category of laced categories
of \cite{HNS}.  Thus a morphism
\[
 (u,\eta)\colon(\cC,\cM)\longrightarrow(\cD,\mathcal N)
\]
consists of an exact functor \(u\colon\cC\to\cD\) and a map
\(\eta\colon\cM\to(u^{\operatorname{op}}\otimes u)^*\mathcal N\).

\begin{definition}
\label{def:coefficient-morita-equivalence}
A morphism \((u,\eta)\colon(\cC,\cM)\to(\cD,\mathcal N)\) is a
\emph{coefficient Morita equivalence} if \(u\) is an equivalence in the
Morita \(\infty\)-category of small stable categories and the adjoint map
\[
 (u^{\operatorname{op}}\otimes u)_!\cM\longrightarrow\mathcal N
\]
is an equivalence of \(\cD\)-bimodules.  For spectral-category models, all
Kan extensions in this definition are derived.
\end{definition}

For an admissible spectral pair \((A,M)\), let \(\cM_M\) be the categorical
bimodule obtained by extending \(M\) from representable modules and then
restricting to perfect modules.  We call a morphism of admissible spectral
pairs a coefficient Morita equivalence when its induced morphism
\[
 (\Perf(A),\cM_M)\longrightarrow(\Perf(B),\cM_N)
\]
satisfies Definition~\ref{def:coefficient-morita-equivalence}.

There is a natural comparison
\begin{equation}\label{eq:spectral-coefficient-coend-comparison}
 \THH(A;M)\simeq\THH(\Perf(A),\cM_M).
\end{equation}
Indeed, derived co-Yoneda identifies the restriction of \(\cM_M\) to
representables with \(M\).  On cofibrant spectral models, the Yoneda
embedding of Lemma~\ref{lem:many-object-perfect-model} is DK fully faithful
and its image generates the perfect modules under finite cofibers and
retracts.  Thus \cite[Theorem~4.15]{CLMPZ} identifies their coefficient
cyclic bars.  The latter bar realizes the coend defining
\(\THH(\Perf(A),\cM_M)\): both constructions preserve colimits in the
coefficient and agree on representable bimodules, as in
\cite[Definition~4.4, Example~4.5, and Remark~4.8]{HNS}.
Both the derived coefficient extension and the bar comparison are natural
in maps of spectral pairs.

\begin{lemma}
\label{rem:coefficient-morita-unfolded}
Let \((u,\eta)\colon(\cC,\cM)\to(\cD,\mathcal N)\) be a coefficient Morita
equivalence.  Transport along \(u\) is compatible with relative tensor
products.  In particular, for every \(s\geq1\) there is a natural
equivalence
\begin{equation}\label{eq:monoidal-base-change}
 (u^{\operatorname{op}}\otimes u)_!
 \bigl(\cM^{\circ s}\bigr)
 \simeq
 \mathcal N^{\circ s}.
\end{equation}
These equivalences are unital and coherently multiplicative in \(s\).
\end{lemma}

\begin{proof}
Under the \(\infty\)-categorical Eilenberg--Watts equivalence, bimodules are
colimit-preserving module functors and relative tensor product is their
composition; see
\cite[Theorems~4.8.4.1 and~4.8.4.6 and Remark~4.8.4.9]{LurieHA}.
Conjugation by the equivalence of module categories induced by \(u\)
therefore gives a monoidal equivalence of bimodule \(\infty\)-categories.
Formula \eqref{eq:monoidal-base-change} and all of its coherences are
the images of the associativity and unit constraints for composition.
\end{proof}

\begin{theorem}
\label{thm:coefficient-thh-morita-invariance}
Coefficient topological Hochschild homology defines a functor
\[
 \THH\colon
 \operatorname{Pair}^{\operatorname{perf}}_\infty\longrightarrow\Sp
\]
which sends coefficient Morita equivalences to equivalences.  More
precisely, if \((u,\eta)\) is a coefficient Morita equivalence, then the
canonical map
\begin{equation}\label{eq:coefficient-thh-morita-invariance}
 \THH(\cC,\cM)\longrightarrow\THH(\cD,\mathcal N)
\end{equation}
is an equivalence, naturally in commutative squares of coefficient Morita
equivalences.  For each fixed \(s\geq1\), applying \(\THH\) to the
base-change equivalence
\eqref{eq:monoidal-base-change} gives a commutative comparison square
for the \(s\)-fold coefficients.
\end{theorem}

\begin{proof}
Harpaz--Nikolaus--Saunier define coefficient topological Hochschild homology
on the \(\infty\)-category of laced categories by the coend
\begin{equation}\label{eq:hns-coefficient-thh-coend}
 \THH(\cC,\cM)
 =\int^{x\in\cC}\cM(x,x);
\end{equation}
see \cite[Definition~4.4]{HNS}.  In particular, this construction is a
functor on \(\operatorname{Pair}^{\operatorname{perf}}_\infty\).

Because \(\cC\) and \(\cD\) are idempotent-complete, a Morita equivalence
\(u\colon\cC\to\cD\) is an equivalence of stable \(\infty\)-categories.  In
the bicartesian fibration of bimodules, the morphism \((u,\eta)\) factors as
the cocartesian transport
\[
 (\cC,\cM)\longrightarrow
 \bigl(\cD,(u^{\operatorname{op}}\otimes u)_!\cM\bigr)
\]
followed by the adjoint coefficient map to \((\cD,\mathcal N)\).  Both
arrows are equivalences under the hypotheses of
Definition~\ref{def:coefficient-morita-equivalence}; hence
\((u,\eta)\) is an equivalence in
\(\operatorname{Pair}^{\operatorname{perf}}_\infty\).  The functor
\eqref{eq:hns-coefficient-thh-coend} therefore sends it to
\eqref{eq:coefficient-thh-morita-invariance}, naturally in
commutative squares.  Applying the same functor to the coherently
multiplicative transport equivalences of
Lemma~\ref{rem:coefficient-morita-unfolded} gives the comparison for
the \(s\)-fold coefficients.
\end{proof}

\begin{lemma}
\label{lem:coefficient-thh-presentation}
Let \((A,f)\) be an admissible strict spectral presentation of an exact
endofunctor pair \((\cC,F)\), and let \(m\geq1\).  There is a natural
equivalence
\[
 \THH\bigl(A;A(f^m-,-)\bigr)
 \simeq
 \THH(\cC,\cM_{F^m}).
\]
On \(\pi_0\), the cyclic-bar zero-simplex class of an arrow
\(f^ma\to a\) corresponds to the zero-simplex class of its laced image.
\end{lemma}

\begin{proof}
The enriched comparison \((f^m)_!y_A\cong y_Af^m\) of
Lemma~\ref{lem:many-object-perfect-model} identifies the derived extension
of \(A(f^m-,-)\) with \(\cM_{F^m}\).  Apply
\eqref{eq:spectral-coefficient-coend-comparison} and transport along
\(\Perf(A)\simeq\cC\) using
Theorem~\ref{thm:coefficient-thh-morita-invariance}.  In
simplicial degree zero the Yoneda comparison sends \(f^ma\to a\) to the
corresponding point of \(\Map(F^ma,a)\), proving the last assertion.
\end{proof}

The point-set construction needed to retain the genuine cyclic actions is
Proposition~\ref{prop:admissible-pair-restriction-system} in
Appendix~\ref{app:pair-replacement-details}: it makes
\(\{\THH^{(n)}(A;M)\}_{n\geq1}\) functorial in admissible pairs and
compatible with all restriction arrows.

\begin{theorem}
\label{thm:admissible-pair-morita-invariance}
A coefficient Morita equivalence
\((u,\eta)\colon(A,M)\to(B,N)\) of admissible spectral pairs induces an
equivalence
\[
 \{\THH^{(n)}(A;M)\}_{n\geq1}
 \xrightarrow{\ \simeq\ }
 \{\THH^{(n)}(B;N)\}_{n\geq1}
\]
of genuine restriction systems.
\end{theorem}

\begin{proof}
Fix \(n\geq1\), let \(C_d\leq C_n\), and write \(n=ds\).
Proposition~7.4 and the unwinding equivalence of Proposition~7.6 in
\cite{CLMPZ} identify the induced geometric \(C_d\)-fixed-point map, after
forgetting the residual action, with
\[
 \THH\bigl(A;M^{\odot_A s}\bigr)
 \longrightarrow
 \THH\bigl(B;N^{\odot_B s}\bigr).
\]
Lemma~\ref{lem:pointwise-cofibrant-graph-multibar} identifies these ordinary
multibars with the corresponding derived coefficient powers.
Under~\eqref{eq:spectral-coefficient-coend-comparison},
Lemma~\ref{rem:coefficient-morita-unfolded} identifies the target power with
the Morita transport of the source power.
Theorem~\ref{thm:coefficient-thh-morita-invariance} therefore makes this map an
equivalence.  This holds for every subgroup, hence at every genuine cyclic
level.  Naturality of the norm diagonal makes the level maps a morphism of
restriction systems, proving the result.
\end{proof}

For graph coefficients we use the global pointwise-cofibrant replacement
\[
 Q^{\mathrm{cat}}\colon\operatorname{SpCat}^{\mathcal O}
 \longrightarrow\operatorname{SpCat}^{\mathcal O}
\]
of \cite[Theorem~3.11]{CLMPZ}.  It is an ordinary functor with a natural
pointwise-equivalence \(Q^{\mathrm{cat}}\to\operatorname{id}\).  Applied to
the strict diagram \((A,f)\), it gives the strict diagram
\((Q^{\mathrm{cat}}A,Q^{\mathrm{cat}}f)\); its graph coefficient is rebuilt
from the replaced mapping spectra.

\begin{proposition}
\label{prop:graph-coefficient-functor}
There is a functor
\begin{equation}\label{eq:strict-graph-restriction-system-functor}
 \mathbf E_{\operatorname{gr}}^{\operatorname{adm}}\colon
 \mathsf{Pres}_{\mathcal U}
 \longrightarrow\operatorname{RSys}^{g}_{\infty,\mathcal W}
\end{equation}
obtained from the graph pair
\[
 \bigl(Q^{\mathrm{cat}}A,\,
 Q^{\mathrm{cat}}A(Q^{\mathrm{cat}}f-,-)\bigr).
\]
It sends every morphism of
\(W_{\mathrm{Mor}}\) to an equivalence.  If \((A,f)\) is already admissible,
its value is represented by
\[
 \bigl\{\THH^{(n)}(A;A(f-,-))\bigr\}_{n\geq1}.
\]
\end{proposition}

\begin{proof}
Functoriality follows from \cite[Theorem~3.11 and Example~8.7]{CLMPZ}.
The cited theorem also makes \(Q^{\mathrm{cat}}A\) and its graph coefficient
pointwise cofibrant, so the displayed graph pair is admissible in the sense
of \S\ref{subsec:twistings-graph-coefficients}.
Lemma~\ref{lem:pointwise-cofibrant-graph-multibar} identifies its
successive ordinary bar powers with the derived coefficient powers, while
\cite[Proposition~7.4 and Example~8.7]{CLMPZ} supplies the restriction-system
structure.

Let
\(u\colon(A,f)\to(B,g)\) lie in \(W_{\mathrm{Mor}}\).  By
Lemma~\ref{lem:orthogonal-morita-model}(1), its orthogonal realization,
again denoted \(u\), is a Morita equivalence.  Derived extension of scalars
\[
 \mathbf Lu_!\colon\operatorname{Mod}_A\xrightarrow{\ \simeq\ }
 \operatorname{Mod}_B
\]
is an equivalence, and the equality \(uf=gu\), which holds on the nose,
induces a natural equivalence
\[
 \mathbf Lu_!\,\mathbf Lf_!
 \simeq
 \mathbf Lg_!\,\mathbf Lu_!.
\]
Under the \(\infty\)-categorical Eilenberg--Watts equivalence of
\cite[Theorems~4.8.4.1 and~4.8.4.6]{LurieHA}, conjugation by
\(\mathbf Lu_!\) therefore identifies the graph bimodules:
\[
 (u^{\operatorname{op}}\otimes u)_!A(f-,-)
 \simeq
 B(g-,-).
\]
Thus the graph-pair morphism is a coefficient Morita equivalence.  The
natural map \(Q^{\mathrm{cat}}A\to A\) is also a coefficient Morita
equivalence on graph pairs, naturally in \((A,f)\).  The natural square formed
by these augmentations, \(Q^{\mathrm{cat}}u\), and \(u\), together with
two-out-of-three in
\(\operatorname{Pair}^{\operatorname{perf}}_\infty\), shows that the graph
pair map induced by \(Q^{\mathrm{cat}}u\) is a coefficient Morita
equivalence.  Applying
Theorem~\ref{thm:admissible-pair-morita-invariance} to this admissible
coefficient Morita equivalence proves the invariance assertion.  Applying
it to the augmentation when \((A,f)\) is admissible proves the final assertion.
\end{proof}

\begin{lemma}
\label{lem:graph-perfect-morita-comparison}
For every strict presentation \(x=(A,f)\), spectral Yoneda, zero reduction,
and the strict morphism \(j_x\) of
Proposition~\ref{thm:split-perfect-strictification} give a
pseudonatural zigzag of coefficient Morita equivalences
\[
 (A,A(f-,-))
 \longrightarrow
 (\mathcal P_A,\mathcal P_A(f_!-,-))
 \xleftarrow{\ i_{\mathcal P_A}\ }
 (\mathcal P_A^\circ,
  \mathcal P_A^\circ((f_!)^\circ-,-))
 \xrightarrow{\ j_x\ }
 (\mathcal P_x^{\operatorname{st}},
  \mathcal P_x^{\operatorname{st}}(T_x^{\operatorname{st}}-,-)).
\]
The first arrow is induced by the Yoneda embedding, the middle arrow by the
inclusion \(i_{\mathcal P_A}\colon\mathcal P_A^\circ\to\mathcal P_A\), and
the last by the strict morphism of twistings \(j_x\).  Its image in the
Morita localization is a natural equivalence between the first and last graph
pairs.
\end{lemma}

\begin{proof}
Spectral Yoneda is Morita-dense and carries the graph bimodule of \(f\) to
the bimodule representing extension of scalars by \(f\).  For a spectral
functor \(u\colon A\to B\), the enriched Yoneda isomorphisms
\[
 u_!A(-,a)\cong B(-,u(a))
\]
give the pseudonaturality constraint for this first arrow; the
extension-of-scalars compositors give its unit and composition coherence.
The middle arrow is a coefficient Morita equivalence by
Lemma~\ref{lem:zero-reduction}: the comparison
\[
 i_{\mathcal P_A}(f_!)^\circ
 \xRightarrow{\ \sim\ }
 f_!i_{\mathcal P_A}
\]
induces its map of graph coefficients.  The last arrow is a coefficient
Morita equivalence because \(j_x\) is spectrally fully faithful and
essentially surjective.  The pseudonaturality of the middle comparison is the
pseudonaturality of the inclusions in
Lemma~\ref{lem:zero-reduction}.  For a morphism of presentations
\(u\colon x\to y\), the identity of \(u_!P\) defines an enriched natural
isomorphism
\[
 \gamma_u\colon u_*^{\operatorname{st}}j_x
 \xRightarrow{\ \sim\ }j_y(u_!)^\circ;
\]
at a zero image, take the identity of \(0_y\).  The mapping-spectrum formula
\eqref{eq:split-transition-map} gives naturality, normality and the
pseudofunctor pentagon give unit and composition coherence, and
\eqref{eq:split-beta-coherence} gives
compatibility with the distinguished endofunctors.  Thus the last comparison
is pseudonatural as well.  Compatibility of rectification with composition,
as stated in
Lemma~\ref{lem:pseudonatural-pair-morita-localization}, turns the
entire zigzag into the asserted natural equivalence after localization.
\end{proof}

Assertion~(ii) of
Theorem~\ref{thm:arrow-valued-morita-descent} now follows from
Proposition~\ref{prop:graph-coefficient-functor} and
Lemma~\ref{lem:graph-perfect-morita-comparison}.

\begin{theorem}
\label{thm:restriction-target-descent}
There is a functor
\begin{equation}\label{eq:derived-restriction-target-functor}
 \mathbf E_{\mathcal U}\colon
 \operatorname{Fun}
 \bigl(B\mathbb N,
       \operatorname{Cat}^{\mathrm{perf}}_{\infty,\mathcal U}\bigr)
 \longrightarrow
 \operatorname{RSys}^{g}_{\infty,\mathcal W}.
\end{equation}
If \((A,f)\) is an admissible strict spectral presentation of \((\cC,F)\),
then \(\mathbf E(\cC,F)\) is represented by
\begin{equation}\label{eq:derived-restriction-target-model}
 \mathbf E(A,f)=\{\THH^{(n)}(A;A(f-,-))\}_{n\geq1}.
\end{equation}
For every \(m\geq1\), there is a natural equivalence
\begin{equation}\label{eq:derived-restriction-target-level}
 V_m\mathbf E_m(\cC,F)
 \simeq
 \THH(\cC,\cM_F^{\circ m})
 \simeq
 \THH(\cC,\cM_{F^m}).
\end{equation}
The functor \(\mathbf E\) and the equivalences
\eqref{eq:derived-restriction-target-level} are compatible with
enlargement of universes.
\end{theorem}

\begin{proof}
The functor
\eqref{eq:strict-graph-restriction-system-functor} inverts
\(W_{\mathrm{Mor}}\).  The universal property of localization and
\eqref{eq:endofunctor-diagram-rigidification} give
\eqref{eq:derived-restriction-target-functor}.  For an admissible
presentation, the augmentation of \(Q^{\mathrm{cat}}\) is a coefficient
Morita equivalence, proving
\eqref{eq:derived-restriction-target-model}.  At level \(m\), the
unwinding equivalence of \cite[Proposition~7.6]{CLMPZ} gives the spectrum
\(\THH(A;A(f-,-)^{\odot m})\).  Lemma~
\ref{lem:pointwise-cofibrant-graph-multibar} replaces the ordinary
power by the derived one, and enriched co-Yoneda identifies its coefficient
with \(A(f^m-,-)\).  Lemma~
\ref{lem:coefficient-thh-presentation} then gives
\eqref{eq:derived-restriction-target-level}.  The replacement and
all comparisons are defined in a common larger universe, which gives the
last assertion.
\end{proof}

\subsection{Morita invariance of the source}
\label{subsec:source-morita}

The source of the CLMPZ trace is, level by level, the \(K\)-theory of the
cyclic laced categories of
Definition~\ref{def:cyclic-laced-category}.

\begin{lemma}
\label{lem:module-envelope-source-morita}
Let \(u\colon(A,f)\to(B,g)\) be a strict morphism lying in
\(W_{\mathrm{Mor}}\).  Then the exact functor
\(u_*^{\mathrm{st}}\) induces an equivalence
\[
 \mathbf K^{\mathrm{res}}(A,f)
 \xrightarrow{\ \simeq\ }
 \mathbf K^{\mathrm{res}}(B,g)
\]
of genuine restriction systems.
\end{lemma}

\begin{proof}
Derived extension of scalars is an exact equivalence
\(\Perf(A)\simeq\Perf(B)\), intertwined with the endofunctors by the
canonical isomorphism \(\beta_u\colon u_!f_!\cong g_!u_!\); it is
therefore an equivalence of endofunctor pairs
\[
 \bigl(\Perf(A),\mathbf Lf_!\bigr)
 \xrightarrow{\ \simeq\ }
 \bigl(\Perf(B),\mathbf Lg_!\bigr).
\]
Cyclic laced categories are functorial in equivalences of endofunctor
pairs, so for every \(r\geq1\) there is an induced exact equivalence
\(\mathcal L_r(\Perf(A),\mathbf Lf_!)\simeq
  \mathcal L_r(\Perf(B),\mathbf Lg_!)\).
By the cyclic source comparison
\eqref{eq:split-source-lace-natural}, the map induced by
\(u^{\mathrm{st}}_*\) on the \(r\)-th naive level of
\eqref{eq:module-envelope-source-system} is identified with
\(K\) of this equivalence, hence is an equivalence of spectra for every
\(r\).

These levelwise \(K\)-theory spectra form the particular naive restriction
system of \cite[Example~8.6]{CLMPZ}: the structure maps on categorical fixed
points are the canonical isomorphisms identifying a fixed \(r\)-cycle with
its repeated shorter cycle.  The construction is natural in strict morphisms
of twistings.  The suspension of this naive fixed-cycle system and its
passage to the prolonged genuine system are described in
\cite[Example~8.5, Theorem~8.8, Examples~A.3--A.4, and
Proposition~A.1]{CLMPZ}.  The fixed-cycle description of the CLMPZ source
shows that these levelwise equivalences commute with the restriction maps.
The derived prolongation of \cite[Proposition~A.7]{CLMPZSW} preserves them, so
they assemble to an equivalence of genuine restriction systems.
\end{proof}

\subsection{Descent of the trace arrow}
\label{subsec:arrow-descent}

\begin{proof}[Proof of Theorem~\ref{thm:arrow-valued-morita-descent}]
Assertion~(i) is Lemma~\ref{lem:module-envelope-source-morita}.
Assertion~(ii) is
Proposition~\ref{prop:graph-coefficient-functor}, transported along
the coefficient-Morita zigzag of
Lemma~\ref{lem:graph-perfect-morita-comparison}.

By~(i) and~(ii),
\(\mathcal T_{\mathrm{CLMPZ}}^{\mathrm{st}}\) sends
\(W_{\mathrm{Mor}}\) to equivalences, since equivalences in the arrow
\(\infty\)-category are detected at the two vertices.  The universal property
of localization gives
\eqref{eq:localized-clmpz-arrow-functor}, and
\eqref{eq:endofunctor-diagram-rigidification} identifies its domain
with the intrinsic endofunctor category.  The two vertices of the descended
functor are, by definition, \(\mathbf K^{\mathrm{cyc}}\) and \(\mathbf E\).
On spectral presentations, the pseudonatural coefficient-Morita zigzag of
Lemma~\ref{lem:graph-perfect-morita-comparison} compares the latter
with the functor of Theorem~\ref{thm:restriction-target-descent}.
Applying \(Q^{\mathrm{cat}}\) gives a zigzag of admissible pairs whose
augmentation is a natural coefficient Morita equivalence.  Lemma~
\ref{lem:pseudonatural-pair-morita-localization} rectifies this
zigzag and identifies the two descents.
\end{proof}

\subsection{The intrinsic trace and its ghosts}
\label{subsec:intrinsic-trace}

We now extract from
\eqref{eq:intrinsic-restriction-system-arrow} the trace
transformation of Theorem~A and compute its images under the ghost maps.

\begin{lemma}
\label{lem:module-envelope-truncated-source}
For every nonempty truncation set \(S\), there is a natural transformation
\begin{equation}\label{eq:intrinsic-source-extraction}
 \kappa_S\colon
 K^{\mathrm{lace}}(\cC,\cM_F)
 \longrightarrow
 R_S\bigl(\mathbf K^{\mathrm{cyc}}(\cC,F)\bigr),
\end{equation}
compatible with inclusions of truncation sets.  For a strict presentation
\(x=(A,f)\), it is represented by the composite
\begin{equation}\label{eq:module-envelope-source-to-derived-tr}
 K\!\operatorname{End}
 \bigl({}_{T_x^{\mathrm{st}}}\mathcal P_x^{\mathrm{st}}/
       (\mathcal P_x^{\mathrm{st}})_{\operatorname{id}}\bigr)
 \xrightarrow{\ \simeq\ }
 R_S^{\mathrm{un}}\bigl(\mathbf K^{\mathrm{res}}(A,f)\bigr)
 \longrightarrow
 R_S\bigl(\mathbf K^{\mathrm{res}}(A,f)\bigr).
\end{equation}
\end{lemma}

\begin{proof}
Every nonempty truncation set contains \(1\).  For the naive source system of
\cite[Example~8.6]{CLMPZ}, the repeated-cycle structure maps identify each
term \(\mathbf K_n^{\mathrm{res}}(A,f)^{C_n}\) with the first level.
Consequently the underived divisibility diagram over \(\mathcal I_S\) is
canonically constant.
This is the truncation-set version of \cite[Lemma~8.12]{CLMPZ}, and it gives
the first equivalence in
\eqref{eq:module-envelope-source-to-derived-tr}.  The second arrow is the
canonical CLMPZ map \(R_S^{\mathrm{un}}\to R_S\), induced by fibrant
replacement and passage to the limit in
\eqref{eq:derived-truncation-limit}.  Finally,
Proposition~\ref{thm:split-perfect-strictification}(4) identifies the first
level naturally with \(K^{\mathrm{lace}}(\cC,\cM_F)\), and
Theorem~\ref{thm:arrow-valued-morita-descent} descends the
construction to the intrinsic endofunctor category.
\end{proof}

To identify the \(r\)-th ghost, we first fix the following nonequivariant
unwinding equivalence.

\begin{definition}
\label{def:marked-unwinding}
Let \((A,f)\) be a pointwise-cofibrant strict presentation and put
\[
 X_f(a,b)=A(fa,b).
\]
After forgetting the \(C_r\)-action, the diagonal model of
\(\THH^{(r)}(A;X_f)\) has \(r\) cyclically ordered coefficient slots.  Choosing
slot \(0\) as the initial slot gives an equivalence
\[
 \omega_r^f\colon
 V_r\THH^{(r)}(A;X_f)
 \xrightarrow{\ \simeq\ }
 \THH(A;X_f^{\odot_A r}).
\]
Degreewise, it preserves the order of all mapping and coefficient factors.
This is a nonequivariant equivalence; another choice of initial slot differs
by cyclic rotation.
\end{definition}

For a nonempty truncation set \(S\) and \(r\in S\), let
\begin{equation}\label{eq:intrinsic-raw-truncated-ghost}
 \widetilde g^{F}_{r,S}\colon
 R_S\bigl(\mathbf E(\cC,F)\bigr)
 \longrightarrow
 V_r\mathbf E_r(\cC,F)
\end{equation}
be projection to \(\mathbf E_r(\cC,F)^{C_r}\), followed by the canonical map
\(\mathbf E_r(\cC,F)^{C_r}\to V_r\mathbf E_r(\cC,F)\).  Denote by
\(\omega_r^F:V_r\mathbf E_r(\cC,F)\simeq
\THH(\cC,\cM_{F^r})\) the equivalence above followed by enriched co-Yoneda
and the presentation comparison
\eqref{eq:derived-restriction-target-level}, and put
\begin{equation}\label{eq:intrinsic-truncated-ghost}
 \overline g^{F}_{r,S}:=
 \omega_r^F\widetilde g^F_{r,S}.
\end{equation}
The superscript \(F\) will be suppressed when \(F\) is clear from context.

\begin{theorem}
\label{thm:integral-module-envelope-trace-descent}
Fix a universe bound \(\mathcal U\).  For every nonempty
truncation set \(S\), there is a natural transformation
\begin{equation}\label{eq:intrinsic-integral-restriction-trace}
 \operatorname{tr}^{\mathrm{res,int}}_S\colon
 K^{\mathrm{lace}}(\cC,\cM_F)
 \longrightarrow
 R_S\bigl(\mathbf E(\cC,F)\bigr)
\end{equation}
on the full subcategory of endofunctor diagrams admitting an
\(\mathcal U\)-small spectral presentation.  It is the composite
\begin{equation}\label{eq:intrinsic-trace-as-composite}
 K^{\mathrm{lace}}(\cC,\cM_F)
 \xrightarrow{\ \kappa_S\ }
 R_S\bigl(\mathbf K^{\mathrm{cyc}}(\cC,F)\bigr)
 \xrightarrow{\ R_S(\overline{\operatorname{tr}}^{\mathrm{CLMPZ}})\ }
 R_S\bigl(\mathbf E(\cC,F)\bigr),
\end{equation}
where the second arrow is obtained from
\eqref{eq:intrinsic-restriction-system-arrow}.  These
transformations are compatible with inclusions of truncation sets and with
enlargement of universes.

For a strict presentation \(x=(A,f)\), this intrinsic arrow is represented
by the CLMPZ trace for
\((\mathcal P_x^{\operatorname{st}},T_x^{\operatorname{st}})\).  The strict
map \(j_x\) identifies it, through equivalences at both endpoints, with the
trace on \((\mathcal P_A^\circ,(f_!)^\circ)\).  The Morita equivalence
\(\mathcal P_A^\circ\to\mathcal P_A\) then recovers the standard
perfect-module presentation after localization.
\end{theorem}

\begin{proof}
Apply \(R_S\) to the intrinsic restriction-system arrow
\eqref{eq:intrinsic-restriction-system-arrow} and precompose with the map \(\kappa_S\) of
Lemma~\ref{lem:module-envelope-truncated-source}.  This gives
\eqref{eq:intrinsic-trace-as-composite}.  Naturality in \(S\)
follows by restriction of the divisibility diagram, and universe
independence follows by forming the presentation models in a common larger
universe.  The comparison with the standard perfect-module model is
induced by the strict morphism of twistings \(j_x\) and
Lemma~\ref{lem:clmpz-strict-functoriality}.
\end{proof}

To compare the traces for \(F\) and \(F^m\), we use the power-coherent
strictification constructed in Appendix~\ref{app:split-strictification}.

\begin{proposition}
\label{lem:power-coherent-perfect-envelope}
There is an ordinary functor
\(x\mapsto(\widetilde{\mathcal P}_x,S_x)\) from strict presentations to
strict twistings.  Its underlying spectral Waldhausen category has the same
objects and mapping spectra as \(\mathcal P_x^{\operatorname{st}}\), and it
has the following properties.
\begin{enumerate}
\item For every \(m\geq1\), there is a strict exact morphism of twistings
\begin{equation}\label{eq:power-comparison-map}
 q_{m,x}\colon
 \bigl(\widetilde{\mathcal P}_x,S_x^m\bigr)
 \longrightarrow
 \bigl(\widetilde{\mathcal P}_{x^{[m]}},S_{x^{[m]}}\bigr),
 \qquad x^{[m]}=(A,f^m),
\end{equation}
which is fully faithful and essentially surjective by actual isomorphisms
before applying \(Q_{\mathrm{red}}\).  It is natural in \(x\) and satisfies
\begin{equation}\label{eq:power-comparison-coherence}
 q_{m,y}\widetilde u_*=
 \widetilde u_*q_{m,x},
 \qquad
 q_{1,x}=\operatorname{id},
 \qquad
 q_{n,x^{[m]}}q_{m,x}=q_{mn,x}
\end{equation}
literally.

\item For every \(m\geq1\), there is a natural span of strict exact morphisms
\begin{equation}\label{eq:power-envelope-bridge}
 \bigl(\mathcal P_x^{\operatorname{st}},
       (T_x^{\operatorname{st}})^m\bigr)
 \longrightarrow
 \bigl(\mathcal B_x^{(m)},U_x^{(m)}\bigr)
 \longleftarrow
 \bigl(\widetilde{\mathcal P}_x,S_x^m\bigr).
\end{equation}
Its two legs are DK equivalences before replacement.  After applying
\(Q_{\mathrm{red}}\), they and the maps \(q_{m,x}\) induce equivalences on
the source and target restriction systems at every cyclic level, and hence
equivalences of the corresponding CLMPZ trace arrows.

\item The induced equivalence from the one-fold trace of
\((\mathcal P_x^{\operatorname{st}},(T_x^{\operatorname{st}})^m)\) to the
CLMPZ trace associated with \(x^{[m]}\) is natural in \(x\).  Under the
identification of the source with laced \(K\)-theory, ordered \(m\)-fold
composition of a lacing is the intrinsic iteration functor \(P_m\).  These
equivalences are unital and compatible with multiplication of the indices:
restricting the equivalence for \(m\) along \(\mu_n\) and then applying the
equivalence for \(n\) gives the equivalence for \(mn\).
\end{enumerate}
\end{proposition}

\begin{proof}
The construction of Appendix~\ref{app:split-strictification}
relates the two strictifications by the two-copy category
\(\mathcal B_x^{(m)}\) and provides the maps \(q_{m,x}\).
Proposition~\ref{prop:replacement-comparison-diagrams} shows that this
comparison remains an equivalence after applying \(Q_{\mathrm{red}}\).
Ordered composition of strict sections is restriction along \(\mu_m\) by
Lemma~\ref{lem:strict-sections-iteration}, and its compatibility with
the unit and repeated iteration is
Lemma~\ref{lem:powered-bridge-multiplicative-coherence}.
\end{proof}

At the \(m\)-th cyclic level, cyclic composition of the lacing arrows defines
a map from the \(m\)-fold source to the source of the one-fold trace for
\(f^m\).  The corresponding
identification on the target is the unwinding equivalence of
Definition~\ref{def:marked-unwinding}.  The following theorem makes
this comparison precise; its multisimplicial compatibility is proved in
Appendix~\ref{app:power-comparison-details}.

\begin{theorem}
\label{lem:clmpz-graph-power-comparison}
Let \(A\) be a pointwise-cofibrant spectral Waldhausen category and let
\[
 f\colon A\longrightarrow A
\]
be a strict exact spectral endofunctor.  There is an exact functor
\[
 \operatorname{Comp}_m\colon
 \operatorname{End}^{(m)}
 \bigl({}_fA/A_{\operatorname{id}}\bigr)
 \longrightarrow
 \operatorname{End}
 \bigl({}_{f^m}A/A_{\operatorname{id}}\bigr)
\]
sending
\[
 fa_i\xrightarrow{\alpha_i}a_{i+1},
 \qquad i\in\mathbb Z/m,
\]
to
\[
 \left(
 a_0,\,
 \alpha_{m-1}f(\alpha_{m-2})\cdots f^{m-1}(\alpha_0)
 \colon f^ma_0\to a_0
 \right).
\]

Let
\begin{equation}\label{eq:clmpz-marked-unwinding-power}
 \overline\omega_m^f\colon
 V_m\THH^{(m)}(A;X_f)
 \xrightarrow{\ \simeq\ }
 \THH(A;X_f^{\odot m})
 \xrightarrow{\ \simeq\ }
 \THH(A;X_{f^m})
\end{equation}
be the equivalence \(\omega_m^f\) followed by enriched co-Yoneda.  Then there
is a natural
homotopy-commutative square
\begin{equation}\label{eq:clmpz-graph-power-comparison}
\begin{tikzcd}[column sep=large]
 K\operatorname{End}^{(m)}({}_fA/A_{\operatorname{id}})
 \arrow[r,"V_m\operatorname{trc}^{(m)}"]
 \arrow[d,"K(\operatorname{Comp}_m)"'] &
 V_m\THH^{(m)}(A;X_f)
 \arrow[d,"\overline\omega_m^f"]\\
 K\operatorname{End}({}_{f^m}A/A_{\operatorname{id}})
 \arrow[r,"\operatorname{trc}^{(1)}"'] &
 \THH(A;X_{f^m}).
\end{tikzcd}
\end{equation}
It is natural for strict morphisms \(uf=gu\).  If \(\Delta_m\) denotes CLMPZ
duplication, then
\begin{equation}\label{eq:comp-duplication-iteration}
 \operatorname{Comp}_m\Delta_m(a,\alpha)
 =
 \bigl(a,\alpha f(\alpha)\cdots f^{m-1}(\alpha)\bigr).
\end{equation}
\end{theorem}

\begin{proof}
Ordered composition, retaining the component at the chosen initial vertex on
morphisms, is exact because the cyclic lacing relations are preserved and the
Waldhausen structure is coordinatewise.  Applied at every iterated
\(S\)-construction, it gives the source comparison.  On the target,
\(\overline\omega_m^f\) gives the corresponding comparison by sending the
coefficient factors to the same ordered composite.
This comparison respects the backwards twisted-additivity arrow and hence
passes through realization, genuine suspension, and derived prolongation.
Duplication
repeats the same lacing arrow in every slot and therefore gives
\eqref{eq:comp-duplication-iteration}.  The verification for the
complete \(\Sigma_\Delta\)-diagram is carried out in
Appendix~\ref{app:power-comparison-details}.
\end{proof}

\begin{proposition}
\label{prop:powered-trace-ghost-square}
For an exact endofunctor \(G\), let
\[
 \tau_G^{\mathrm{CL},1}:=
 \operatorname{tr}^{\mathrm{res,int}}_{\{1\}}\colon
 K^{\mathrm{lace}}(\cC,\cM_G)
 \longrightarrow R_{\{1\}}\bigl(\mathbf E(\cC,G)\bigr)
 \simeq\THH(\cC,\cM_G).
\]
For every \(m\geq1\),
Proposition~\ref{lem:power-coherent-perfect-envelope} gives a natural
equivalence between the one-fold CLMPZ trace formed from
\((\mathcal P_x^{\mathrm{st}},(T_x^{\mathrm{st}})^m)\) and
\(\tau_{F^m}^{\mathrm{CL},1}\).  We denote the resulting transformation by
\begin{equation}\label{eq:intrinsic-powered-clmpz-trace}
 \tau_{F;m}^{\mathrm{CL}}\colon
 K^{\mathrm{lace}}(\cC,\cM_{F^m})
 \longrightarrow\THH(\cC,\cM_{F^m}),
\end{equation}
and thus have a canonical natural equivalence
\begin{equation}\label{eq:rooted-separate-trace-identification}
 \tau_{F;m}^{\mathrm{CL}}\simeq\tau_{F^m}^{\mathrm{CL},1}.
\end{equation}
For every \(m\in S\), there is a natural homotopy-commutative square
\begin{equation}\label{eq:integral-intrinsic-spectrum-ghost}
\begin{tikzcd}[column sep=large]
 K^{\mathrm{lace}}(\cC,\cM_F)
 \arrow[r,"\operatorname{tr}^{\mathrm{res,int}}_S"]
 \arrow[d,"K(P_m)"'] &
 R_S\bigl(\mathbf E(\cC,F)\bigr)
 \arrow[d,"\overline g_{m,S}"]\\
 K^{\mathrm{lace}}(\cC,\cM_{F^m})
 \arrow[r,"\tau_{F;m}^{\mathrm{CL}}"'] &
 \THH(\cC,\cM_{F^m}).
\end{tikzcd}
\end{equation}
\end{proposition}

\begin{proof}
Apply \(Q_{\mathrm{red}}\) to the finite strict diagram of input twistings, as
in Remark~\ref{rem:standing-replacement-convention}.  Proposition~8.20 of
\cite{CLMPZ} identifies the composite of the CLMPZ trace with the \(m\)-th
ghost as duplication followed by the \(m\)-fold Dennis trace, naturally in
strict morphisms of twistings.  Theorem~
\ref{lem:clmpz-graph-power-comparison} turns this into the one-fold
trace of \((\mathcal P_x^{\mathrm{st}},(T_x^{\mathrm{st}})^m)\), and
Proposition~\ref{lem:power-coherent-perfect-envelope} compares that
twisting with the strictified perfect-module presentation for \(x^{[m]}\).

Naturality of the CLMPZ ghost comparison gives a diagram of trace arrows,
natural in \(x\), whose comparison maps are induced by \(q_{m,x}\) and the
inclusions into \(\mathcal B_x^{(m)}\).  These maps are equivalences by
Proposition~\ref{prop:replacement-comparison-diagrams}.  Inverting them leaves
a natural square and hence a functor
\[
 N\mathsf{Pres}_{\mathcal U}
 \longrightarrow
 \operatorname{Fun}(\Delta^1\times\Delta^1,\Sp).
\]
The lower horizontal arrow is the standard one-fold trace for \(x^{[m]}\).
Theorem~\ref{thm:integral-module-envelope-trace-descent} identifies its
descent with the intrinsic trace, giving
\eqref{eq:rooted-separate-trace-identification}.

Under the equivalence with laced \(K\)-theory from
Proposition~\ref{thm:split-perfect-strictification}(4),
Lemma~\ref{lem:strict-sections-iteration} identifies
\(\operatorname{Comp}_m\Delta_m\) with the intrinsic functor
\(P_m\).  On the target, the unwinding equivalence for
the chosen initial coefficient slot, enriched co-Yoneda, and
Lemma~\ref{lem:coefficient-thh-presentation} identify the right
vertical map with \(\overline g_{m,S}\).  The four vertex functors of this
square invert \(W_{\mathrm{Mor}}\) by the source and target descent theorems,
so the square descends through Morita localization and gives
\eqref{eq:integral-intrinsic-spectrum-ghost}.
\end{proof}

\begin{proposition}
\label{prop:laced-trace-on-objects}
On \(\pi_0\), the laced trace sends a laced object
\[
 (x,\alpha),
 \qquad
 \alpha\in\Omega^\infty\cM(x,x),
\]
to the class represented by \(\alpha\) in simplicial degree zero of the
coefficient cyclic bar.
\end{proposition}

\begin{proof}
The HNS trace is the additive extension of the natural map from a laced object
to coefficient \(\THH\), and that map is the inclusion of simplicial degree
zero in the coend model; see
\cite[Section~4.4 and Theorem~4.45, with Theorem~3.11 and
Corollary~4.41 as inputs]{HNS}.
\end{proof}

\begin{corollary}
\label{cor:integral-intrinsic-ghost-formula}
On \(\pi_0\), the trace \(\tau_{F;m}^{\mathrm{CL}}\) agrees with the HNS
laced trace:
\begin{equation}\label{eq:powered-clmpz-hns-pi0}
 \pi_0\tau_{F;m}^{\mathrm{CL}}
 =\pi_0\operatorname{tr}^{\mathrm{lace}}.
\end{equation}
Consequently, for every \(\xi\in K_0\Lace(\cC,\cM_F)\),
\begin{equation}\label{eq:integral-intrinsic-ghost-formula}
 \pi_0\overline g_{m,S}
 \bigl(\operatorname{tr}^{\mathrm{res,int}}_S(\xi)\bigr)
 =
 \operatorname{tr}^{\mathrm{lace}}
 \bigl(P_m(\xi)\bigr)
 \quad\text{in }\pi_0\THH(\cC,\cM_{F^m}).
\end{equation}
\end{corollary}

\begin{proof}
On an object of the twisted-endomorphism category in a chosen spectral
Waldhausen presentation, the one-fold CLMPZ trace
\cite[Lemma~7.13]{CLMPZ} and the HNS trace of
Proposition~\ref{prop:laced-trace-on-objects} are represented by the
same coefficient cyclic-bar zero-simplex.  Additivity and generation of
\(K_0\) by object classes prove
\eqref{eq:powered-clmpz-hns-pi0}.  Applying \(\pi_0\) to
\eqref{eq:integral-intrinsic-spectrum-ghost} gives
\eqref{eq:integral-intrinsic-ghost-formula}.
\end{proof}

\section{Iteration and Frobenius}
\label{sec:iteration-frobenius}

The comparison with iteration is governed by ordered block composition.  At
level \(n\), restriction from \(C_{mn}\) to \(C_n\) groups the \(mn\)
coefficient slots into \(n\) blocks of length \(m\), and co-Yoneda identifies
each block with the coefficient of \(F^m\).  These levelwise maps commute with
the CLMPZ restriction arrows and give the Frobenius square of Theorem~A.

We use the fixed-point conventions of
\S\ref{subsec:restriction-system-conventions}.  The norm-diagonal calculation
is made on cofibrant CLMPZ models, whose point-set geometric fixed points
compute the derived ones by the left-deformability discussion preceding
\cite[Proposition~2.19]{CLMPZ}.

\begin{lemma}
\label{lem:iteration-multiplicative-coherence}
For \(m,n\geq1\), restriction of right-lax cones supplies a canonical natural
equivalence
\[
 P_n^{F^m}P_m^F\simeq P_{mn}^F.
\]
These equivalences satisfy the unit and associativity coherences.
\end{lemma}

\begin{proof}
The monoid maps of Section~\ref{sec:graph-restriction} satisfy
\(\mu_n\mu_m=\mu_{mn}\).  Functoriality of precomposition gives the displayed
comparison, and associativity follows from associativity of composition in
\(B\mathbb N\), together with the associator for relative tensor products.
\end{proof}

The strict model comparisons in
\eqref{eq:power-comparison-coherence} represent these equivalences after
Morita localization.

\subsection{Frobenius and ordered block maps}

\begin{definition}
\label{def:restriction-system-block-shift}
For each \(q\geq1\), write \(C_q=\langle\rho_q\rangle\), where \(\rho_q\)
is the chosen standard generator.  For a genuine restriction system \(E\)
and \(m\geq1\), define its block
shift by
\begin{equation}\label{eq:restriction-system-block-shift}
 (\operatorname{sh}_mE)_n
 :=\operatorname{Res}^{C_{mn}}_{C_n}E_{mn}.
\end{equation}
Here \(C_n=\langle\rho_{mn}^{m}\rangle\subset
C_{mn}=\langle\rho_{mn}\rangle\).  For \(r\mid n\), the shifted
restriction arrow is the CLMPZ restriction arrow from level \(mn\) to level \(mr\),
restricted along \(\langle\rho_{mn}^{m}\rangle\) and
\(\langle\rho_{mr}^{m}\rangle\).  The inherited restriction maps make
\(\operatorname{sh}_mE\) a restriction system.  Concretely, its arrow from
level \(qr\) to level \(r\) is
\begin{equation}\label{eq:block-shift-restriction-arrow}
 \Phi^{C_q}\operatorname{Res}^{C_{mqr}}_{C_{qr}}E_{mqr}
 \cong
 \operatorname{Res}^{C_{mr}}_{C_r}\Phi^{C_q}E_{mqr}
 \longrightarrow
 \operatorname{Res}^{C_{mr}}_{C_r}E_{mr},
\end{equation}
with the identifications of cyclic groups induced by the chosen generators
suppressed; the last map is the
restriction arrow of \(E\).
\end{definition}

This block shift is the genuine analogue of the polygonic shift of
\cite[Example~2.28]{KMN}.

\begin{proposition}
\label{prop:graph-block-shift-frobenius}
For every small idempotent-complete stable \(\infty\)-category \(\cC\)
equipped with an exact endofunctor \(F\),
 grouping coefficient slots and applying
enriched co-Yoneda gives a natural equivalence of
genuine restriction systems
\begin{equation}\label{eq:intrinsic-graph-block-shift}
 b_m\colon
 \operatorname{sh}_m\mathbf E(\cC,F)
 \xrightarrow{\ \simeq\ }
 \mathbf E(\cC,F^m).
\end{equation}
Its level \(n\) representative for a
pointwise-cofibrant strict
presentation \((A,f)\) is the genuine \(C_n\)-equivalence
\begin{equation}\label{eq:graph-block-shift-level}
 b_{m,n}\colon
 \operatorname{Res}^{C_{mn}}_{C_n}\THH^{(mn)}(A;X_f)
 \xrightarrow{\ \simeq\ }
 \THH^{(n)}(A;X_{f^m}),
\end{equation}
obtained by grouping the \(mn\) coefficient slots into \(n\) consecutive
blocks of length \(m\) and applying enriched co-Yoneda
\(X_f^{\odot m}\simeq X_{f^m}\).  Using the equivalences \(\omega\) of
Definition~\ref{def:marked-unwinding}, write
\begin{equation}\label{eq:unwound-block-shift}
 \overline b_{m,n}^{F}:=
 \omega_n^{F^m}\,V_n(b_{m,n})\,(\omega_{mn}^{F})^{-1}:
 \THH(\cC,\cM_{F^{mn}})\xrightarrow{\ \simeq\ }
 \THH(\cC,\cM_{F^{mn}}).
\end{equation}
For every pair \(m,n\), the equivalences satisfy the natural binary
multiplicativity comparison
\begin{equation}\label{eq:block-shift-associativity}
 (b_n)^{(F^m)}\circ
 \operatorname{sh}_n\bigl((b_m)^{(F)}\bigr)
 \simeq (b_{mn})^{(F)}
\end{equation}
after the canonical identification
\(\operatorname{sh}_n\operatorname{sh}_m\simeq
\operatorname{sh}_{mn}\).
\end{proposition}

\begin{proof}
The comparison
\(\operatorname{sh}_m\THH(R,M)\simeq\THH(R,M^{\otimes_Rm})\) is due to
Krause--McCandless--Nikolaus in the polygonic setting
\cite[Example~2.28]{KMN}; here we construct its genuine
restriction-system refinement for graph coefficients.
For a pointwise-cofibrant strict presentation \((A,f)\),
Lemma~\ref{lem:admissible-graph-target-block} constructs
\eqref{eq:graph-block-shift-level} directly on the admissible graph
pair.  It proves at the same time that these level maps commute with all
restriction arrows.  Its geometric-fixed-point calculation is
\begin{equation}\label{eq:block-shift-geometric-fixed}
 \begin{aligned}
  \Phi^{C_d}\operatorname{Res}^{C_{mn}}_{C_n}
        \THH^{(mn)}(A;X_f)
  &\simeq
  \operatorname{Res}^{C_{mn/d}}_{C_{n/d}}
        \THH^{(mn/d)}(A;X_f) \\
  &\xrightarrow{\ b_{m,n/d}\ }
  \THH^{(n/d)}(A;X_{f^m})
  \simeq \Phi^{C_d}\THH^{(n)}(A;X_{f^m}).
 \end{aligned}
\end{equation}
Here \(d\mid n\), and every term retains the residual
\(C_n/C_d\cong C_{n/d}\)-action.  After forgetting this action, the marked
unwinding equivalences identify the middle map with enriched co-Yoneda at
level \(n/d\).  This also proves the asserted description of the level
representative.

For descent, apply the lemma to the functorial graph pair
\[
 \bigl(Q^{\mathrm{cat}}A,\,
 Q^{\mathrm{cat}}A(Q^{\mathrm{cat}}f-,-)\bigr)
\]
of Proposition~\ref{prop:graph-coefficient-functor}.  Since
\(Q^{\mathrm{cat}}\) is an ordinary functor, it preserves composition
strictly, so
\[
 Q^{\mathrm{cat}}(f^m)=(Q^{\mathrm{cat}}f)^m.
\]
Consequently, for each \(m\) the maps of the lemma assemble on
\(N\mathsf{Pres}_{\mathcal U}\) into a natural equivalence
\[
 \operatorname{sh}_m\mathbf E_{\operatorname{gr}}^{\operatorname{adm}}
 \xrightarrow{\ \simeq\ }
 \mathbf E_{\operatorname{gr}}^{\operatorname{adm}}\circ(-)^{[m]}.
\]
Both sides take \(W_{\mathrm{Mor}}\) to equivalences.
For \(\mathbf E_{\operatorname{gr}}^{\operatorname{adm}}\) this is
Proposition~\ref{prop:graph-coefficient-functor}; on the other side one
also uses that the power functor preserves \(W_{\mathrm{Mor}}\) and that
\(\operatorname{sh}_m\) preserves termwise genuine equivalences.  The natural
equivalence therefore descends.  Under
\eqref{eq:endofunctor-diagram-rigidification} and
Theorem~\ref{thm:restriction-target-descent}, its descent is
\eqref{eq:intrinsic-graph-block-shift}.

If \(A\) is already pointwise cofibrant, the augmentation
\(Q^{\mathrm{cat}}A\to A\) is a coefficient Morita equivalence of graph
pairs, and strict naturality in
Lemma~\ref{lem:admissible-graph-target-block} identifies the descended
map with \(b_{m,n}\).  The unit, binary associativity, and fourfold coherence
diagrams in that lemma descend as well, yielding
\eqref{eq:block-shift-associativity} and its coherence.
\end{proof}

\begin{proposition}
\label{prop:restriction-frobenius}
For a truncation set \(S\) and \(m\geq1\), put
\begin{equation}\label{eq:truncation-quotient}
 S/m:=\{n\geq1:mn\in S\}.
\end{equation}
This set is closed under positive divisors and is nonempty exactly when
\(m\in S\).  If \(m\in S\), then for each \(n\in S/m\), restriction from
\(C_{mn}\) to \(C_n\), followed by the block equivalence, gives
\begin{equation}\label{eq:restriction-frobenius-coordinate}
 R_S\bigl(\mathbf E(\cC,F)\bigr)\longrightarrow
 \mathbf E_{mn}(\cC,F)^{C_{mn}}\longrightarrow
 \bigl(\operatorname{Res}^{C_{mn}}_{C_n}
        \mathbf E_{mn}(\cC,F)\bigr)^{C_n}
 \xrightarrow{\ b_{m,n}^{C_n}\ }
 \mathbf E_n(\cC,F^m)^{C_n}.
\end{equation}
Taking the limit over \(\mathcal I_{S/m}\) defines a natural map
\begin{equation}\label{eq:intrinsic-restriction-frobenius}
 \operatorname{Fr}_{m,S}\colon
 R_S\bigl(\mathbf E(\cC,F)\bigr)
 \longrightarrow
 R_{S/m}\bigl(\mathbf E(\cC,F^m)\bigr).
\end{equation}
For \(m=1\), the block maps are the identity and
\(\operatorname{Fr}_{1,S}\simeq\operatorname{id}\).  They satisfy
\begin{equation}\label{eq:restriction-frobenius-coherence}
\operatorname{Fr}_{n,S/m}\operatorname{Fr}_{m,S}
 \simeq\operatorname{Fr}_{mn,S}
\end{equation}
whenever \(mn\in S\).  On underlying ghosts,
\begin{equation}\label{eq:restriction-frobenius-ghost}
 \overline g^{F^m}_{n,S/m}\operatorname{Fr}_{m,S}
 \simeq \overline b_{m,n}^{F}\,
 \overline g^{F}_{mn,S}.
\end{equation}
\end{proposition}

\begin{proof}
The middle arrow of
\eqref{eq:restriction-frobenius-coordinate} forgets from
\(C_{mn}\)-fixed points to \(C_n\)-fixed points.  The restriction-system
axioms and compatibility of \(b_m\) just proved make these arrows a cone on
the \(S/m\)-diagram.  The universal property of the limit gives
\eqref{eq:intrinsic-restriction-frobenius}.  Restricting a truncation
set restricts the same cone, so these maps are natural under inclusions of
truncation sets.  For \(m=1\), subgroup forgetting and one-fold block
composition are identities, which gives
\(\operatorname{Fr}_{1,S}\simeq\operatorname{id}\).  Associativity of
subgroup forgetting together with
\eqref{eq:block-shift-associativity} gives
\eqref{eq:restriction-frobenius-coherence}; projection to level \(n\),
followed by the two equivalences \(\omega\), gives
\eqref{eq:restriction-frobenius-ghost}.
\end{proof}

\subsection{The Frobenius theorem}

\begin{theorem}
\label{thm:restriction-trace-frobenius}
For every truncation set \(S\) and \(m\in S\), there is a natural
homotopy-commutative square
\begin{equation}\label{eq:restriction-trace-frobenius-square}
\begin{tikzcd}[column sep=large,row sep=large]
 K^{\mathrm{lace}}(\cC,\cM_F)
 \arrow[r,"\operatorname{tr}^{\mathrm{res,int}}_S"]
 \arrow[d,"K(P_m)"'] &
 R_S\bigl(\mathbf E(\cC,F)\bigr)
 \arrow[d,"\operatorname{Fr}_{m,S}"]\\
 K^{\mathrm{lace}}(\cC,\cM_{F^m})
 \arrow[r,"\operatorname{tr}^{\mathrm{res,int}}_{S/m}"'] &
 R_{S/m}\bigl(\mathbf E(\cC,F^m)\bigr).
\end{tikzcd}
\end{equation}
The homotopies are compatible with inclusions of truncation sets.  Whenever
\(mn\in S\), the square for \(mn\) agrees with the pasting of the squares for
\(m\) and \(n\) through
\eqref{eq:restriction-frobenius-coherence}.  For three positive
integers, the two iterated pastings agree; hence the trace squares form a
coherently multiplicative natural transformation under
\((F,S)\mapsto(F^m,S/m)\).
\end{theorem}

\begin{proof}
Work in \((\widetilde{\mathcal P}_x,S_x)\), after applying
\(Q_{\mathrm{red}}\) to the strict diagram of zero-reduced input twistings as
in Remark~\ref{rem:standing-replacement-convention}.
The construction in Appendix~\ref{app:block-comparison-details}, applied to
\(H=S_x\), gives, for each \(n\), the square of genuine \(C_n\)-spectra
\[
\begin{tikzcd}[column sep=large]
 \operatorname{Res}^{C_{mn}}_{C_n}\mathbf K_{mn}(S_x)
 \arrow[r,"\operatorname{Res}\operatorname{trc}^{(mn)}"]
 \arrow[d,"K(\operatorname{Comp}_{m\mid n})"'] &
 \operatorname{Res}^{C_{mn}}_{C_n}\mathbf E_{mn}(S_x)
 \arrow[d,"b_{m,n}"]\\
 \mathbf K_n(S_x^m)
 \arrow[r,"\operatorname{trc}^{(n)}"'] &
 \mathbf E_n(S_x^m).
\end{tikzcd}
\]
Here \(\mathbf K_r(S_x)\) and \(\mathbf E_r(S_x)\) denote the source and
target systems of the replaced twisting
\((\widetilde{\mathcal P}_x,S_x)\).  The same square is compatible with the
complete \(\Sigma_\Delta\)-diagram and all CLMPZ restriction arrows.  On the
source,
\[
 \operatorname{Comp}_{m\mid n}\Delta_{mn}=\Delta_nP_m
\]
as an identity of source functors.  The span through
\(\mathcal B_x^{(1)}\) compares the upper row with the strictified
perfect-module model for \(x\).  For the lower row, \(q_{m,x}\) and the span
through \(\mathcal B_{x^{[m]}}^{(1)}\) give the corresponding comparison with
the model for \(x^{[m]}\).  These comparison maps are strict morphisms of
zero-reduced twistings.  After applying
\(Q_{\mathrm{red}}\), Proposition~\ref{prop:replacement-comparison-diagrams}
makes them equivalences in
\(\operatorname{Fun}(\Delta^1,\operatorname{RSys}^g_\infty)\), naturally in
\(x\).
The two spans are the \(m=1\) case of
Lemma~\ref{lem:powered-bridge-multiplicative-coherence} and identify the two
rows of the block square naturally in \(x\).

After taking categorical fixed points levelwise and passing to the limit over
\(\mathcal I_{S/m}\), let
\[
 \pi_{\mathbf K}\colon
 R_S\bigl(\mathbf K^{\mathrm{cyc}}(\cC,F)\bigr)
 \longrightarrow
 R_{S/m}\bigl(\operatorname{sh}_m
                    \mathbf K^{\mathrm{cyc}}(\cC,F)\bigr),
 \qquad
 \pi_{\mathbf E}\colon
 R_S\bigl(\mathbf E(\cC,F)\bigr)
 \longrightarrow
 R_{S/m}\bigl(\operatorname{sh}_m\mathbf E(\cC,F)\bigr)
\]
be induced by projection to the levels \(mn\) and restriction from
\(C_{mn}\)-fixed points to \(C_n\)-fixed points.  In this diagram, the
comparison maps above are equivalences.  The vertex functors of the middle
block square invert
\(W_{\mathrm{Mor}}\), so that square descends to
\[
\begin{tikzcd}[column sep=large,row sep=large]
 R_{S/m}\bigl(\operatorname{sh}_m
                    \mathbf K^{\mathrm{cyc}}(\cC,F)\bigr)
 \arrow[r,"{R_{S/m}(\operatorname{sh}_m
                    \overline{\operatorname{tr}}^{\mathrm{CLMPZ}})}"]
 \arrow[d,"R_{S/m}(c_m)"'] &
 R_{S/m}\bigl(\operatorname{sh}_m\mathbf E(\cC,F)\bigr)
 \arrow[d,"R_{S/m}(b_m)"]\\
 R_{S/m}\bigl(\mathbf K^{\mathrm{cyc}}(\cC,F^m)\bigr)
 \arrow[r,"{R_{S/m}(
                    \overline{\operatorname{tr}}^{\mathrm{CLMPZ}})}"'] &
 R_{S/m}\bigl(\mathbf E(\cC,F^m)\bigr).
\end{tikzcd}
\]
Here \(c_m\) is induced by ordered block composition and \(b_m\) is the
graph block equivalence.

For the strict source model, \cite[Example~8.6]{CLMPZ} gives the natural
underived identification
\[
 R^{\mathrm{un}}_{S/m}\bigl(\operatorname{sh}_m
                    \mathbf K^{\mathrm{cyc}}(\cC,F)\bigr)
 \cong
 K\!\operatorname{End}^{(m)}
 \bigl({}_{T_x^{\operatorname{st}}}\mathcal P_x^{\operatorname{st}}/
       (\mathcal P_x^{\operatorname{st}})_{\operatorname{id}}\bigr).
\]
The duplication \(\Delta_m\), followed by the canonical comparison
\(R^{\mathrm{un}}_{S/m}\to R_{S/m}\), defines
\(\kappa^{\operatorname{sh}}_{m,S}\).  Naturality of the underived
fixed-cycle identification gives
\[
 \kappa^{\operatorname{sh}}_{m,S}
 =\pi_{\mathbf K}\circ\kappa_S.
\]
The identity
\(\operatorname{Comp}_{m\mid n}\Delta_{mn}=\Delta_nP_m\), together with
Lemma~\ref{lem:strict-sections-iteration}, identifies the left outer
square with \(K(P_m)\) and \(\kappa_{S/m}\).  On the target,
\[
 \operatorname{Fr}_{m,S}
 =R_{S/m}(b_m)\circ\pi_{\mathbf E}.
\]
Pasting the two outer squares with the localized middle square gives
\eqref{eq:restriction-trace-frobenius-square}.  An inclusion of
truncation sets restricts this diagram and its limiting cone, which proves
naturality in \(S\).

For multiplication of indices, the identity
\(q_{n,x^{[m]}}q_{m,x}=q_{mn,x}\), the block associativity proved in
Appendix~\ref{app:block-comparison-details}, and
Lemma~\ref{lem:powered-bridge-multiplicative-coherence} identify the
square for \(mn\) with the composite of the squares for \(m\) and \(n\).  In
the last comparison, the restriction of \(\chi_{m,x}\) along \(\mu_n\),
followed by \(\chi_{n,x^{[m]}}\), agrees with \(\chi_{mn,x}\).  The one-fold
block map supplies the unit, and
Lemma~\ref{lem:iteration-multiplicative-coherence} identifies the source
pasting with \(P_{mn}^F\).  For three indices, the block maps reassociate the
same ordered relative tensor product, so the pentagon and the associativity
coherence of the \(\chi_{m,x}\) identify the two pastings.
\end{proof}

\begin{remark}
\label{rem:scope}
The ghost map need not be injective.  Indeed,
\(\pi_0\TR(H\mathbb F_p)\cong W(\mathbb F_p)\) has nonzero elements with all
integral ghosts zero; this follows from the Witt ghost formula
\cite[equation~(9.1) and Lemma~9.6]{CLMPZ}.
\end{remark}

\section{Twisted traces of graded algebras}
\label{sec:twisted-traces-graded}

We apply the intrinsic construction to grading automorphisms.  Let
\(R\) be a commutative ring, let \(\Lambda\) be a free abelian group of
finite rank, and let
\[
 A=\bigoplus_{\lambda\in\Lambda}A_\lambda
\]
be a unital \(\Lambda\)-graded \(R\)-algebra.  Write
\[
 Z=\operatorname{Hom}(\Lambda,\Gm)
\]
for the grading-parameter torus.  For a commutative \(R\)-algebra \(S\) and
\(z\in Z(S)\), put \(A_S=A\otimes_R S\) and define an \(S\)-linear map
\[
 \sigma_z\colon A_S\longrightarrow A_S,
 \qquad
 \sigma_z(a)=z(\lambda)a\quad(a\in A_\lambda).
\]

Then \(\sigma_z\) is an \(S\)-algebra automorphism with inverse
\(\sigma_{z^{-1}}\), and \(\sigma_{z^m}=\sigma_z^m\) for \(m\geq1\).

Define an exact endofunctor on perfect right modules by
\[
 F_z=(-)_{\sigma_z^{-1}}\colon\Perf(A_S)\longrightarrow\Perf(A_S),
 \qquad
 m*a=m\sigma_z^{-1}(a).
\]

\begin{lemma}
\label{lem:graded-graph-identification}
The functor \(F_z\) is an exact automorphism of \(\Perf(A_S)\).  Its graph
coefficient
\[
 \cM_{F_z}(x,y)=\Map(F_zx,y)
\]
has, on the regular generator, the bimodule
\[
 \cM_{F_z}(A_S,A_S)\cong {}_1(A_S)_{\sigma_z},
\]
whose underlying left module is free of rank one and whose right action is
through \(\sigma_z\).
\end{lemma}

\begin{proof}
Twisting by an algebra automorphism preserves perfect modules and exact
sequences, and twisting by \(\sigma_z\) is inverse to \(F_z\).  Evaluation at
\(1\) identifies \(\Map(F_zA_S,A_S)\) with \(A_S\), since
\(\phi(c)=\phi(1)\sigma_z(c)\).  Under this identification, postcomposition
by \(a\) and precomposition by \(b\) act by \(x\mapsto ax\) and
\(x\mapsto x\sigma_z(b)\), respectively.  Hence the graph coefficient on the
regular generator is \({}_1(A_S)_{\sigma_z}\).
\end{proof}

\begin{theorem}
\label{thm:intrinsic-graded-twisted-traces}
For every \(A\), \(S\), and \(z\in Z(S)\) as above, the pair
\((\Perf(A_S),F_z)\) has an intrinsic genuine restriction system and, for
every nonempty truncation set \(\mathcal S\), an intrinsic trace
\[
 \operatorname{tr}^{\mathrm{res,int}}_{\mathcal S}\colon
 K^{\mathrm{lace}}\bigl(\Perf(A_S),\cM_{F_z}\bigr)
 \longrightarrow
 R_{\mathcal S}\bigl(\mathbf E(\Perf(A_S),F_z)\bigr).
\]
For every \(m\in\mathcal S\), the target of the \(m\)-th ghost on
\(\pi_0\) is canonically the \(\sigma_{z^m}\)-twisted cocenter
\begin{equation}\label{eq:graded-twisted-cocenter}
 \pi_0\THH\bigl(\Perf(A_S),\cM_{F_z^m}\bigr)
 \cong
 A_S\big/\operatorname{span}_S
 \{ab-b\sigma_{z^m}(a):a,b\in A_S\}.
\end{equation}
Moreover, the intrinsic Frobenius implements the substitution
\(z\mapsto z^m\): it is a natural map
\[
 \operatorname{Fr}_{m,\mathcal S}\colon
 R_{\mathcal S}\bigl(\mathbf E(\Perf(A_S),F_z)\bigr)
 \longrightarrow
 R_{\mathcal S/m}\bigl(\mathbf E(\Perf(A_S),F_{z^m})\bigr)
\]
compatible with \(\sigma_{z^m}=\sigma_z^m\) and with iteration of laced
objects.
\end{theorem}

\begin{proof}
The existence of the intrinsic system and trace follows from
Theorem~\ref{thm:integral-module-envelope-trace-descent}, applied to the exact
automorphism of
Lemma~\ref{lem:graded-graph-identification}.  The graph calculation
and \(\sigma_{z^m}=\sigma_z^m\) identify
\(\cM_{F_z^m}\) on the regular generator with
\({}_1(A_S)_{\sigma_{z^m}}\).  Lemma~\ref{lem:coefficient-thh-presentation}
identifies \(\pi_0\) with the coefficient Hochschild group
\[
 HH_0\bigl(A_S;{}_1(A_S)_{\sigma_{z^m}}\bigr)
 =A_S/\operatorname{span}_S
   \{ab-b\sigma_{z^m}(a):a,b\in A_S\}.
\]
The ghost formula is
Corollary~\ref{cor:integral-intrinsic-ghost-formula}.  Finally,
Theorem~\ref{thm:restriction-trace-frobenius}, together with
\(F_z^m=F_{z^m}\), gives the asserted Frobenius substitution and its
compatibility with iteration.
\end{proof}

\subsection{The DKK coefficient on the open parameter torus}
\label{subsec:dkk-open-torus-coefficient}

Dinkins--Karpov--Krylov define a cyclic bimodule \(\mathcal B\) over the
partial compactification \(\operatorname{Spec}\mathbb C[R^-]\)
\cite[Section~1.6.4]{DKK}.  Let \(L=\mathcal O(\mathsf A)\) be the coordinate
ring of the open parameter torus, and put
\[
 \mathcal A^\circ
 =\mathcal A[\hbar^{-1}]\otimes_{\mathbb C}L,
 \qquad
 \mathcal B^\circ
 =\mathcal B[\hbar^{-1}]\otimes_{\mathbb C[R^-]}L.
\]
If \(\zeta\in\mathsf A(L)\) is the universal point, their defining relations
identify
\[
 \mathcal B^\circ\cong{}_1(\mathcal A^\circ)_{\sigma_\zeta}.
\]
Indeed, after localization the monomials \(z^\xi\) are invertible, and the
relations say \(ex=\sigma_\zeta(x)e\) for every homogeneous generator \(x\)
of \(\mathcal A^\circ\); see \cite[equation~(10)]{DKK}.
Applying Theorem~\ref{thm:intrinsic-graded-twisted-traces} to this graph
coefficient gives a genuine restriction system.  For every truncation set
\(\mathcal S\) and \(m\in\mathcal S\), the \(m\)-th ghost target is the
coefficient Hochschild object twisted by \(\sigma_{\zeta^m}\); at \(m=1\),
this is the open-torus localization of their coefficient Hochschild object.

\section{Higher ghosts in toric dynamics}
\label{sec:toric-lacings}

For compactly supported equivariant motives, the \(m\)-th ghost is a sum over
the toric strata fixed by the \(m\)-th iterate.  We prove this by constructing
a semilinear lacing and applying nil-invariant localization along the orbit
stratification.

\subsection{Semilinear compact support and lacing}
\label{sec:semilinear-compact-support}

Let \(B\) be a scalloped algebraic stack and let \(g\colon B\to B\) be a
morphism.  We use the six operations on \(B\) and on representable finite-type
\(B\)-stacks as in \cite{KhanRavi}.  Every exceptional pushforward below is
along a representable morphism of finite type.  Exceptional base change is
provided by \cite[Theorem~7.1(iii)]{KhanRavi}; the ordinary inverse image
may be taken along the arbitrary morphism \(g\).  For a representable
proper morphism \(a\), we also use the identification \(a_!\simeq a_*\) of
\cite[Theorem~7.1(ii)]{KhanRavi}.

If \(p\colon Y\to B\) is representable, separated, and of finite type, define
its compactly supported motive over \(B\) by
\[
  \Mcompact_B(Y)=p_!\one_Y\in\SH(B).
\]
Write
\[
  Y^{(g)}=Y\times_{B,g}B,
  \qquad p^{(g)}\colon Y^{(g)}\longrightarrow B.
\]
Exceptional base change \cite[Theorem~7.1(iii)]{KhanRavi} gives a canonical
equivalence
\begin{equation}\label{eq:base-change-compact-support}
  g^*\Mcompact_B(Y)
  \simeq (p^{(g)})_!\one_{Y^{(g)}}
  =\Mcompact_B(Y^{(g)}).
\end{equation}

Choose a small idempotent-complete stable full subcategory
\(\cC\subset\SH(B)\) that is closed under \(g^*\) and contains the compactly
supported motives under consideration.  In our applications it is the thick
closure of finitely many such motives and their \(g^*\)-iterates.  Set
\begin{equation}\label{eq:left-representable-bimodule}
  \cM_g(x,y)=\Map_{\SH(B)}(g^*x,y),
  \qquad x,y\in\cC.
\end{equation}
This is the left-representable lacing associated with the exact functor \(g^*\)
\cite[Remark~5.15]{HNS}.
Theorem~\ref{thm:integral-module-envelope-trace-descent} gives an intrinsic
restriction trace for \((\cC,g^*)\).  We abbreviate
\[
  \Lace_g(\cC)=\Lace(\cC,\cM_g).
\]
For this left-representable bimodule, an object of the laced category is a pair
\((x,\alpha)\) with
\[
  \alpha\colon g^*x\longrightarrow x.
\]

\begin{definition}
\label{def:semilinear-object}
Let \(a\colon Y\to Y^{(g)}\) be a representable proper \(B\)-morphism.  Its
compactly supported lacing is
\[
 \lambda_a\colon
 g^*\Mcompact_B(Y)
 \overset{\eqref{eq:base-change-compact-support}}{\simeq}
 \Mcompact_B(Y^{(g)})
 \xrightarrow{a_c^*}
 \Mcompact_B(Y),
\]
where \(a_c^*\) is obtained from the unit
\(
  \one_{Y^{(g)}}\to a_*a^*\one_{Y^{(g)}}=a_*\one_Y
\)
and the equivalence \(a_*\simeq a_!\) for representable proper \(a\) from
\cite[Theorem~7.1(ii)]{KhanRavi}.  We denote the resulting object by
\[
  \Xi_g(Y,a)=\bigl(\Mcompact_B(Y),\lambda_a\bigr)
  \in\Lace_g(\cC).
\]
\end{definition}

The construction is contravariantly functorial in proper maps: the composite of
two proper pullbacks agrees with the proper pullback of the composite.  This
follows from the units of the two adjunctions together with the compatibility
of the equivalence \(a_!\simeq a_*\) with composition of proper maps.

\subsection{Twisting an equivariant action}

For the toric applications from this point onward, fix a field \(k\), a
finite-rank lattice \(N\), its dual \(M=N^\vee\), and the split torus
\(T=T_N=\operatorname{Spec}k[M]\).  For a finite-type \(k\)-space \(Y\), set
\(\Mcompact_k(Y):=\Mcompact_{\Spec k}(Y)\); when \(Y\) carries a
\(T\)-action, set \(\Mcompact_T(Y):=\Mcompact_{BT}([Y/T])\).
A fan is a finite rational polyhedral fan
in \(N_{\mathbb R}\), and \(X_\Sigma\) denotes the associated toric
\(k\)-variety.  For \(\sigma\in\Sigma\), put
\[
 N_\sigma=N\cap\operatorname{span}_{\mathbb R}(\sigma),
 \qquad T_\sigma=T_{N_\sigma}\subset T,
 \qquad O_\sigma=T/T_\sigma.
\]
Thus \(T_\sigma\) is the scheme-theoretic stabilizer of the orbit
\(O_\sigma\), and
\(\dim O_\sigma=\operatorname{rk}N-\dim\sigma\).  Whenever fans or their
stabilizer multisets are compared, they lie in this same ambient lattice and
torus; ``embedded stabilizer'' means the actual closed subgroup
\(T_\sigma\hookrightarrow T\), including its position in \(T\).
Let \(\varphi\colon T\to T\) be an endomorphism.
For a \(T\)-space \(Y\), write \(\Res_\varphi Y\) for the same underlying space
equipped with the action
\[
  t\star y=\varphi(t)\cdot y.
\]
A \(\varphi\)-semilinear map \(f\colon Y\to Y\), meaning
\(
  f(t\cdot y)=\varphi(t)\cdot f(y),
\)
is equivalently a \(T\)-equivariant map
\[
  \widetilde f\colon Y\longrightarrow\Res_\varphi Y.
\]
The change of action is encoded by a cartesian square
\begin{equation}\label{eq:change-of-action-square}
\begin{tikzcd}[column sep=large]
 {[\Res_\varphi Y/T]} \arrow[r] \arrow[d,"q_\varphi"']
   & {[Y/T]} \arrow[d,"q"] \\
 BT \arrow[r,"B\varphi"'] & BT.
\end{tikzcd}
\end{equation}

\begin{proposition}
\label{prop:semilinear-proper-pullback}
Suppose that \(Y\) is a qcqs separated finite-type \(T\)-algebraic space and
that \(f\colon Y\to Y\) is proper and \(\varphi\)-semilinear.  Then the map
\[
  [\widetilde f/T]\colon[Y/T]\longrightarrow[\Res_\varphi Y/T]
\]
is representable and proper.  With \(B=BT\) and \(g=B\varphi\),
Definition~\ref{def:semilinear-object} therefore produces a canonical
arrow
\begin{equation}\label{eq:semilinear-proper-arrow}
 (B\varphi)^*\Mcompact_T(Y)\longrightarrow\Mcompact_T(Y).
\end{equation}
The arrows for composable semilinear proper maps agree with the arrows obtained
by composing the corresponding proper pullbacks under the canonical base-change
identifications.
\end{proposition}

\begin{proof}
Square \eqref{eq:change-of-action-square} identifies the source of
\eqref{eq:semilinear-proper-arrow} with the compactly supported motive of
\([\Res_\varphi Y/T]\).  The quotient map induced by \(\widetilde f\) uses the
same group on source and target.  It is therefore representable, and its
properness can be checked after pulling back along the atlas
\(\Res_\varphi Y\to[\Res_\varphi Y/T]\), where it becomes a base change of
\(f\).  The proper-pullback construction now applies.  Compatibility with
composition is the functoriality of the adjunction units together with coherence
of exceptional base change.
\end{proof}

For a proper \(\varphi\)-semilinear map \(f\colon Y\to Y\), write
\[
 \Xi_f(Y):=\Xi_{B\varphi}
 \bigl([Y/T],[\widetilde f/T]\bigr)
\]
for the resulting semilinear compactly supported lacing.  All lacings of
toric power and monomial maps used below are instances of this notation.

We now choose an identification \(N\simeq\mathbb Z^n\), so
\(T\simeq\Gm^n\), fix \(r\geq1\), and specialize to \(\varphi=[r]\).  Put
\[
  F_\varphi=(B\varphi)^*,
  \qquad \cM_\varphi(x,y)=\Map(F_\varphi x,y),
  \qquad \Lace_\varphi(\cC)=\Lace(\cC,\cM_\varphi),
  \qquad \Kzero^\varphi(\cC)=\Kzero\Lace_\varphi(\cC).
\]
For a subtorus \(H\subset T\), let
\(
  H^{(1)}=\varphi^{-1}(H).
\)
Base change along \(B\varphi\) identifies
\begin{equation}\label{eq:base-change-BH}
 F_\varphi(BH\to BT)_!\one_{BH}
 \simeq (BH^{(1)}\to BT)_!\one_{BH^{(1)}}.
\end{equation}
The inclusion \(H\subset H^{(1)}\) has finite quotient and hence gives a
representable proper map \(BH\to BH^{(1)}\).

\begin{definition}
\label{def:stabilizer-lacing}
Let \(E_T(H)=(BH\to BT)_!\one_{BH}\).  We denote by
\(\mathcal E_\varphi(H)\) the laced object
\[
 \left(
 E_T(H),
 F_\varphi E_T(H)
 \overset{\eqref{eq:base-change-BH}}{\simeq}
 E_T(H^{(1)})
 \longrightarrow E_T(H)
 \right),
\]
where the last arrow is proper pullback along \(BH\to BH^{(1)}\).  Its class in
\(\Kzero^\varphi(\cC)\) is written \(\eps_\varphi(H)\).
\end{definition}

Both \(H\) and \(\varphi^{-1}(H)\) are of multiplicative type, hence nice in
the sense of
\cite[Definition~2.1]{KhanRavi}, and their classifying stacks are scalloped by
\cite[Theorem~2.14]{KhanRavi}.  More generally, if \(H\subset H'\) is an
inclusion of multiplicative-type group schemes with finite quotient, then
\(BH\to BH'\) is representable and finite, hence proper.

\subsection{Nil-invariant localization and the scalar orbit formula}
\label{sec:localization-orbit}

An invariant closed subscheme is not necessarily its own scheme-theoretic
inverse image: the map \(z\mapsto z^r\) at the origin of \(\Aone\) is the
basic example.  The next proposition uses nil-invariance to compare the
resulting thickening with the reduced closed subscheme.

Retain the notation of Section~\ref{sec:semilinear-compact-support}.  Thus
\(a\colon Y\to Y^{(g)}\) is a representable proper \(B\)-morphism and
\(\Xi_g(Y,a)\) is its compactly supported lacing.

\begin{proposition}
\label{prop:nil-semi-localization}
Let \(i\colon Z\hookrightarrow Y\) be a finite-type closed immersion with
quasi-compact open complement \(j\colon U\hookrightarrow Y\).  Suppose that
\(a\) restricts to representable proper \(B\)-morphisms
\[
  a_Z\colon Z\longrightarrow Z^{(g)},
  \qquad
  a_U\colon U\longrightarrow U^{(g)},
\]
and assume the following two scheme-theoretic conditions:
\begin{enumerate}
\item \(a^{-1}(U^{(g)})=U\);
\item for
  \(
    Z'=Y\times_{Y^{(g)}}Z^{(g)},
  \)
  the induced map \(\nu\colon Z\to Z'\) is a finite-type surjective nilpotent
  closed immersion.
\end{enumerate}
If \(\Mcompact_B(U)\), \(\Mcompact_B(Y)\), \(\Mcompact_B(Z)\), and
\(\Mcompact_B(Z')\) belong to \(\cC\), then there is
a cofiber sequence
\begin{equation}\label{eq:laced-localization-sequence}
 \Xi_g(U,a_U)\longrightarrow
 \Xi_g(Y,a)\longrightarrow
 \Xi_g(Z,a_Z)
\end{equation}
in \(\Lace_g(\cC)\).
\end{proposition}

\begin{proof}
Put \(x_W=\Mcompact_B(W)\) and write \(\alpha_W:g^*x_W\to x_W\) for the
lacing.  Apply the unit \(\mathrm{id}\to a_*a^*\) to the localization triangle on
\(Y^{(g)}\).  Exceptional base change and properness of \(a\) give a morphism
of triangles
\begin{equation}\label{eq:first-localization-morphism}
\begin{tikzcd}[column sep=large]
 g^*x_U \arrow[r] \arrow[d,"\alpha_U"']&g^*x_Y \arrow[r]
 \arrow[d,"\alpha_Y"']&g^*x_Z \arrow[d,"(a')_c^*"]\\
 x_U \arrow[r]&x_Y \arrow[r]&\Mcompact_B(Z'),
\end{tikzcd}
\end{equation}
where \(a':Z'\to Z^{(g)}\); hypothesis~(1) identifies the lower open term
with \(U\).  Nil-invariance makes
\(\nu_c^*:\Mcompact_B(Z')\to\Mcompact_B(Z)\) an equivalence
\cite[Corollary~5.9]{KhanRavi}; compare \cite{ElmantoKhan}.  The corresponding
morphism of localization triangles is the identity on the open and middle
terms and \(\nu_c^*\) on the closed term.  Since \(a_Z=a'\nu\),
\(\nu_c^*(a')_c^*=(a_Z)_c^*=\alpha_Z\); thus its composite with
\eqref{eq:first-localization-morphism} is a morphism from the \(g^*\)-image of
the \((U,Y,Z)\)-triangle to itself with components \(\alpha_U,\alpha_Y,
\alpha_Z\).  Exactness of \(g^*\) and pointwise finite cofibers in the HNS
arrow model \cite[Construction~2.5]{HNS} prove
\eqref{eq:laced-localization-sequence}.
\end{proof}

\subsection{The toric filtration}

Return to \(T=\Gm^n\), \(\varphi=[r]\), and a finite fan \(\Sigma\).  Let
\(X_{\leq d}\subset X_\Sigma\) be the union of all torus orbits of dimension at
most \(d\), with its reduced induced structure.  Orbit-closure relations imply
that it is closed.  On a toric affine chart its ideal is a radical monomial
ideal.  Pullback by the power map sends a monomial \(\chi^m\) to
\(\chi^{rm}\).  Consequently:
\begin{itemize}
\item the complementary open is its exact scheme-theoretic inverse image;
\item the inverse image of \(X_{\leq d}\) is a nilpotent thickening of
  \(X_{\leq d}\).
\end{itemize}
The power map is finite on each affine chart because the corresponding monoid
algebra is finite over the image of \(\chi^m\mapsto\chi^{rm}\).  Hence the
semilinear map used throughout the filtration is proper.

Let \(\cC\subset\SH(BT)\) be the small thick stable subcategory generated by
the
compactly supported quotient-stack motives of these filtration stages, their
strata, and all \(F_\varphi\)-iterates.
Proposition~\ref{prop:semilinear-proper-pullback} defines
\[
  \Xi_{\Sigma,r}
  =\bigl(\Mcompact_T(X_\Sigma),\lambda_{\Sigma,r}\bigr)
  \in\Lace_\varphi(\cC).
\]

\begin{theorem}
\label{thm:semilinear-orbit}
For every finite fan \(\Sigma\) and every \(r\geq1\),
\begin{equation}\label{eq:semilinear-orbit}
  [\Xi_{\Sigma,r}]
  =\sum_{\sigma\in\Sigma}\eps_\varphi(T_\sigma)
  \qquad\text{in }\Kzero^\varphi(\cC).
\end{equation}
The summand at \(\sigma\) is represented by the proper-pullback arrow
\[
 (BT_\sigma^{(1)}\to BT)_!\one
 \longrightarrow
 (BT_\sigma\to BT)_!\one,
 \qquad T_\sigma^{(1)}=\varphi^{-1}(T_\sigma).
\]
This includes inseparable isogenies and non-smooth quotients
\(T_\sigma^{(1)}/T_\sigma\).
\end{theorem}

\begin{proof}
Apply Proposition~\ref{prop:nil-semi-localization} successively to
\(
 X_{\leq d-1}\subset X_{\leq d}.
\)
The open complement at stage \(d\) is the finite disjoint union
\[
  X_{\leq d}\setminus X_{\leq d-1}
  =\coprod_{\dim O_\sigma=d}O_\sigma.
\]
Taking Grothendieck classes of the resulting cofiber sequences gives
\begin{equation}\label{eq:sum-of-orbit-lacings}
 [\Xi_{\Sigma,r}]
  =\sum_{\sigma\in\Sigma}
    [\Xi_{B\varphi}(O_\sigma,[r]|_{O_\sigma})].
\end{equation}

Set \(H=T_\sigma\).  The quotient of the homogeneous orbit and the quotient of
its twisted action are
\[
  [O_\sigma/T]\simeq BH,
  \qquad
  [\Res_\varphi O_\sigma/T]\simeq BH^{(1)}.
\]
Under these equivalences, the map induced by the power map is
\(
  BH\to BH^{(1)}
\)
from \(H\subset H^{(1)}\).  Base change identifies its source motive as in
\eqref{eq:base-change-BH}, and its lacing arrow is the proper pullback in
Definition~\ref{def:stabilizer-lacing}.  Thus the summand in
\eqref{eq:sum-of-orbit-lacings} is \(\eps_\varphi(T_\sigma)\), proving
\eqref{eq:semilinear-orbit}.
\end{proof}

\begin{corollary}
\label{cor:all-iterates}
For every \(m\geq1\), the power functor sends
\(\Xi_{\Sigma,r}\) to \(\Xi_{\Sigma,r^m}\).  It sends the stabilizer term at
\(H\subset T\) to the lacing associated with
\[
  BH\longrightarrow B\varphi^{-m}(H).
\]
Consequently Theorem~\ref{thm:semilinear-orbit} yields the orbit formula
for \([r]^m=[r^m]\) with coefficient bimodule \(\cM_{\varphi^m}\).
\end{corollary}

\begin{proof}
Proposition~\ref{lem:coefficient-composition} identifies
\(P_m\Xi_{\Sigma,r}\) with the lacing whose arrow is the ordered composite of
the \(m\) proper pullbacks.  Compatibility of semilinear proper pullback with
composition, as established in
Proposition~\ref{prop:semilinear-proper-pullback}, identifies this arrow
with the proper pullback for \([r]^m=[r^m]\).  Thus the resulting lacing is
\(\Xi_{\Sigma,r^m}\).  On a homogeneous orbit, the composite is attached to
\[
 H\subset\varphi^{-1}(H)\subset\cdots\subset\varphi^{-m}(H),
\]
and is therefore the proper pullback along
\(BH\to B\varphi^{-m}(H)\).  Apply the exact functor \(P_m\) to
\eqref{eq:semilinear-orbit}.
\end{proof}

\begin{remark}
\label{rem:inseparable-power}
After choosing split coordinates on the quotient torus,
\[
 T_\sigma^{(1)}/T_\sigma
 \simeq\ker\bigl([r]\colon T/T_\sigma\to T/T_\sigma\bigr)
 \simeq\mu_r^{\dim O_\sigma}.
\]
This group scheme is non-smooth when the characteristic divides \(r\).
\end{remark}

\subsection{Cones fixed by iterates of monomial isogenies}

Applying the laced trace to the scalar orbit formula gives the following
coefficient-valued identity.

\begin{corollary}
\label{cor:thh-orbit}
For every finite fan \(\Sigma\) and every \(r\geq1\),
\begin{equation}\label{eq:thh-orbit}
 \operatorname{tr}^{\mathrm{lace}}(\Xi_{\Sigma,r})
 =\sum_{\sigma\in\Sigma}
   \operatorname{tr}^{\mathrm{lace}}
   \bigl(\mathcal E_\varphi(T_\sigma)\bigr)
 \quad\text{in }\pi_0\THH(\cC,\cM_\varphi).
\end{equation}
For every \(m\geq1\), the same formula applied to the iterated lacing is an
identity in \(\pi_0\THH(\cC,\cM_{\varphi^m})\), with the summand at
\(T_\sigma\) represented by \(BT_\sigma\to B\varphi^{-m}(T_\sigma)\).
\end{corollary}

\begin{proof}
Apply the homomorphism on \(\pi_0\) induced by
\eqref{eq:laced-trace} to
Theorem~\ref{thm:semilinear-orbit}.  For the second assertion, first
apply Corollary~\ref{cor:all-iterates} and then use the laced trace with
coefficient \(\cM_{\varphi^m}\).
\end{proof}

\begin{theorem}
\label{thm:cones-fixed-by-iterates-trace}
Over the fixed field \(k\), let \(N\simeq\mathbb Z^d\), let \(T=T_N\), and let
\(A\colon N\to N\) be injective with finite cokernel.  Write
\(\psi_A\colon T\to T\) for the induced isogeny.  Suppose that \(\Sigma\) is
a finite fan such that
\begin{equation}\label{eq:strictly-fan-permuting}
 A_{\mathbb R}(\sigma)\in\Sigma
 \quad\text{as a cone for every }\sigma\in\Sigma.
\end{equation}
Then \(A_{\mathbb R}\) permutes the cones of \(\Sigma\), and the induced
toric endomorphism
\(f_A\colon X_\Sigma\to X_\Sigma\) is finite and
\(\psi_A\)-semilinear.  Put \(F_A=(B\psi_A)^*\), and take a small
\(F_A\)-stable thick category \(\cC\subset\SH(BT)\) containing the orbit
motives, the filtration-stage motives, and the compactly supported motive of
\(X_\Sigma\).

For a cycle \(C\in\Sigma/\langle A\rangle\), set
\[
 O_C=\coprod_{\sigma\in C}O_\sigma
\]
and let \(\mathcal E_A(C)\in\Lace(\cC,\cM_{F_A})\) be the compactly
supported lacing of \(f_A|_{O_C}\).  Then
\begin{equation}\label{eq:cone-cycle-decomposition-kzero}
 [\Xi_{\Sigma,A}]
 =\sum_{C\in\Sigma/\langle A\rangle}[\mathcal E_A(C)]
 \quad\text{in }K_0\Lace(\cC,\cM_{F_A}),
\end{equation}
where \(\Xi_{\Sigma,A}\) is the semilinear compactly supported lacing of
\(f_A\).

For every \(m\geq1\), its iterated coefficient trace is
\begin{equation}\label{eq:cones-fixed-by-iterates-thh}
 \operatorname{tr}^{\mathrm{lace}}
 \bigl(P_m\Xi_{\Sigma,A}\bigr)
 =
 \sum_{\substack{\sigma\in\Sigma\\A^m\sigma=\sigma}}
 \operatorname{tr}^{\mathrm{lace}}
 \bigl(E_T(T_\sigma),\lambda_{\sigma,m}\bigr)
 \quad\text{in }\pi_0\THH(\cC,\cM_{F_A^m}),
\end{equation}
where
\[
 \lambda_{\sigma,m}\colon
 F_A^mE_T(T_\sigma)
 \simeq E_T(\psi_A^{-m}(T_\sigma))
 \longrightarrow E_T(T_\sigma)
\]
is proper pullback along
\(BT_\sigma\to B\psi_A^{-m}(T_\sigma)\).  When
\(A=r\,\mathrm{id}_N\) with \(r\geq1\),
\(A\sigma=\sigma\) for every cone \(\sigma\), and
\eqref{eq:cones-fixed-by-iterates-thh} reduces to
Corollary~\ref{cor:thh-orbit}.
\end{theorem}

\begin{proof}
Since \(A_{\mathbb R}\) is invertible, its self-map of the finite set
\(\Sigma\) is injective and hence permutes the cones.  For \(\tau=A\sigma\),
the affine-chart map is
\[
 k[\tau^\vee\cap M]\longrightarrow k[\sigma^\vee\cap M],
 \qquad \chi^u\longmapsto\chi^{A^\vee u}.
\]
For \(v\in\sigma^\vee\cap M\), some \(qv=A^\vee u\), with
\(u\in\tau^\vee\); hence \(\chi^v\) is integral over the image.  Finite
generation makes every chart map finite.  Moreover the cones mapping into
\(\tau\) are exactly the faces of \(\sigma\), so \(f_A^{-1}U_\tau=U_\sigma\)
scheme-theoretically.  This proves finiteness; semilinearity is the defining
equivariance of the toric map.

The orbit-dimension filtration is invariant.  On a noetherian affine chart,
if \(I\) is the radical ideal of a reduced stage and \(J=f_A^*I\), invariance
and equality of underlying closed sets give \(J\subseteq I=\sqrt J\); hence
\(I/J\) is nilpotent.  Thus both hypotheses of
Proposition~\ref{prop:nil-semi-localization} hold (and descend from the
atlas to the quotient stack).  Applying it stagewise and grouping each open
stratum into \(A\)-cycles proves
\eqref{eq:cone-cycle-decomposition-kzero}.

For a cycle \(C=(\sigma_0,\ldots,\sigma_{q-1})\), the lacing has precisely the
cyclic blocks
\(F_AE_T(T_{\sigma_{i+1}})\to E_T(T_{\sigma_i})\), induced by
\(T_{\sigma_i}\subseteq\psi_A^{-1}(T_{\sigma_{i+1}})\).  After \(P_m\) the
matrix is shifted by \(m\) positions.  Proposition~\ref{prop:laced-trace-matrix} therefore kills its
laced trace unless \(q\mid m\); when \(q\mid m\), its diagonal
blocks are exactly the return maps \(\lambda_{\sigma_i,m}\).  Since this
divisibility is equivalent to \(A^m\sigma_i=\sigma_i\), summing over cycles
gives \eqref{eq:cones-fixed-by-iterates-thh}.
\end{proof}
\subsection{Intrinsic restriction classes and an example detected by the second ghost}

\begin{corollary}
\label{cor:intrinsic-cones-fixed-by-iterates-restriction}
In the setting of Theorem~\ref{thm:cones-fixed-by-iterates-trace}, let \(S\) be
a nonempty truncation set.  The class
\[
 \mathcal Z^{\mathrm{res,int}}_{\Sigma,A,S}
 :=\pi_0\operatorname{tr}^{\mathrm{res,int}}_S
   [\Xi_{\Sigma,A}]
 \in\pi_0R_S\bigl(\mathbf E(\cC,F_A)\bigr)
\]
has the following \(m\)-th ghost for every \(m\in S\):
\begin{equation}\label{eq:intrinsic-cones-fixed-by-iterates-ghost}
 \pi_0\overline g_{m,S}
 \bigl(\mathcal Z^{\mathrm{res,int}}_{\Sigma,A,S}\bigr)
 =
 \sum_{\substack{\sigma\in\Sigma\\A^m\sigma=\sigma}}
 \operatorname{tr}^{\mathrm{lace}}
 \bigl(E_T(T_\sigma),\lambda_{\sigma,m}\bigr)
 \quad\text{in }
 \pi_0\THH(\cC,\cM_{F_A^m}).
\end{equation}
For \(S=\mathbb N_{>0}\), we write
\(\mathcal Z^{\mathrm{res,int}}_{\Sigma,A}\).
\end{corollary}

\begin{proof}
Apply Theorem~\ref{thm:integral-module-envelope-trace-descent} to
\([\Xi_{\Sigma,A}]\), and combine its intrinsic ghost formula with
Theorem~\ref{thm:cones-fixed-by-iterates-trace}.
\end{proof}

For the scalar matrix \(A=r\,\operatorname{id}_N\), we abbreviate
\(\mathcal Z^{\mathrm{res,int}}_{\Sigma,A,S}\) to
\(\mathcal Z^{\mathrm{res,int}}_{\Sigma,r,S}\).

\begin{corollary}
\label{cor:toric-restriction-frobenius}
In the setting of
Corollary~\ref{cor:intrinsic-cones-fixed-by-iterates-restriction}, let
\(m\geq1\) and suppose \(S/m\ne\varnothing\).  Under
the canonical equivalence \(F_A^m\simeq F_{A^m}\), Frobenius transports the
intrinsic class by
\begin{equation}\label{eq:toric-restriction-frobenius}
 (\pi_0\operatorname{Fr}_{m,S})
 \bigl(\mathcal Z^{\mathrm{res,int}}_{\Sigma,A,S}\bigr)
 =\mathcal Z^{\mathrm{res,int}}_{\Sigma,A^m,S/m}.
\end{equation}
Consequently, under the conjugated block equivalence
\(\overline b_{m,1}^{F_A}\), the \(m\)-th ghost indexed by cones fixed by
\(A^m\) is the first ghost of the Frobenius-transported class:
\begin{equation}\label{eq:fixed-by-iterate-ghost-as-first-frobenius}
 (\pi_0\overline b_{m,1}^{F_A})\,
 (\pi_0\overline g_{m,S})
 \bigl(\mathcal Z^{\mathrm{res,int}}_{\Sigma,A,S}\bigr)
 =
 (\pi_0\overline g_{1,S/m})
 \bigl(\mathcal Z^{\mathrm{res,int}}_{\Sigma,A^m,S/m}\bigr).
\end{equation}
\end{corollary}

\begin{proof}
Composition of the proper-pullback units defining the semilinear lacing gives
a canonical equivalence
\(P_m\Xi_{\Sigma,A}\simeq\Xi_{\Sigma,A^m}\).  Apply
Theorem~\ref{thm:restriction-trace-frobenius} to this class to obtain
\eqref{eq:toric-restriction-frobenius}.  Formula
\eqref{eq:restriction-frobenius-ghost} with \(n=1\) gives
\eqref{eq:fixed-by-iterate-ghost-as-first-frobenius}.
\end{proof}

\begin{definition}
\label{def:trace-evaluation-datum}
Let \(\mathcal V\) be a stable symmetric-monoidal category whose tensor
product is exact in each variable, and let
\(i\colon\cD\hookrightarrow\mathcal V\) be an exact strong
symmetric-monoidal inclusion of a small rigid stable full subcategory
containing \(\one\).  A trace-evaluation datum on \(i\) is a map
\[
 \Theta_{\cD}\colon\THH(\cD)\longrightarrow
 \operatorname{End}_{\mathcal V}(\one)
\]
such that, for every \(y\in\cD\) and every
\([h]\in\pi_0\Map_{\cD}(y,y)\),
\begin{equation}\label{eq:ramzi-zero-simplex-composite}
 \pi_0\Theta_{\cD}\bigl((\pi_0j_y)[h]\bigr)
 =\operatorname{Tr}_{\mathcal V}(h).
\end{equation}
Here \(j_y\) is the inclusion of cyclic-bar zero-simplices.
\end{definition}

\begin{lemma}
\label{lem:hns-ramzi-comparison}
A trace-evaluation datum carries the laced-trace class of every endomorphism
\(h\colon y\to y\) to its categorical trace:
\begin{equation}\label{eq:ramzi-hns-trace-compatibility}
 \pi_0\Theta_{\cD}
 \bigl(\operatorname{tr}^{\mathrm{lace}}(y,h)\bigr)
 =\operatorname{Tr}_{\mathcal V}(h).
\end{equation}
\end{lemma}

\begin{proof}
By Proposition~\ref{prop:laced-trace-on-objects}, the laced class is
\((\pi_0j_y)[h]\).  The defining equality of the datum gives
\eqref{eq:ramzi-hns-trace-compatibility}.
\end{proof}

\begin{remark}
Ramzi constructs a map \(\Theta_{\cD}\) and maps
\(\iota_y\colon\Map_{\cD}(y,y)\to\THH(\cD)\) whose composites have the
required trace values on \(\pi_0\)
\cite[Construction~2.8 and Corollary~2.16]{Ramzi}.  For the \(GW\)-valued
applications in Section~\ref{sec:log-zeta}, we assume
\(\pi_0\iota_y=\pi_0j_y\), so that \(\Theta_{\cD}\) is a trace-evaluation
datum.
\end{remark}

\begin{theorem}
\label{thm:categorical-realization-laced}
Let \(R\) be a commutative ring spectrum and let
\(L\colon\cD\to\Perf(R)\) be an exact functor.  Morita invariance identifies
\(\THH(\Perf(R))\) with \(\THH(R)\), and multiplication in the commutative
algebra \(R\) defines the cyclic-bar augmentation in the composite
\begin{equation}\label{eq:categorical-realization-thh}
 \Theta_L\colon
 \THH(\cD)\longrightarrow\THH(\Perf(R))
 \simeq\THH(R)\longrightarrow R.
\end{equation}
For the last map, see \cite[Remark~4.6.5.10]{LurieHA}.  The composite sends the
laced-trace class of \(h\colon y\to y\) to the
\(R\)-linear categorical trace of \(L(h)\).  More generally, an exact laced
functor
\[
 (\cC,\cM_F)\longrightarrow(\Perf(R),\cM_{\operatorname{id}})
\]
induces the corresponding map on coefficient \(\THH\), again carrying laced
traces to \(R\)-linear categorical traces.
\end{theorem}

\begin{proof}
Functoriality of \(\THH\), Morita invariance, and the commutative
multiplication of \(R\) give \eqref{eq:categorical-realization-thh}.
The zero-simplex comparison for compact modules is
\cite[Theorems~1.2 and~7.8]{CampbellPonto}; hence the image of a zero-simplex
endomorphism is its \(R\)-linear categorical trace.  Proposition
\ref{prop:laced-trace-on-objects} proves the assertion for laced-trace
classes, and the coefficient version follows by functoriality of laced
\(\THH\).
\end{proof}

\begin{corollary}
\label{cor:motivic-laced-realizations}
For the motivic application, let \(e\colon\Spec k\to BT\) classify the
trivial torsor and let
\(\theta\colon e^*F_\varphi\simeq e^*\) be induced by
\(B\varphi\circ e\simeq e\).  The bimodule maps
\[
 \Map_{\SH(BT)}(F_\varphi x,y)\longrightarrow
 \Map_{\SH(k)}(e^*F_\varphi x,e^*y)
 \xrightarrow{\ (-)\circ\theta_x^{-1}\ }
 \Map_{\SH(k)}(e^*x,e^*y)
\]
give \(e^*\) a laced refinement.  If a small rigid subcategory
\(\cD\subset\SH(k)\) containing \(e^*\cC\) is equipped with a
trace-evaluation datum, this refinement gives
\begin{equation}\label{eq:quadratic-realization-laced-thh}
 \rho_\varphi^{GW}\colon
 \pi_0\THH(\cC,\cM_{F_\varphi})\longrightarrow GW(k).
\end{equation}
For a proper \(\varphi\)-semilinear map \(f\colon Y\to Y\) with
\(\Mcompact_T(Y)\in\cC\), the value of \(\rho_\varphi^{GW}\) on the laced
class of \(\Xi_f(Y)\) is \(\operatorname{Tr}(f_c^*\mid\Mcompact_k(Y))\).

Now let \(k=\mathbb C\), and suppose that the objects of \(e^*\cC\) have
constructible Hodge realizations, so that their underlying rational complexes
are perfect; this holds for the compactly supported motives of finite-type
spaces used below.  Drew's realization, followed by the forgetful functor to
\(\Perf(\mathbb Q)\), gives a linear realization
\begin{equation}\label{eq:betti-realization-laced-thh}
 \rho_\varphi^{\mathrm B}\colon
 \pi_0\THH(\cC,\cM_{F_\varphi})\longrightarrow\mathbb Q,
\end{equation}
whose value on the laced class of \(\Xi_f(Y)\) is
\[
 \sum_i(-1)^i
 \operatorname{Tr}\bigl(f_c^*\mid
 H_c^i(Y^{\mathrm{an}},\mathbb Q)\bigr).
\]
\end{corollary}

\begin{proof}
The displayed bimodule morphism gives the laced refinement of \(e^*\);
exceptional base change and naturality of the proper-pullback unit identify
the induced endomorphism with \(f_c^*\).  A trace-evaluation datum and
Lemma~\ref{lem:hns-ramzi-comparison} give the quadratic assertion;
\(\pi_0\operatorname{End}_{\SH(k)}(\one)=GW(k)\), with the bilinear-form
convention also in characteristic \(2\), by \cite[equation~(1.2) and the
discussion following it]{HoyoisTrace}.

Over \(\mathbb C\), Drew's realization is symmetric monoidal, has adjointable
exchange maps, and preserves the relevant units and counits
\cite[Theorem~1.2, Definition~5.1, Definition~5.8,
Proposition~7.14, and Example~8.12]{Drew}.  At the point it takes values in the
derived category of polarizable mixed Hodge structures
\cite[Corollary~4.13 and Definition~4.14]{Drew}.  On compact objects the
forgetful functor lands in \(\Perf(\mathbb Q)\), so the linear construction
above gives \eqref{eq:betti-realization-laced-thh} and the alternating
compactly supported Betti trace.
\end{proof}

\begin{lemma}
\label{lem:gm-power-trace}
For every \(q\geq1\),
\[
 \operatorname{Tr}_c\bigl([q]\mid\Gm;\mathbb Q\bigr)=q-1,
\]
and by K\"unneth the alternating compactly supported trace of the
\(q\)-power map on \(\Gm^{e}\) is \((q-1)^{e}\).
\end{lemma}

\begin{proof}
The only nonzero groups are
\[
 H_c^1(\Gm,\mathbb Q)\simeq\mathbb Q,
 \qquad H_c^2(\Gm,\mathbb Q)\simeq\mathbb Q(-1),
\]
and \([q]_c^*\) has traces \(1\) and \(q\), respectively.
\end{proof}

\begin{proposition}
\label{prop:laced-trace-matrix}
Let \(F\colon\cC\to\cC\) be exact, and let
\[
 x=\bigoplus_{i=1}^q x_i,
 \qquad
 \alpha\colon Fx\longrightarrow x
\]
be a laced object.  Write
\[
 \alpha_{ij}=p_i\alpha F(\iota_j)\colon Fx_j\longrightarrow x_i.
\]
Then
\begin{equation}\label{eq:graph-coefficient-matrix-trace}
 \operatorname{tr}^{\operatorname{lace}}(x,\alpha)
 =
 \sum_{i=1}^q
 \operatorname{tr}^{\operatorname{lace}}(x_i,\alpha_{ii})
 \quad
 \text{in }\pi_0\THH(\cC,\cM_F).
\end{equation}
In particular, a purely off-diagonal lacing has zero laced trace.
\end{proposition}

\begin{proof}
By Proposition~\ref{prop:laced-trace-on-objects}, the trace is
represented by \(\alpha\) in simplicial degree zero.  Write
\[
 \alpha=\sum_{i,j}\iota_i\alpha_{ij}F(p_j).
\]
In the coend
\[
 \THH(\cC,\cM_F)
 \simeq
 \int^{y\in\cC}\Map(Fy,y),
\]
dinaturality for \(\iota_i\colon x_i\to x\) carries the \((i,j)\)-summand
to the class represented at \(x_i\) by
\[
 \alpha_{ij}F(p_j\iota_i).
\]
This vanishes when \(i\neq j\) and equals \(\alpha_{ii}\) when \(i=j\).
\end{proof}

\begin{example}
\label{ex:toric-exchanged-strata-second-ghost}
Let \(k=\mathbb C\), let \(T=\Gm^2\), and let
\[
 A=\begin{pmatrix}0&2\\2&0\end{pmatrix}\colon\mathbb Z^2\longrightarrow
 \mathbb Z^2,
 \qquad
 \psi(t_1,t_2)=(t_2^2,t_1^2).
\]
Consider the two \(A\)-stable fans
\[
 \Sigma_\circ=\{0\},
 \qquad
 \Sigma_\times=
 \{0,\mathbb R_{\geq0}e_1,\mathbb R_{\geq0}e_2\}.
\]
Their toric varieties are
\(U=X_{\Sigma_\circ}=\Gm^2\) and
\(X=X_{\Sigma_\times}=\mathbb A^2\setminus\{0\}\).  The same finite toric
endomorphism
\[
 f(x,y)=(y^2,x^2)
\]
acts on both and is \(\psi\)-semilinear.  Put \(F=(B\psi)^*\), and choose a
small common \(F\)-stable thick subcategory \(\cC\subset\SH(BT)\) containing
the compactly supported motives of \(U\), \(X\), their orbit strata, and all
objects used below, such that the objects of \(e^*\cC\) are dualizable.  As
above, \(\Xi_f(Y)\in\Lace(\cC,\cM_F)\) denotes the
semilinear compactly supported lacing of \(f|_Y\).

For every truncation set \(S\) containing \(2\), define the intrinsic classes
\[
 \mathcal Z^{\mathrm{res,int}}_{Y,f,S}
 =\pi_0\operatorname{tr}^{\mathrm{res,int}}_S[\Xi_f(Y)]
 \in\pi_0R_S\bigl(\mathbf E(\cC,F)\bigr)
 \qquad(Y=U,X).
\]
Their first ghosts agree:
\begin{equation}\label{eq:toric-exchanged-strata-first-ghost}
 \pi_0\overline g_{1,S}
 \bigl(\mathcal Z^{\mathrm{res,int}}_{X,f,S}\bigr)
 =
 \pi_0\overline g_{1,S}
 \bigl(\mathcal Z^{\mathrm{res,int}}_{U,f,S}\bigr)
 \quad\text{in }\pi_0\THH(\cC,\cM_F).
\end{equation}
Their second ghosts are unequal.  Let
\[
 \rho_2\colon\pi_0\THH(\cC,\cM_{F^2})\longrightarrow\mathbb Q
\]
be the Betti linear realization
\eqref{eq:betti-realization-laced-thh} of
Corollary~\ref{cor:motivic-laced-realizations}, with \(F\) replaced by
\(F^2\).  Then
\begin{equation}\label{eq:toric-exchanged-strata-second-ghost}
 \rho_2\!\left(
  \pi_0\overline g_{2,S}
  \bigl(
   \mathcal Z^{\mathrm{res,int}}_{X,f,S}
   -\mathcal Z^{\mathrm{res,int}}_{U,f,S}
  \bigr)
 \right)=2(4-1)=6.
\end{equation}
The two restriction classes are therefore distinct, although their first
ghosts agree.  The nonzero value in
\eqref{eq:toric-exchanged-strata-second-ghost} comes from the two toric orbits
exchanged by \(f\).  Since scalar power maps fix every cone, this
orbit-exchange contribution is absent in the scalar case.

Indeed, here \(Z=X\setminus U\simeq\Gm\amalg\Gm\), and \(f\) exchanges its
components.  The relations \(x^2=f^*y\), \(y^2=f^*x\) prove finiteness, while
\(f^*(xy)=(xy)^2\) shows that the inverse image of \(Z\) is its nilpotent
thickening.  Proposition~\ref{prop:nil-semi-localization} therefore
gives \(\Xi_f(U)\to\Xi_f(X)\to\Xi_f(Z)\).

On \(\Mcompact(Z)=x_1\oplus x_2\), the lacing has only the off-diagonal
blocks \(Fx_2\to x_1\) and \(Fx_1\to x_2\).  Its laced trace is zero by
Proposition~\ref{prop:laced-trace-matrix}, so additivity and the
intrinsic ghost formula give
\eqref{eq:toric-exchanged-strata-first-ghost}.  But
\(f^2(x,y)=(x^4,y^4)\): after \(P_2\), each component has return map \([4]\).
Thus the realized second-ghost difference is
\(2\operatorname{Tr}_c([4]\mid\Gm)=2(4-1)=6\) by
Corollary~\ref{cor:motivic-laced-realizations} and
Lemma~\ref{lem:gm-power-trace}.  This proves
\eqref{eq:toric-exchanged-strata-second-ghost}; equality of the
restriction classes would force equality of this ghost.
\end{example}

\section{Reconstruction from higher ghosts}
\label{sec:finite-reconstruction}

The fixed-cone formula expresses the higher ghosts in terms of periodic
strata.
At the trivial subgroup, these ghosts determine the number of orbits in each
dimension.  With coherent fixed-point realization data, nontrivial fixed
subgroups detect embedded stabilizers and the cycle type on the torus fixed
locus.

\subsection{Geometric fixed-point realizations}

Work over \(\mathbb C\), let \(T=\Gm^d\), \(N=X_*(T)\), \(d\geq1\), and
\(r\geq1\).  We compose geometric fixed points with trace evaluation to
obtain numerical invariants of the stabilizer terms.  Put
\[
 \varphi=[r]\colon T\longrightarrow T,
 \qquad F_\varphi=(B\varphi)^*,
 \qquad \SH_c(\mathbb C)=\SH(\mathbb C)^\omega,
\]
and, for any finite collection of fans, let \(\cC\subset\SH(BT)\) be the
common small \(F_\varphi\)-stable thick subcategory generated by their orbit
and filtration-stage motives and iterates.

Let \(i\colon H\hookrightarrow T\) be a finite constant subgroup such that
\(a=\varphi|_H\) is an automorphism.  This includes a cyclic subgroup
\(C_\ell\) for a prime \(\ell\nmid r\), after choosing a primitive
\(\ell\)-th root of unity, and the subgroup \(T[\ell]\simeq(C_\ell)^d\) used
in the cycle-index theorem.  The canonical comparison of
\cite[Remark~4.9]{KhanRavi} identifies \(\SH(BH)\) with the genuine
\(H\)-equivariant motivic category of Hoyois.  We therefore define
\begin{equation}\label{eq:fixed-point-realization-definition}
  \cF_H\colon
  \SH(BT)\xrightarrow{(Bi)^*}\SH(BH)
  \xrightarrow{\ \simeq\ }\SH^H(\mathbb C)
  \xrightarrow{\Phi^H}\SH(\mathbb C).
\end{equation}
Here \(\Phi^H\) is the geometric fixed-point functor of
\cite[Definition~4.9 and Proposition~4.10]{GepnerHeller}.  It is an exact
symmetric monoidal left adjoint and satisfies
\(
 \Phi^H\Sigma_H^\infty Y_+\simeq\Sigma^\infty(Y^H)_+
\)
for smooth \(H\)-schemes \(Y\).

The orbit filtration and purity reduce the generators of \(\cC\) to compact
motives of smooth torus orbits.  Pullback along stacks preserves compactness
\cite[Theorem~4.10(iii)]{KhanRavi}.  The right adjoint of \(\Phi^H\) is
\(\widetilde E\mathcal F[H]\otimes\pi^*(-)\)
\cite[Proposition~4.10]{GepnerHeller}.  Here \(\pi^*\) preserves colimits by
the fixed-point adjunction constructed in
\cite[Section~4.2]{GepnerHeller}, and tensoring with
\(\widetilde E\mathcal F[H]\) preserves colimits.  The right adjoint therefore
preserves filtered colimits.  Since the source and target are compactly
generated stable \(\infty\)-categories, the adjunction criterion for compact
objects shows that \(\Phi^H\) preserves compact objects.  Consequently
\(\cF_H\) preserves compact objects and \(\cF_H\cC\) consists of compact
objects.

\begin{lemma}
\label{lem:fixed-point-orbit-motive}
For every torus orbit \(O_\sigma=T/T_\sigma\), one has
\begin{equation}\label{eq:fixed-point-orbit-motive}
 \cF_H\Mcompact_T(O_\sigma)\simeq
 \begin{cases}
  \Mcompact_{\mathbb C}(O_\sigma),&H\subset T_\sigma,\\
  0,&H\not\subset T_\sigma.
 \end{cases}
\end{equation}
\end{lemma}

\begin{proof}
Let \(u\colon O_\sigma\to\operatorname{Spec}(\mathbb C)\) be the structure
morphism.  Exceptional base change and the smooth pullback equivalence
\(\Sigma^{-\Omega_u}u^!\simeq u^*\)
\cite[Theorem~6.18(2)]{HoyoisSix} (see also the published corrigendum
\cite{HoyoisSixCorrigendum}) identify the pullback of the orbit motive with
the compactly supported equivariant Thom spectrum on \(O_\sigma\).
Geometric fixed points commute with the Thom twists that occur here.  Indeed,
let \(V\to Y\) be an \(H\)-equivariant vector bundle on a smooth
\(H\)-scheme.  Its equivariant Thom spectrum is the cofiber of
\[
 \Sigma_H^\infty(V\smallsetminus Y)_+
 \longrightarrow \Sigma_H^\infty V_+.
\]
Both schemes in this cofiber are smooth, so exactness of \(\Phi^H\) and the
suspension-spectrum formula of \cite[Proposition~4.10]{GepnerHeller}
identify its geometric fixed points with
\[
 \operatorname{cofib}\bigl(
   \Sigma^\infty\bigl((V\smallsetminus Y)^H\bigr)_+
   \longrightarrow \Sigma^\infty(V^H)_+
 \bigr).
\]
Since \(H\) is finite and the base field has characteristic zero, the
averaging idempotent splits \(V|_{Y^H}\) into its invariant subbundle and a
complement.  Hence the fixed locus of the total space is the total space of
\(V^H\to Y^H\), and
\((V\smallsetminus Y)^H=V^H\smallsetminus Y^H\).  The last cofiber is
precisely the Thom spectrum of \(V^H\to Y^H\); in other words,
\[
 \Phi^H\Sigma_H^\infty\operatorname{Th}_Y(V)
 \simeq
 \Sigma^\infty\operatorname{Th}_{Y^H}(V^H).
\]
The equivalence is compatible with direct sums.  Since \(\Phi^H\) is
symmetric monoidal, it extends to virtual Thom twists.

Apply this observation to the purity bundles above.  If
\(H\subset T_\sigma\), then every element of \(H\) acts as the identity on
\(O_\sigma=T/T_\sigma\).  Its differential is therefore the identity, so
\(H\) acts trivially on the tangent and cotangent bundles of \(O_\sigma\).
Their fixed subbundles are the full bundles, and the Thom fixed-point formula
recovers \(\Mcompact_{\mathbb C}(O_\sigma)\) with no residual representation
sphere.  If \(H\not\subset T_\sigma\), translation on \(T/T_\sigma\) has no
\(H\)-fixed point.  The fixed base of every Thom space in the compactly
supported orbit motive is then empty, so its geometric fixed points vanish.
\end{proof}

\begin{lemma}
\label{lem:geometric-fixed-automorphism-invariance}
Let \(H\) be a finite constant group over \(\mathbb C\), let
\(a\in\operatorname{Aut}(H)\), and write \(a^\star\) for restriction of a
genuine \(H\)-spectrum along \(a\).  Change of groups gives a natural
symmetric monoidal equivalence
\begin{equation}\label{eq:geometric-fixed-automorphism-invariance}
 \epsilon_a\colon\Phi^H a^\star\xrightarrow{\ \simeq\ }\Phi^H,
\end{equation}
coherent under composition of automorphisms.
\end{lemma}

\begin{proof}
The assignment sending a finite constant group to its genuine motivic stable
category extends to the groupoid of finite groups and isomorphisms: an
isomorphism \(a\) induces the symmetric monoidal restriction equivalence
\(a^\star\), functorially in \(a\).  By
\cite[Definition~4.9 and Proposition~4.10]{GepnerHeller}, the functor
\(\Phi^H\) is computed by
\[
 \Phi^H(X)\simeq
 \bigl(X\otimes\widetilde{\mathrm E}\mathcal P\bigr)^H,
\]
where \(\mathcal P\) is the family of proper subgroups of \(H\) and
\(\widetilde{\mathrm E}\mathcal P\) is the unreduced suspension of the
associated universal motivic \(H\)-space.  An automorphism \(a\) permutes
the proper subgroups of \(H\), hence preserves the family \(\mathcal P\);
consequently \(a^\star\widetilde{\mathrm E}\mathcal P\simeq
\widetilde{\mathrm E}\mathcal P\) canonically, by the universal property
defining \(\widetilde{\mathrm E}\mathcal P\), and coherently in \(a\).
Substitution in the fixed-point formula proves
\eqref{eq:geometric-fixed-automorphism-invariance} and its composition
coherence.  For the topological analogue of this change-of-groups formalism,
compare
\cite[Chapter~V]{MandellMay} and
\cite[Theorem~3.12]{MalkiewichTHHDX}.
\end{proof}

\begin{definition}
\label{def:arrow-level-fixed-point-compatibility}
Let \(\psi\colon T\to T\) be a torus isogeny, put
\(F_\psi=(B\psi)^*\), and let \(H\subset T\) be a finite constant subgroup
such that \(a=\psi|_H\) is an automorphism.  Let \(\mathcal P\) be a
collection of pairs \((m,K)\), where \(m\geq1\), \(K\subset T\) is a subtorus,
\(H\subset K\), and \(\psi^m(K)=K\).  Write
\[
 \gamma_H\colon\SH(BH)\xrightarrow{\ \simeq\ }\SH^H(\mathbb C)
\]
for the comparison of \cite[Remark~4.9]{KhanRavi}.  A coherent fixed-point
realization datum for \((\psi,H)\) on \(\mathcal P\) consists of symmetric
monoidal comparisons
\begin{equation}\label{eq:motivic-change-of-groups-datum}
 \chi_m\colon
 \gamma_H(Ba^m)^*\xrightarrow{\ \simeq\ }(a^m)^\star\gamma_H,
 \qquad m\geq1,
\end{equation}
compatible with multiplication of powers.  Together with
Lemma~\ref{lem:geometric-fixed-automorphism-invariance}, these
comparisons specify coherent equivalences
\[
 \theta_{H,m}\colon\cF_HF_\psi^m\xrightarrow{\ \simeq\ }\cF_H.
\]

For every \((m,K)\in\mathcal P\), the datum also includes the compatibility of
\(\theta_{H,m}\) with equivariant purity, exceptional base change, and the
proper-pullback adjunction unit, expressed by the commutative square
\begin{equation}\label{eq:arrow-level-fixed-point-compatibility}
\begin{tikzcd}[column sep=large]
 \cF_HF_\psi^mE_T(K) \arrow[r,"{\cF_H(\lambda^\psi_{K,m})}"]
   \arrow[d,"\theta_{H,m}"'] &
 \cF_HE_T(K) \arrow[d,"\simeq"]\\
 \Mcompact_{\mathbb C}(T/K)
   \arrow[r,"{(\bar\psi_{K,m})_c^*}"'] &
 \Mcompact_{\mathbb C}(T/K),
\end{tikzcd}
\end{equation}
where \(\bar\psi_{K,m}\colon T/K\to T/K\) is the return map and
\(\lambda^\psi_{K,m}\) is proper pullback along
\(BK\to B\psi^{-m}(K)\).  These squares are required to be compatible with
composition of return maps and with the compositors of the \(\theta_{H,m}\)
whenever all pairs in the comparison belong to \(\mathcal P\).  When
\(\mathcal P\) is the collection of all pairs satisfying the conditions
above, we omit the qualifier \emph{on \(\mathcal P\)}.

For \(H\ne1\), the compatibility
\eqref{eq:arrow-level-fixed-point-compatibility} is part of the realization
datum.  For \(H=1\), it follows from exceptional base change and naturality of
the adjunction unit.
\end{definition}

\begin{lemma}
\label{lem:fixed-points-orbit-lacing}
Assume that \(H=1\), or choose a coherent fixed-point realization datum for
\((\varphi,H)\) on the collection of pairs \((m,K)\) under consideration.
The specified
equivalence
\[
 \theta_{H,m}\colon\cF_HF_\varphi^m\simeq\cF_H
\]
makes \(\cF_H\) a laced functor.  Under this functor, every stabilizer lacing
satisfies
\begin{equation}\label{eq:fixed-point-orbit-lacing}
 \cF_H\bigl(\mathcal E_{\varphi^m}(K)\bigr)
 \simeq
 \begin{cases}
  \bigl(\Mcompact_{\mathbb C}(T/K),[r^m]_c^*\bigr),&H\subset K,\\
  (0,0),&H\not\subset K.
 \end{cases}
\end{equation}
\end{lemma}

\begin{proof}
The underlying-object assertion is
\eqref{eq:fixed-point-orbit-motive}.  If \(H\not\subset K\), then
\((T/K)^H=\varnothing\), so geometric fixed points kill the underlying Thom
spectrum and its lacing is the zero endomorphism.  If \(H\subset K\) and
\(H\ne1\), the arrow assertion is exactly
\eqref{eq:arrow-level-fixed-point-compatibility} for
\(\psi=\varphi\), whose return map is \([r^m]\colon T/K\to T/K\).  For
\(H=1\), exceptional base change along the trivial-torsor point and
naturality of the proper-pullback unit identify the resulting arrow directly
with \([r^m]_c^*\) on \(T/K\).
\end{proof}

\subsection{Scalar \texorpdfstring{\(p\)}{p}-typical reconstruction}

For scalar power maps we use the Betti linear realization
\eqref{eq:betti-realization-laced-thh} of
Corollary~\ref{cor:motivic-laced-realizations}.  The relevant traces
are computed in Lemma~\ref{lem:gm-power-trace}.

\begin{theorem}
\label{thm:finite-ptypical-stabilizer-reconstruction}
Let \(k=\mathbb C\), let \(T=\Gm^d\) with \(d\geq1\) and cocharacter
lattice \(N=X_*(T)\), let \(r>1\), and fix a prime \(p\).  Put
\[
 \varphi=[r]\colon T\longrightarrow T,
 \qquad F_\varphi=(B\varphi)^*.
\]
Put
\[
 S_{p,d}=\langle p^d\rangle=\{1,p,\ldots,p^d\}.
\]
Suppose that \(\mathfrak F\) is a finite collection of finite fans in the
common vector space \(N_{\mathbb R}\) and that
\(\cC\subset\SH(BT)\) is the small common thick \(F_\varphi\)-stable
subcategory generated by their orbit and filtration-stage motives and all
of their \(F_\varphi\)-iterates.  Let \(H\hookrightarrow T\) be
either the trivial subgroup or a cyclic constant subgroup of prime order
\(\ell\nmid r\).  If \(H\neq1\), choose a coherent fixed-point realization
datum for \((\varphi,H)\) on the collection of pairs \((m,K)\) with
\(m\geq1\) and \(K\) a stabilizer occurring in \(\mathfrak F\) such that
\(H\subset K\).
Then, for every
\(m\geq1\), there is a numerical
fixed-point trace
\begin{equation}\label{eq:numerical-fixed-point-coefficient-trace}
 \lambda_{H,m}\colon
 \pi_0\THH(\cC,\cM_{F_\varphi^m})\longrightarrow\mathbb Q
\end{equation}
with
\begin{equation}\label{eq:numerical-fixed-point-orbit-value}
 \lambda_{H,m}\!\left(
  \operatorname{tr}^{\mathrm{lace}}
  \bigl(\mathcal E_{\varphi^m}(K)\bigr)
 \right)
 =
 \begin{cases}
  (r^m-1)^{d-\dim K},&H\subset K,\\
  0,&H\not\subset K
 \end{cases}
\end{equation}
for every stabilizer \(K=T_\sigma\) occurring in a fan of \(\mathfrak F\).

For \(\Sigma\in\mathfrak F\), define, for \(0\leq j\leq d\),
\begin{equation}\label{eq:fixed-subgroup-rank-ghost}
 a_{H,j}(\Sigma,r)=
 \lambda_{H,p^j}\!\left(
  \pi_0\overline g_{p^j,S_{p,d}}
  \bigl(\mathcal Z^{\mathrm{res,int}}_{\Sigma,r,S_{p,d}}\bigr)
 \right).
\end{equation}
Then these \(d+1\) numbers recover
\begin{equation}\label{eq:fixed-subgroup-orbit-polynomial}
 P_{\Sigma,H}(z)=
 \sum_{\substack{\sigma\in\Sigma\\H\subset T_\sigma}}
 (z-1)^{\dim O_\sigma}
\end{equation}
by the formula
\begin{equation}\label{eq:fixed-subgroup-lagrange}
 P_{\Sigma,H}(z)=
 \sum_{j=0}^{d}a_{H,j}(\Sigma,r)
 \prod_{\substack{0\leq q\leq d\\q\neq j}}
 \frac{z-r^{p^q}}{r^{p^j}-r^{p^q}}.
\end{equation}
\end{theorem}

\begin{corollary}
\label{cor:embedded-stabilizer-detection}
In the setting of
Theorem~\ref{thm:finite-ptypical-stabilizer-reconstruction},
Corollary~\ref{cor:toric-restriction-frobenius} identifies, under the block
equivalence, the \(p^j\)-th ghost occurring in the definition of
\(a_{H,j}(\Sigma,r)\) with the first ghost of the corresponding Frobenius
image.  Let
\(\mathcal S\) be the finite set of stabilizers occurring in
\(\mathfrak F\).  For every positive-dimensional \(K\in\mathcal S\), there
exists a cyclic subgroup \(H_K\subset K\) of prime order not dividing \(r\)
such that
\begin{equation}\label{eq:separating-test-family-statement}
 H_K\subset L\quad\Longleftrightarrow\quad K\subset L
 \qquad(L\in\mathcal S).
\end{equation}
Choose one such subgroup for each positive-dimensional \(K\), and choose
coherent fixed-point realization data for every nontrivial \(H_K\) on the
collection of pairs \((m,L)\) with \(m\geq1\), \(L\in\mathcal S\), and
\(H_K\subset L\).
Together with \(H_1=1\), these subgroups form a separating test family.  As
\(H\) ranges over this family, the values
\(a_{H,0}(\Sigma,r)=P_{\Sigma,H}(r)\) determine the multiset
\begin{equation}\label{eq:reconstructed-embedded-stabilizer-multiset}
 \bigl\{T_\sigma\hookrightarrow T:\sigma\in\Sigma\bigr\},
\end{equation}
including multiplicities, by triangular inversion over the poset of
stabilizers.  For fixed \(H\), the remaining \(d\) values interpolate
\eqref{eq:fixed-subgroup-orbit-polynomial}.  A finite test family suffices for
any fixed finite collection of fans.

For two finite fans \(\Sigma\) and \(\Sigma'\) in the same
lattice, take \(\mathfrak F=\{\Sigma,\Sigma'\}\).  After placing their orbit
motives in the common category \(\cC\), the following are equivalent:
\begin{enumerate}
 \item the natural images of their classes
 \(\mathcal Z^{\mathrm{res,int}}_{\Sigma,r,S_{p,d}}\) and
 \(\mathcal Z^{\mathrm{res,int}}_{\Sigma',r,S_{p,d}}\) agree in
 \[
  \pi_0R_{S_{p,d}}
  \bigl(\mathbf E(\cC,F_\varphi)\bigr);
 \]
 \item their realized first-ghost values \(a_{H,0}\) agree for a separating
 test family
 satisfying \eqref{eq:separating-test-family-statement} for the union
 of the stabilizers of \(\Sigma\) and \(\Sigma'\);
 \item their multisets of embedded stabilizers agree.
\end{enumerate}
\end{corollary}

\begin{proof}[Proof of
Theorem~\ref{thm:finite-ptypical-stabilizer-reconstruction}]
The numerical map is obtained from \(\cF_H\) of
\eqref{eq:fixed-point-realization-definition}, with \(\cF_1=e^*\), followed
by Drew's symmetric-monoidal six-functor Hodge realization
\cite[Theorem~1.2, Corollary~4.13, Definition~4.14,
Proposition~7.14, and Example~8.12]{Drew} and the forgetful functor:
\[
 \cC\xrightarrow{\ \cF_H\ }\SH_c(\mathbb C)
 \xrightarrow{\ R_{\mathrm{Hdg}}\ }
 D^b(\operatorname{MHS}^{p}_{\mathbb Q})
 \longrightarrow\Perf(\mathbb Q).
\]
Drew identifies the value at the point with the unbounded derived category of
ind-objects of polarizable mixed Hodge structures.  A compact motivic spectrum
is sent to a constructible object, so the composite lands in
\(D^b(\operatorname{MHS}^{p}_{\mathbb Q})\).
For \(H\neq1\), the chosen realization datum and the condition
\(\ell\nmid r\) allow us to apply
Lemma~\ref{lem:fixed-points-orbit-lacing} for every \(m\).  For \(H=1\) we
use ordinary pullback.
Theorem~\ref{thm:categorical-realization-laced}, applied to this laced
realization, defines \(\lambda_{H,m}\) and sends laced-trace classes to the
corresponding \(\mathbb Q\)-linear traces.  The orbit-lacing
lemma kills the
term unless \(H\subset K\), and otherwise gives \([r^m]_c^*\) on
\(T/K\simeq\Gm^{d-\dim K}\).  Adjointable exchange for Hodge realization
preserves compactly supported pullback; Lemma~\ref{lem:gm-power-trace}
therefore proves \eqref{eq:numerical-fixed-point-orbit-value}.

The intrinsic ghost formula
\eqref{eq:integral-intrinsic-ghost-formula} and the orbit formula now
give
\[
 a_{H,j}(\Sigma,r)
 =\sum_{\substack{\sigma\in\Sigma\\H\subset T_\sigma}}
   (r^{p^j}-1)^{\dim O_\sigma}
 =P_{\Sigma,H}(r^{p^j}).
\]
The polynomial has degree at most \(d\), and the \(d+1\) nodes
\(r,r^p,\ldots,r^{p^d}\) are distinct because \(r>1\).  Lagrange
interpolation proves \eqref{eq:fixed-subgroup-lagrange}.
\end{proof}

\begin{proof}[Proof of
Corollary~\ref{cor:embedded-stabilizer-detection}]
For joint detection, fix positive-dimensional \(K\in\mathcal S\).  For each
\(L\not\supset K\), a character vanishing on \(L\) but not on \(K\) shows
that \(K[\ell]\cap L\) is a proper \(\mathbb F_\ell\)-subspace for all but
finitely many \(\ell\).  Choose \(\ell\nmid r\), outside these exceptions and
larger than the number of such \(L\).  Their union cannot cover \(K[\ell]\),
so a nonzero \(h_K\) outside it generates \(H_K\) satisfying
\begin{equation}\label{eq:generic-torsion-subgroup-test}
 H_K\subset L\quad\Longleftrightarrow\quad K\subset L.
\end{equation}
Writing \(n_L=\#\{\sigma:T_\sigma=L\}\), the values \(a_{H_K,0}\) give
\[
 a_{H_K,0}(\Sigma,r)=P_{\Sigma,H_K}(r)
 =\sum_{L\supseteq K}n_L\,(r-1)^{d-\dim L}.
\]
Ordered by decreasing dimension, this system is triangular: the diagonal is
\((r-1)^{d-\dim K}\ne0\), and a connected \(L\supset K\) of equal dimension
equals \(K\).  It recovers every \(n_K\), with \(H=1\) handling the trivial
 stabilizer.  The remaining values recover each \(P_{\Sigma,H}\) by
\eqref{eq:fixed-subgroup-lagrange}.

Applying \(\lambda_{H,1}\circ\pi_0\overline g_{1,S_{p,d}}\) proves
(i)\(\Rightarrow\)(ii), and the triangular inversion above proves
(ii)\(\Rightarrow\)(iii).  Conversely, equal
embedded-stabilizer multisets give equal laced Grothendieck classes by
Theorem~\ref{thm:semilinear-orbit}, hence equal images under the
intrinsic restriction trace at every truncation set.  This proves
(iii)\(\Rightarrow\)(i).
\end{proof}

\begin{corollary}
\label{cor:orbit-dimension-reconstruction}
In the setting of
Theorem~\ref{thm:finite-ptypical-stabilizer-reconstruction}, take
\(H=1\).  The \(d+1\) values \(a_{1,j}\), indexed by
\(1,p,\ldots,p^d\), recover the polynomial
\[
 P_{\Sigma,1}(z)=\sum_{\sigma\in\Sigma}(z-1)^{\dim O_\sigma}.
\]
In particular they recover the number of toric orbits in every dimension.
\end{corollary}

\begin{proof}
Apply Lemma~\ref{lem:fixed-points-orbit-lacing} with \(H=1\).
Equation~\eqref{eq:fixed-subgroup-lagrange} then gives \(P_{\Sigma,1}\), since
every stabilizer contains the trivial subgroup.  Its coefficients in the basis
\(1,(z-1),\ldots,(z-1)^d\) are the orbit counts by dimension.
\end{proof}

\subsection{The torus fixed-point cycle index}

\begin{theorem}
\label{thm:toric-fixed-point-cycle-index}
Let \(T=\Gm^d\) over \(\mathbb C\), let \(\Sigma\) be a finite fan in
\(N_{\mathbb R}\), and let \(B\in\operatorname{GL}(N)\) satisfy
\(B\Sigma=\Sigma\).  Fix \(r>1\), put \(A=rB\), and let \(f_A\) be the
induced finite toric endomorphism.  Write \(\Sigma(d)\) for the
full-dimensional cones and let \(h\) be the order of the permutation induced
by \(B\) on \(\Sigma(d)\).  Let \(S\) be any nonempty truncation set
containing every divisor of \(h\); one may take \(S=\langle h\rangle\).
Choose a prime \(\ell\nmid r\), and, after choosing a primitive
\(\ell\)-th root of unity, set
\[
 H=T[\ell]\simeq(C_\ell)^d\subset T(\mathbb C).
\]
Choose a coherent fixed-point realization datum for \((\psi_A,H)\) on the
collection of pairs \((m,T)\) with \(m\in S\).  Choose the small
\(F_A\)-stable category \(\cC\) as in
Corollary~\ref{cor:intrinsic-cones-fixed-by-iterates-restriction}.
For every \(m\in S\), there is a numerical trace
\begin{equation}\label{eq:maximal-cone-numerical-trace}
 \lambda^{\max}_{H,m}\colon
 \pi_0\THH(\cC,\cM_{F_A^m})\longrightarrow\mathbb Q
\end{equation}
such that
\begin{equation}\label{eq:maximal-cone-fixed-count}
 \nu_m:=\lambda^{\max}_{H,m}\!\left(
  \pi_0\overline g_{m,S}
  \bigl(\mathcal Z^{\mathrm{res,int}}_{\Sigma,A,S}\bigr)
 \right)
 =\#\operatorname{Fix}
 \bigl(B^m\colon\Sigma(d)\longrightarrow\Sigma(d)\bigr).
\end{equation}
Under the block identification, \(\nu_m\) is obtained by applying
\(\lambda^{\max}_{H,m}\) to the first ghost of
\(\operatorname{Fr}_{m,S}
 (\mathcal Z^{\mathrm{res,int}}_{\Sigma,A,S})\).

If \(c_q\) is the number of \(q\)-cycles of this permutation, then the
subfamily \(\{\nu_m:m\mid h\}\) recovers its complete cycle type by
\begin{equation}\label{eq:maximal-cone-mobius}
 c_q=\frac1q\sum_{e\mid q}\mu(q/e)\nu_e,
 \qquad q\mid h.
\end{equation}
In particular the intrinsic restriction class determines the permutation
zeta function of the torus fixed locus,
\begin{equation}\label{eq:maximal-cone-zeta}
 \zeta^{T}_{\Sigma,B}(t)
 :=\exp\!\left(\sum_{m\geq1}
 \#\operatorname{Fix}(B^m\mid\Sigma(d))\frac{t^m}{m}\right)
 =\prod_{q\mid h}(1-t^q)^{-c_q}.
\end{equation}
This zeta function depends on the action of \(B\), not merely on the multiset
of embedded orbit stabilizers: the latter records the number of full-dimensional
cones but not their permutation.
\end{theorem}

\begin{proof}
Since \(A=rB\) is invertible modulo \(\ell\), it restricts to an automorphism
of \(H\).  The chosen realization datum supplies
\(\theta_{H,m}\colon\cF_HF_A^m\simeq\cF_H\).  Composing the resulting laced
functor with Hodge realization and applying
Theorem~\ref{thm:categorical-realization-laced} gives
\eqref{eq:maximal-cone-numerical-trace}.

For every proper orbit stabilizer \(K<T\),
\(|K[\ell]|=\ell^{\dim K}<\ell^d=|H|\), so
Lemma~\ref{lem:fixed-point-orbit-motive} leaves only the point orbits
indexed by \(\Sigma(d)\) in the fixed-cone ghost formula.
Such a cone is fixed by \(A^m\) exactly when it is fixed by \(B^m\).
For \(K=T\), \eqref{eq:arrow-level-fixed-point-compatibility} identifies
each surviving lacing with the identity of the point motive, whose trace is
\(1\).  This proves \eqref{eq:maximal-cone-fixed-count};
Corollary~\ref{cor:toric-restriction-frobenius} gives its Frobenius
interpretation.

A \(q\)-cycle contributes all of its \(q\) elements to \(\nu_m\) exactly when
\(q\mid m\).  Therefore
\(
 \nu_m=\sum_{q\mid m}q c_q
\).
M\"obius inversion gives \eqref{eq:maximal-cone-mobius}; every cycle
length divides \(h\), so the displayed finite family suffices.  The standard
exponential identity for a finite permutation gives
\eqref{eq:maximal-cone-zeta}.
\end{proof}

\begin{example}
\label{ex:toric-fixed-point-cycle}
Let \(\Sigma_\square\) be the fan of
\(\mathbb P^1\times\mathbb P^1\), with rays
\(\pm e_1,\pm e_2\), and put
\[
 B_2=-I,
 \qquad
 B_4=\begin{pmatrix}0&-1\\1&0\end{pmatrix},
 \qquad A_i=rB_i.
\]
Use the common truncation set \(S=\langle4\rangle=\{1,2,4\}\) for both
endomorphisms, and choose for each \(A_i\) the coherent fixed-point
realization datum required in
Theorem~\ref{thm:toric-fixed-point-cycle-index}.
Both \(A_i\) give finite toric endomorphisms of the same
toric surface, so
their embedded-stabilizer multisets agree.  The action of \(B_2\) on
the four maximal cones consists of two \(2\)-cycles, while that of \(B_4\)
is one \(4\)-cycle.  The numerical fixed-point data for the two coefficient
systems are
\[
\begin{array}{c|ccc|c}
 &\nu_1&\nu_2&\nu_4&\zeta^T(t)\\ \hline
 B_2&0&4&4&(1-t^2)^{-2}\\
 B_4&0&0&4&(1-t^4)^{-1}.
\end{array}
\]
The displayed values follow immediately from the indicated cycle
decompositions and Theorem~\ref{thm:toric-fixed-point-cycle-index}.
\end{example}

\subsection{Limitations of additive invariants for scalar power maps}

\begin{proposition}
\label{prop:additive-ceiling}
The class \([\Xi_{\Sigma,r}]\) is determined by the multiset
\[
  \bigl\{T_\sigma\hookrightarrow T:\sigma\in\Sigma\bigr\},
\]
including multiplicities.  Hence, whenever two finite fans in the same lattice
have the same multiset of embedded stabilizers, their semilinear laced
Grothendieck classes agree.  Every homomorphism out of
\(\Kzero^\varphi(\cC)\), including every additive trace or fixed-point
realization, has the same limitation.
\end{proposition}

\begin{proof}
By Theorem~\ref{thm:semilinear-orbit}, the class is the sum of the terms
\(\eps_\varphi(T_\sigma)\).  Each term is functorially determined by the
embedded subtorus \(T_\sigma\subset T\) and its inverse image under
\(\varphi\); no face relation or attaching map occurs after passage to the
Grothendieck group.  Equal multisets therefore give equal sums, and applying any
group homomorphism preserves that equality.
\end{proof}

\begin{example}
\label{ex:incidence-not-recoverable}
For example, fans with the same embedded-stabilizer multiset can have
different cone-incidence relations.  In \(N=\mathbb Z^2\),
consider the two fans whose cones are the zero cone, the four rays spanned by
\(\pm e_1\) and \(\pm e_2\), and two cones of dimension two: for the first fan the
opposite quadrants
\(\operatorname{cone}(e_1,e_2)\) and \(\operatorname{cone}(-e_1,-e_2)\), and
for the second fan the adjacent quadrants
\(\operatorname{cone}(e_1,e_2)\) and \(\operatorname{cone}(-e_1,e_2)\).
Both are fans, and both are stable under every scalar map \([r]\).  The rays
\(\pm e_1\) span the same sublattice, as do \(\pm e_2\), so the two
embedded-stabilizer multisets coincide: the trivial subgroup once, each of
the two coordinate subtori twice, and \(T\) twice.  In the first fan the two
two-dimensional cones meet only at the origin; in the second they share the ray
spanned by \(e_2\).  By Proposition~\ref{prop:additive-ceiling}, the
semilinear laced Grothendieck classes of the two associated toric systems
agree, so no homomorphism out of \(\Kzero^\varphi(\cC)\) distinguishes them.
\end{example}

\section{Intrinsic classes and motivic trace evaluation}
\label{sec:log-zeta}

For an ordinary endomorphism the ambient functor is
\(F=\operatorname{id}\), and the endomorphism itself is the lacing.  Given a
trace-evaluation datum, the \(m\)-th realized ghost is the categorical trace
of the \(m\)-th iterate.  These traces are the coefficients of the
\(\mathbb A^1\)-logarithmic zeta function of
Bilu--Ho--Srinivasan--Vogt--Wickelgren \cite{BHSVW}.  The restriction class
is also Frobenius-compatible.  A finite-field example shows that the resulting
realized ghost sequence need not come from the standard big-Witt ring.

Let \(k\) be any field, let \(\mathbf1_k\) be the unit of \(\SH(k)\), and let
\(GW(k)\) denote the Grothendieck--Witt ring of nondegenerate symmetric
bilinear forms over \(k\) (also in characteristic \(2\)).
Morel's computation, in the arbitrary-field form recalled by Hoyois, gives
\[
 \pi_0\operatorname{End}_{\SH(k)}(\mathbf1_k)\cong GW(k);
\]
see \cite[Corollary~6.43 and Lemma~3.10]{MorelA1}; the removal of the
perfectness hypothesis is recalled in \cite[footnote~1]{HoyoisTrace}.

\subsection{The intrinsic TR-trace class}

For every smooth proper \(k\)-scheme \(X\), motivic Atiyah duality makes
\(\Sigma^\infty_+X\) strongly dualizable in \(\SH(k)\); see
\cite[Section~3, equations~(3.1)--(3.4)]{HoyoisTrace}.  Fix a small thick rigid
symmetric-monoidal
\(\mathcal C_k\subset\SH(k)\) containing the required smooth proper motives,
with duals computed in \(\SH(k)\).  Using the reduced bar models of
Theorem~\ref{thm:reduced-bar-replacement}, the construction of
Section~\ref{sec:morita-descent} applies to
\((\mathcal C_k,\operatorname{id})\).  A laced object is an endomorphism
\((x,\alpha)\), with \(P_m(x,\alpha)=(x,\alpha^m)\), and every ghost lands in
\(\THH(\mathcal C_k)\).  A trace-evaluation datum for
\(\mathcal C_k\hookrightarrow\SH(k)\), in the sense of
Definition~\ref{def:trace-evaluation-datum}, realizes the ghosts of this class
in \(GW(k)\).

For a smooth proper \(k\)-scheme \(X\) with an endomorphism
\(\varphi\colon X\to X\), write
\[
 E_X=\Sigma^\infty_+X\in\mathcal C_k,
 \qquad
 f=\Sigma^\infty_+\varphi\colon E_X\longrightarrow E_X.
\]

\begin{definition}
\label{def:tr-trace-class}
Let \(\mathcal C_k\subset\SH(k)\) be the chosen small thick rigid
symmetric-monoidal subcategory containing \(E_X\).  For a nonempty
truncation set \(S\), define
\begin{equation}\label{eq:tr-trace-class}
 \mathcal Z^{\TR}_{X,\varphi,S;\mathcal C_k}
 :=\bigl(\pi_0\operatorname{tr}^{\mathrm{res,int}}_S\bigr)[E_X,f]
 \in\pi_0R_S\bigl(\mathbf E(\mathcal C_k,
 \operatorname{id})\bigr).
\end{equation}
For \(S=\mathbb N_{>0}\), write
\(\mathcal Z^{\TR}_{X,\varphi;\mathcal C_k}\).  When the ambient category is
fixed, we suppress it from the notation.  We also write
\(\operatorname{Fr}_r:=\operatorname{Fr}_{r,\mathbb N_{>0}}\).
\end{definition}

The definition is compatible with an exact fully faithful enlargement
\(j:\mathcal C_k\hookrightarrow\mathcal D_k\):
\[
 j_{*,S}\bigl(\mathcal Z^{\TR}_{X,\varphi,S;\mathcal C_k}\bigr)
 =\mathcal Z^{\TR}_{X,\varphi,S;\mathcal D_k}.
\]
For a fixed \(X\), it is enough to take a small skeleton of the full thick
symmetric-monoidal subcategory of \(\SH(k)\) generated by
\(\mathbf1_k,E_X,E_X^\vee\).  This subcategory is essentially small and
rigid: motivic Atiyah duality makes the three generators strongly dualizable,
and the strongly dualizable objects form a thick symmetric-monoidal
subcategory of \(\SH(k)\).
By Theorem~\ref{thm:integral-module-envelope-trace-descent},
\eqref{eq:tr-trace-class} is independent of the chosen presentation
and compatible with truncation.

\subsection{Comparison with the BHSVW logarithmic zeta function}

Let
\[
 \Theta_{\mathcal C_k}\colon\THH(\mathcal C_k)\longrightarrow
 \operatorname{End}_{\SH(k)}(\mathbf 1_k)
\]
be the map in the chosen trace-evaluation datum.  By
Lemma~\ref{lem:hns-ramzi-comparison}, its map on \(\pi_0\) sends the
laced-trace class of an endomorphism \(h\colon y\to y\) to its motivic
categorical trace \(\operatorname{Tr}_{\SH(k)}(h)\in GW(k)\).
For \(m\geq1\), define the realized ghost homomorphism
\begin{equation}\label{eq:gw-ghost}
 \operatorname{gh}^{GW}_m
 :=(\pi_0\Theta_{\mathcal C_k})\circ
 \pi_0\overline g_{m,\mathbb N_{>0}}\colon
 \pi_0R_{\mathbb N_{>0}}
 \bigl(\mathbf E(\mathcal C_k,\operatorname{id})\bigr)
 \longrightarrow GW(k).
\end{equation}

Bilu--Ho--Srinivasan--Vogt--Wickelgren define, for a smooth proper \(X\) with
an endomorphism \(\varphi\), the \(\mathbb A^1\)-logarithmic zeta function
\cite[Definition~1.1]{BHSVW}
\begin{equation}\label{eq:bhsvw-definition}
 \operatorname{dlog}\zeta^{\mathbb A^1}_{X,\varphi}(t)
 =\sum_{m\geq1}\operatorname{Tr}_{\SH(k)}(\varphi^m)\,t^{m-1}
 \in GW(k)[[t]],
\end{equation}
where \(\operatorname{Tr}_{\SH(k)}(\varphi^m)\) is the categorical trace of
\(\varphi^m\) on the dualizable object \(E_X\); applying the rank recovers the
logarithmic derivative of the classical zeta function.

\begin{theorem}
\label{thm:tr-refines-bhsvw}
Let \(X\) be a smooth proper \(k\)-scheme with an endomorphism \(\varphi\),
and let \(\mathcal C_k\) be as above, equipped with the chosen
trace-evaluation datum.  Then, for every \(m\geq1\),
\begin{equation}\label{eq:ghost-realizes-trace}
 \pi_0\overline g_{m,\mathbb N_{>0}}
 (\mathcal Z^{\TR}_{X,\varphi})
 =\operatorname{tr}^{\mathrm{lace}}(E_X,f^m)
 \quad\text{in }\pi_0\THH(\mathcal C_k),
\end{equation}
and consequently
\begin{equation}\label{eq:gw-ghost-is-trace}
 \operatorname{gh}^{GW}_m(\mathcal Z^{\TR}_{X,\varphi})
 =\operatorname{Tr}_{\SH(k)}(\varphi^m).
\end{equation}
Hence the realized ghosts assemble to the BHSVW logarithmic zeta function:
\begin{equation}\label{eq:dlog-from-ghosts}
 \operatorname{dlog}\zeta^{\mathbb A^1}_{X,\varphi}(t)
 =\sum_{m\geq1}
 \operatorname{gh}^{GW}_m(\mathcal Z^{\TR}_{X,\varphi})\,t^{m-1}.
\end{equation}
Moreover the class is Frobenius-compatible: for every \(r\geq1\),
\begin{equation}\label{eq:frobenius-transports-zeta}
 \operatorname{Fr}_{r}\bigl(\mathcal Z^{\TR}_{X,\varphi}\bigr)
 =\mathcal Z^{\TR}_{X,\varphi^{r}},
\end{equation}
so that
\(\operatorname{gh}^{GW}_n(\operatorname{Fr}_r\mathcal Z^{\TR}_{X,\varphi})
 =\operatorname{gh}^{GW}_{rn}(\mathcal Z^{\TR}_{X,\varphi})\).
\end{theorem}

\begin{proof}
For \(F=\operatorname{id}\), the intrinsic ghost formula gives
\eqref{eq:ghost-realizes-trace} because
\(P_m[E_X,f]=[E_X,f^m]\).  Trace evaluation sends its right
side to \(\operatorname{Tr}_{\SH(k)}(\varphi^m)\), by motivic Atiyah duality
\cite[Section~3 and Proposition~3.6]{HoyoisTrace}; this proves
\eqref{eq:gw-ghost-is-trace}, and summing proves
\eqref{eq:dlog-from-ghosts}.

Likewise the canonical iteration equivalence identifies
\(P_r[E_X,f]\) with \([E_X,f^r]\).  The Frobenius square of
Theorem~\ref{thm:restriction-trace-frobenius}, together with
\(\mathbb N_{>0}/r=\mathbb N_{>0}\), gives
\eqref{eq:frobenius-transports-zeta}.  Applying
\eqref{eq:gw-ghost-is-trace} yields
\(\operatorname{gh}^{GW}_n(\operatorname{Fr}_r\mathcal Z)
=\operatorname{Tr}(\varphi^{rn})\).
\end{proof}

\subsection{The local fixed-point formula as a realized ghost}

When the fixed loci are \'etale, Hoyois's quadratic refinement of the
Grothendieck--Lefschetz--Verdier trace formula computes each realized ghost.

\begin{proposition}
\label{prop:hoyois-local-ghost}
In the setting of Theorem~\ref{thm:tr-refines-bhsvw}, suppose that for
some \(m\geq1\) the fixed scheme \(X^{\varphi^m}\) is \'etale over \(k\), i.e.\
\(1-d\varphi^m_x\) is invertible at every fixed point \(x\).  Then
\begin{equation}\label{eq:hoyois-local-ghost}
 \operatorname{gh}^{GW}_m(\mathcal Z^{\TR}_{X,\varphi})
 =\sum_{x\in X^{\varphi^m}}
 \operatorname{Tr}_{\kappa(x)/k}
 \bigl\langle\det(1-d\varphi^m_x)\bigr\rangle,
\end{equation}
where \(\operatorname{Tr}_{\kappa(x)/k}\) is the Scharlau transfer of the
residue field extension.
\end{proposition}

\begin{proof}
This follows from \eqref{eq:gw-ghost-is-trace} and Hoyois's fixed-point
formula \cite[Corollary~1.10]{HoyoisTrace}.
\end{proof}

\subsection{Failure of standard big-Witt factorization}

For a commutative ring \(A\), let \(HA\) be its Eilenberg--Mac Lane spectrum
and \(\TR_{\mathrm{big}}(HA)\) the limit over all positive fixed-point levels,
as opposed to a \(p\)-typical limit.

For the big Witt ring
\[
 W(A)=\{(a_d)_{d\geq1}:a_d\in A\},
\]
the standard ghost map is
\[
 w=(w_m)_{m\geq1}\colon W(A)\longrightarrow A^{\mathbb N_{>0}},
 \qquad
 w_m((a_d)_d)=\sum_{d\mid m}d\,a_d^{m/d}.
\]
By \cite[Theorem~9.5 and Lemma~9.6]{CLMPZ}, there is a natural ring
isomorphism
\[
 I_A\colon W(A)\xrightarrow{\ \sim\ }
 \pi_0\TR_{\mathrm{big}}(HA)
\]
such that the standard \(m\)-th \(\TR\)-ghost satisfies
\(g_m^{\TR}\circ I_A=w_m\).  We write \(HGW(k)\) for the Eilenberg--Mac
Lane \(E_\infty\)-ring of the commutative ring \(GW(k)\).

\begin{corollary}
\label{cor:no-gw-zeta}
Let \(q\) be odd, choose a nonsquare \(u\in\mathbb F_q^\times\), and put
\(X=\operatorname{Spec}\mathbb F_{q^2}\).  Define
\[
 B_m=
 \begin{cases}
  0,&m\text{ odd},\\
  \langle1\rangle+\langle u\rangle,&m\text{ even}.
 \end{cases}
\]
For the relative \(q\)-power Frobenius \(\varphi\), any trace-evaluation
datum as above identifies the realized ghost sequence of
\(\mathcal Z^{\TR}_{X,\varphi}\) with \((B_m)_m\).  Consequently it is not a
standard big-Witt ghost sequence and admits no ghost-compatible
factorization through the standard big-\(\TR\) target.
\end{corollary}

\begin{proof}
The fixed-point calculation of \cite[Theorem~8.9, equation~(23), and
Example~8.12]{BHSVW}, with the transfer formula of
\cite[Appendix~A, Theorem~A.1]{CalleGinnett}, identifies the traces with
\((B_m)_m\); Theorem~\ref{thm:tr-refines-bhsvw} identifies those
traces with the realized ghosts.

If \((B_m)_m=w((a_d)_d)\), then the first two big-Witt ghost identities give
\[
 B_1=a_1,
 \qquad
 B_2=a_1^2+2a_2.
\]
Thus \(a_1=0\) and \(B_2=2a_2\).  But the determinant homomorphism
\[
 \det\colon GW(\mathbb F_q)\longrightarrow
 \mathbb F_q^\times/(\mathbb F_q^\times)^2\cong\mathbb Z/2
\]
vanishes on \(2\,GW(\mathbb F_q)\), whereas
\(\det(B_2)=u\) is nontrivial.  This contradiction proves the claim; the
statement about the standard big-\(\TR\) target follows from the
ghost-compatible isomorphism \(I_{GW(\mathbb F_q)}\).
\end{proof}

Under the standard identification
\(W(A)\cong(1+tA[[t]])^\times\) of
\cite[equations~(9.1)--(9.2)]{CLMPZ}, the big-Witt ghost map is the negative
logarithmic derivative; see also \cite[Lemma~9.7]{CLMPZ}.  Hence the corollary
also implies
that there is no
\begin{equation}\label{eq:no-gw-power-series}
 Z(t)\in1+t\,GW(\mathbb F_q)[[t]]
 \quad\text{with}\quad
 \frac{Z'(t)}{Z(t)}=\sum_{m\geq1}B_m\,t^{m-1}.
\end{equation}
Indeed, \(Z^{-1}\) would correspond to a big Witt vector with ghost
coordinates \((B_m)_m\).

A related obstruction appears in \cite[Remark~6.3]{BHSVW}: when \(k^\times\)
contains a nonsquare, \(GW(k)\) admits no compatible power structure.

After rationalization one may form
\[
 Z_{\mathbb Q}(t)=\exp\!\left(\sum_{m\geq1}
 \frac{\operatorname{Tr}_{\SH(k)}(\varphi^m)}{m}\,t^m\right)
 \in1+t\,GW(k)_{\mathbb Q}[[t]].
\]
Thus the obstruction above is integral.

\subsection{Relation to the semilinear toric class}

\begin{proposition}
\label{prop:toric-realizes-bhsvw}
Let \(X_\Sigma\) be a smooth proper toric variety with the finite toric
endomorphism \(f_A\) of
Theorem~\ref{thm:cones-fixed-by-iterates-trace}, and let
\(e\colon\Spec k\to BT\) classify the trivial torsor.  Choose the source
category \(\cC\) as in that theorem and assume that every object of
\(e^*\cC\) is dualizable.  Choose a small thick rigid symmetric-monoidal
subcategory \(\mathcal C_k\subset\SH(k)\) containing the full image
\(e^*\cC\), together with a trace-evaluation datum on its inclusion in
\(\SH(k)\).  The equivalence
\(e^*F_A\simeq e^*\) gives a morphism of endofunctor pairs
\[
 (e^*,\theta)\colon(\cC,F_A)\longrightarrow
 (\mathcal C_k,\operatorname{id}).
\]
Write \((e^*,\theta)_*\) for the induced map on
intrinsic restriction targets.  It carries the toric class to
\[
 (e^*,\theta)_*\bigl(\mathcal Z^{\mathrm{res,int}}_{\Sigma,A}\bigr)
 \in\pi_0R_{\mathbb N_{>0}}\bigl(\mathbf E(\mathcal C_k,
 \operatorname{id})\bigr),
\]
the \(\TR\)-trace class attached to the laced object
\((\Mcompact_k(X_\Sigma),f_{A,c}^*)\).  The realized ghosts of this class recover the
BHSVW logarithmic zeta function of \((X_\Sigma,f_A)\):
\begin{equation}\label{eq:toric-ghosts}
 \operatorname{gh}^{GW}_m
 \bigl((e^*,\theta)_*\mathcal Z^{\mathrm{res,int}}_{\Sigma,A}\bigr)
 =\operatorname{Tr}_{\SH(k)}(f_A^m),
 \qquad m\geq1.
\end{equation}
\end{proposition}

\begin{proof}
By exceptional base change, the functor induced by \((e^*,\theta)\) sends
\(\Xi_{\Sigma,A}\) to \((\Mcompact_k(X_\Sigma),f_{A,c}^*)\).  Hence the induced
map on \(K_0\) sends \([\Xi_{\Sigma,A}]\) to the class of this object.  The
intrinsic ghost formula and categorical realization show that its \(m\)-th
realized ghost is \(\operatorname{Tr}((f_{A,c}^*)^m)\).
Motivic Atiyah duality \cite[Section~3, equations~(3.1)--(3.4)]{HoyoisTrace}
identifies this laced object with the dual of
\((\Sigma^\infty_+X_\Sigma,\Sigma^\infty_+f_A)\); trace is invariant under
duality, proving \eqref{eq:toric-ghosts}.
\end{proof}

\begin{remark}
In characteristic zero, the orbit-generated category \(\cC\) satisfies the
dualizability hypothesis.  Smooth pullback expresses the compactly supported
orbit motives as Thom twists of suspension spectra
\cite[Theorem~6.18(2)]{HoyoisSix}, so the finite orbit filtration has compact
stages.  Pullback along \(e\) preserves compact objects
\cite[Theorem~4.10(iii)]{KhanRavi}, and compact objects of \(\SH(k)\) are
strongly dualizable \cite[Theorem~3.2.1]{ElmantoKhan}.  Over a field of
exponential characteristic \(p>0\), the corresponding dualizability statement
holds after passage to \(\SH(-)[1/p]\).
\end{remark}

\appendix
\section{Model replacements and coherence verifications}
\label{app:technical-coherence}

This appendix gives the model-categorical results used in the descent and
Frobenius arguments.  It treats pseudonatural localization, admissible bar
models, change of enrichment, and the compatibility of ordered block
composition with the CLMPZ restriction maps.

\subsection{Rectification of pseudonatural comparisons}
\label{app:pseudonatural-rectification}

\begin{lemma}
\label{lem:pseudonatural-localization}
Let \(I\) be an ordinary category, and let \(\mathcal F,\mathcal G\) be normal
pseudofunctors from \(I\) to relative categories, stable model categories, or
spectral categories.  Assume that their transition functors preserve the
specified weak equivalences.  A pseudonatural transformation
\(\eta\colon\mathcal F\Rightarrow\mathcal G\) whose components preserve
weak equivalences induces a natural transformation between the associated
\(I\)-diagrams of \(\infty\)-categories.  For spectral categories, the
associated diagrams are obtained functorially from their relative categories
of cofibrant perfect modules.  If every component of \(\eta\) becomes an
equivalence after the specified DK localization, the induced transformation
is a natural equivalence.  These constructions are compatible with composition
of pseudonatural transformations and with postcomposition by functors defined
on the localized diagrams.
\end{lemma}

\begin{proof}
For stable model categories, pass to the underlying relative categories.  For
spectral categories, first pass functorially to cofibrant perfect modules with
their weak equivalences; a DK equivalence induces an equivalence after this
localization.  It therefore suffices to treat relative categories.  The
cocartesian straightening statement below is the dual of the cartesian form
of \cite[Theorem~3.2.0.1]{LurieHTT}.

The Grothendieck constructions of the two normal pseudofunctors are
cocartesian fibrations over \(NI\); the compositors are precisely the
cocartesian composition data.  A pseudonatural transformation gives a
functor between the total categories over \(NI\), and because its naturality
constraints are invertible, this functor carries cocartesian edges to
cocartesian edges.
Mark \(NI\) by its equivalences, and mark each total category by the
fiberwise weak equivalences together with the cocartesian lifts of the marked
base edges.  Fiberwise Dwyer--Kan localization then yields a map of
cocartesian fibrations over \(NI\) whose fibers are the localized fibers;
this is the localization of families of \(\infty\)-categories of
\cite[Proposition~2.1.4]{HinichDK}.  Its hypotheses hold by the assumption on
the transition functors; pseudofunctorial inverses over base equivalences
become equivalences after localization.  Straightening--unstraightening
\cite[Theorem~3.2.0.1]{LurieHTT} then produces the asserted natural
transformation of \(I\)-diagrams.  A fiberwise DK equivalence becomes a
fiberwise equivalence, hence a natural equivalence, since equivalences in the
functor \(\infty\)-category are detected objectwise.  Functoriality of the
Grothendieck construction gives compatibility with composition, and
postcomposition is formal in the resulting functor \(\infty\)-categories.
\end{proof}

\begin{lemma}
\label{lem:pseudonatural-pair-morita-localization}
Let \(I\) be an ordinary category.  A normal pseudofunctor from \(I\) to
spectral category--bimodule pairs determines, by derived perfect-module
localization and derived transport of the bimodules, an \(I\)-diagram in
\(\operatorname{Pair}^{\operatorname{perf}}_\infty\).  A pseudonatural
transformation of such diagrams determines a natural transformation after
localization.  If every component is a coefficient Morita equivalence, the
resulting transformation is a natural equivalence.  The construction is
compatible with composition and with postcomposition by coefficient
\(\THH\).  If all spectral pairs involved are admissible, it is also
compatible with postcomposition by the restriction-system functor
\(\mathbf E\) of
Proposition~\ref{prop:admissible-pair-restriction-system}.  This functor
descends across coefficient Morita localization by
Theorem~\ref{thm:admissible-pair-morita-invariance}.
\end{lemma}

\begin{proof}
By \cite[Definition~2.2]{HNS}, the assignment
\(\cC\mapsto\operatorname{Bimod}(\cC)\) classifies a cartesian and
cocartesian fibration whose total \(\infty\)-category is the category of
laced categories.  We use this total category as
\(\operatorname{Pair}^{\operatorname{perf}}_\infty\).  Perfect-module
localization sends a spectral pair to an object of this total category: the
category is sent to \(\operatorname{Perf}(A)\), and the coefficient is derived
left Kan extended from representables.  The invertible compositors of a
normal pseudofunctor supply the cocartesian composition data, while the
constraints of a pseudonatural transformation give a map between the
resulting cocartesian fibrations over \(NI\).  Fiberwise Dwyer--Kan
localization and straightening, exactly as in
Lemma~\ref{lem:pseudonatural-localization}, therefore produce the
asserted diagram and natural transformation in
\(\operatorname{Pair}^{\operatorname{perf}}_\infty\).

By Definition~\ref{def:coefficient-morita-equivalence}, a component
is a coefficient Morita equivalence precisely when the corresponding edge in
this total \(\infty\)-category is an equivalence: its underlying functor is an
equivalence of idempotent-complete stable \(\infty\)-categories, and its
cocartesian coefficient transport is an equivalence.  Equivalences in a
functor \(\infty\)-category are detected objectwise.  Compatibility with
composition and with coefficient \(\THH\) follows from functoriality in the
localized total category.  On admissible point-set models,
Proposition~\ref{prop:admissible-pair-restriction-system} defines
\(\mathbf E\), and
Theorem~\ref{thm:admissible-pair-morita-invariance} shows that it inverts
coefficient Morita equivalences.  It therefore factors through their
localization, proving the final assertion.
\end{proof}

\subsection{Admissible spectral pairs}
\label{app:pair-replacement-details}

The intrinsic argument of
Theorem~\ref{thm:coefficient-thh-morita-invariance} applies to
stable categories and categorical bimodules.  To compare it with the CLMPZ
cyclic models, we isolate the cofibrancy hypotheses under which their
ordinary bars represent the corresponding derived relative tensor products.

The admissible spectral pairs and their morphisms were introduced in
\S\ref{subsec:twistings-graph-coefficients}.  For \(s>1\), we parenthesize
the ordinary power recursively by
\[
 M^{\odot_A s}=B(M^{\odot_A(s-1)},A,M).
\]
At every stage the newly attached factor \(M\) is pointwise cofibrant, so
\cite[Definition~4.7]{CLMPZ} identifies this ordinary two-sided bar with the
corresponding derived coend.  Thus ordinary bar composition on an admissible
pair presents relative tensor product in the Morita \(\infty\)-category.  The
pointed Dennis-trace model also uses the Reedy-cofibrancy statement of
Lemma~\ref{lem:compatible-enhancement-bar-models}.

\begin{lemma}
\label{lem:pointwise-cofibrant-graph-multibar}
Let \(A\) be a pointwise-cofibrant spectral category and let \(M\) be a
pointwise-cofibrant \(A\)-bimodule.  In the pointed case, assume in addition
the pointed admissibility clauses of
\S\ref{subsec:twistings-graph-coefficients}.  Then \((A,M)\) is an
admissible spectral pair.  For every \(s\geq1\), the recursively
parenthesized ordinary bar power represents the derived coefficient power:
\begin{equation}\label{eq:admissible-ordinary-derived-power}
 M^{\odot_A s}\xrightarrow{\ \simeq\ }
 M^{\mathbb L\odot_A s}.
\end{equation}
After forgetting the \(C_s\)-action, the CLMPZ unwinding equivalence is
natural in \((A,M)\) and gives
\begin{equation}\label{eq:admissible-unwinding-derived}
 V_s\THH^{(s)}(A;M)
 \simeq \THH(A;M^{\odot_A s})
 \simeq \THH(A;M^{\mathbb L\odot_A s}).
\end{equation}
Moreover, \(\THH^{(s)}(A;M)\) is cofibrant as an orthogonal
\(C_s\)-spectrum, and the norm diagonal supplies the restriction
equivalences.  In particular, if \(f\colon A\to A\) is a spectral
endofunctor of a pointwise-cofibrant spectral category, then
\((A,A(f-,-))\), regarded as an unpointed pair, is admissible and all the
preceding conclusions apply to it.
\end{lemma}

\begin{proof}
For \(s=1\) there is nothing to prove.  Suppose that the assertion is known
for \(s-1\).  Since \(A\) and the newly attached copy of \(M\) are pointwise
cofibrant, \cite[Definition~4.7]{CLMPZ} identifies
\(B(M^{\odot_A(s-1)},A,M)\) with the derived relative tensor product of its
two bimodule factors.  The inductive hypothesis and associativity of relative
tensor product in the Morita \(\infty\)-category, modeled by the associator
of \cite[Definition~4.10]{CLMPZ}, give
\eqref{eq:admissible-ordinary-derived-power}.

Proposition~7.6 of \cite{CLMPZ} gives the first equivalence in
\eqref{eq:admissible-unwinding-derived}; the second is induced by
\eqref{eq:admissible-ordinary-derived-power}.  Proposition~7.4 of
\cite{CLMPZ} gives the cofibrancy and norm-diagonal assertions for the
successive derived realizations; compare \cite[Remark~6.18]{CLMPZ}.  Finally,
the values of \(A(f-,-)\) are mapping
spectra of \(A\), so the graph bimodule is pointwise cofibrant.
\end{proof}

\begin{proposition}
\label{prop:admissible-pair-restriction-system}
The assignment
\[
 (A,M)\longmapsto\{\THH^{(n)}(A;M)\}_{n\geq1}
\]
defines a functor
\[
 \mathsf{Pair}^{\operatorname{adm}}
 \longrightarrow\operatorname{RSys}^{g}.
\]
\end{proposition}

\begin{proof}
The pointwise-cofibrancy conditions in
\S\ref{subsec:twistings-graph-coefficients} are precisely the input
hypotheses of \cite[Definition~7.3 and Example~8.7]{CLMPZ}.
Proposition~7.4 of \cite{CLMPZ} makes every cyclic level a cofibrant
orthogonal equivariant spectrum and gives natural isomorphisms
\[
 \Phi^{C_r}\THH^{(rs)}(A;M)\cong\THH^{(s)}(A;M)
 \qquad(r,s\geq1).
\]
The raw geometric fixed points therefore compute the left-derived ones, and
Example~8.7 assembles the levels into a genuine restriction system.  Maps of
admissible pairs act on every mapping-spectrum and
coefficient factor in the cyclic bar; naturality of the norm diagonal makes
these level maps compatible with all divisibility arrows.  Functorial genuine
suspension and prolongation give the asserted functor.
\end{proof}

\subsection{Choice of spectral enrichment}
\label{app:change-enrichment-details}

\begin{lemma}
\label{lem:orthogonal-morita-model}
Let \(\operatorname{SpCat}^{\Sigma}\) be the category of small categories
enriched in symmetric spectra, equipped with the DK model structure used in
\cite[Theorem~2.2]{BGT}.  There is an ordinary functor
\begin{equation}\label{eq:orthogonal-realization-functor}
 \mathbb O\colon
 \operatorname{SpCat}^{\Sigma}
 \longrightarrow
 \operatorname{SpCat}^{\mathcal O}
\end{equation}
with the following properties.
\begin{enumerate}
\item There is a natural equivalence
\begin{equation}\label{eq:orthogonal-realization-perfect}
 \operatorname{Perf}(\mathbb O A)
 \simeq \operatorname{Perf}(A)
\end{equation}
after Morita localization.  In particular, \(\mathbb O\) carries Morita
equivalences to Morita equivalences.
\item The functor \(\mathbb O\) carries strict diagrams, strict commutative
squares, and literal power identities to diagrams, squares, and identities
of orthogonal spectral categories holding on the nose.
\item For every small ordinary category \(I\), if
\(W_{\mathrm{Mor}}\) denotes the objectwise Morita equivalences, then
\begin{equation}\label{eq:symmetric-presentation-diagrams}
 N\operatorname{Fun}(I,\operatorname{SpCat}^{\Sigma})
 [W_{\mathrm{Mor}}^{-1}]
 \simeq
 \operatorname{Fun}
 \bigl(NI,\operatorname{Cat}^{\operatorname{perf}}_\infty\bigr).
\end{equation}
\end{enumerate}
Thus a rectified symmetric-spectral diagram determines a strict
orthogonal-spectral diagram to which the CLMPZ construction applies.
\end{lemma}

\begin{proof}
Choose a functorial cofibrant replacement
\(Q_{\Sigma}\to\operatorname{id}\) in the DK model structure on symmetric
spectral categories.  Let
\[
 \mathbb P\colon\operatorname{Sp}^{\Sigma}\longrightarrow
 \operatorname{Sp}^{\mathcal O}
\]
be the composite of geometric realization and prolongation in the left-hand
part of the comparison diagram of
\cite[Section~7]{SchwedeShipleyMonoidal}, and put
\(\mathbb O A=\mathbb P(Q_{\Sigma}A)\).  The functor \(\mathbb P\) is a
strong symmetric-monoidal left Quillen functor, so it acts on enriched
categories and enriched functors without introducing coherence maps.  This
proves part~(2).

The category \(Q_{\Sigma}A\) is cofibrant, and
\cite[Theorem~6.5 and the proof of Corollary~1.2]{SchwedeShipleyMonoidal}
give a Quillen equivalence
\[
 \operatorname{Mod}_{Q_{\Sigma}A}
 \simeq_Q
 \operatorname{Mod}_{\mathbb P Q_{\Sigma}A}.
\]
It sends representable modules to representable modules.  Hence it restricts
to the equivalence of perfect objects in
\eqref{eq:orthogonal-realization-perfect}.  The construction is
natural in spectral functors.  Together with the DK equivalence
\(Q_{\Sigma}A\to A\), it follows that a Morita equivalence
\(A\to B\) induces a Morita equivalence
\(\mathbb O A\to\mathbb O B\).  This proves part~(1).

For part~(3), first give the strict diagram category the projective model
structure associated with the DK model structure.  Its underlying
\(\infty\)-category is
\[
 \operatorname{Fun}
 \bigl(NI,(\operatorname{SpCat}^{\Sigma})_\infty\bigr)
\]
by \cite[Proposition~1.3.4.25]{LurieHA}.  The Morita localization
\((\operatorname{SpCat}^{\Sigma})_\infty\to
\operatorname{Cat}^{\operatorname{perf}}_\infty\) is the accessible
localization of \cite[Theorem~4.23]{BGT}.  Postcomposition with this
localization is again a localization of functor \(\infty\)-categories, and a
natural transformation is inverted precisely when all its components are
Morita equivalences.  Since every DK equivalence is a Morita equivalence,
localizing the ordinary strict diagram category directly at
\(W_{\mathrm{Mor}}\) gives
\eqref{eq:symmetric-presentation-diagrams}.
\end{proof}

\subsection{Reduced bar models}
\label{app:reduced-bar-models-details}

The CLMPZ point-set construction uses a chosen zero object.  The following
reduction replaces all other zero objects by that choice, functorially.

\begin{lemma}
\label{lem:zero-reduction}
Let \(A\) be a spectral Waldhausen category with chosen zero object \(0_A\).
There is a spectral Waldhausen category \(A^\circ\) whose objects are
\[
 \{0_A\}\amalg
 \{a\in\operatorname{ob}A\mid \operatorname{id}_a\ne0_{a,a}\}.
\]
It is the full spectral subcategory of \(A\) on these objects.  In
particular,
\[
 A^\circ(0_A,a)=0=A^\circ(a,0_A)
\]
as literal equalities of spectra, and \(0_A\) is the only zero object of the
Waldhausen base of \(A^\circ\).

Let \(\mathsf{SpWaldCat}^{\mathrm z}\) denote the ordinary category whose
objects are spectral Waldhausen categories and whose morphisms preserve
cofibrations, weak equivalences, pushouts along cofibrations, and the class
of zero objects, but are not required to carry the chosen zero object to the
chosen zero object on the nose.  A morphism \(F\colon A\to B\) in this
category induces a strictly
zero-preserving exact spectral functor
\[
 F^\circ\colon A^\circ\longrightarrow B^\circ
\]
by sending \(a\) to \(F(a)\) when \(F(a)\) is not a zero object and to
\(0_B\) otherwise.  Let \(\mathsf{SpWaldCat}^{\circ}\) be the category of
spectral Waldhausen categories with a unique chosen zero object and exact
spectral functors preserving it on the nose.  The assignment
\(A\mapsto A^\circ\) is an ordinary
functor
\[
 (-)^\circ\colon
 \mathsf{SpWaldCat}^{\mathrm z}
 \longrightarrow\mathsf{SpWaldCat}^{\circ}
\]
and
\[
 (GF)^\circ=G^\circ F^\circ
\]
on the nose.  The inclusions \(i_A\colon A^\circ\to A\) are DK and Morita
equivalences and form a pseudonatural transformation; hence they determine a
natural equivalence after localization.
\end{lemma}

\begin{proof}
The full Waldhausen subcategory on the displayed objects inherits
cofibrations and weak equivalences.  If a pushout in \(A\) has nonzero
pushout object, it belongs to \(A^\circ\).  If its pushout object is a zero
object, transport the pushout cocone along its unique isomorphism to
\(0_A\).  The resulting square is a pushout in the full subcategory.
This proves the pushout axiom; the remaining Waldhausen axioms are inherited
in the same way.  The asserted literal equalities involving \(0_A\) are part
of Definition~3.12(i) of \cite{CLMPZ}.

The collapse is enriched for the following reason.  If
\(z\) is any zero object of the Waldhausen base, the unique maps
\(z\rightleftarrows0_A\) are inverse isomorphisms in that base.  The base
enrichment sends them to enriched arrows, and pre- and postcomposition give
mutually inverse isomorphisms of spectra
\[
 A(z,a)\cong A(0_A,a)=0,
 \qquad
 A(a,z)\cong A(a,0_A)=0.
\]
Thus every mapping spectrum to or from any zero object is a zero object of
spectra, although it need not be the chosen literal zero spectrum occurring
at \(0_A\).

On objects take \(F^\circ\) as in the statement.  On a mapping spectrum
use the map induced by \(F\) if neither image is zero, and otherwise use the
unique map to the zero mapping spectrum of \(B^\circ\).  If the image of an
endpoint is zero, the enriched-functor square commutes uniquely.  If the
endpoint images are nonzero but a composable pair passes through an object
\(b\) for which \(F(b)\) is zero, the
corresponding composite for \(F\) factors through \(B(Fb,-)\) and
\(B(-,Fb)\), which are zero spectra by the preceding paragraph.  It is
therefore the zero map, as required after collapsing \(F(b)\) to
\(0_B\).  When none of the three images is zero, compatibility is that of
\(F\).  The unit at a collapsed object is the unique map
\(\mathbb S\to0\), so identities are also preserved.  Hence \(F^\circ\) is
an enriched functor.

The functor \(F\) sends zero objects to zero objects.  Applying \(F\) to a
pushout square in \(A^\circ\) gives a pushout square in \(B\); replacing any
zero vertices by \(0_B\) transports that square along unique isomorphisms
and therefore leaves it a pushout.  Cofibrations and weak equivalences are
unchanged between nonzero objects and are preserved under these
isomorphisms.  Thus \(F^\circ\) is exact.

For an object \(a\), if \(F(a)\) is zero then both
\(G^\circ F^\circ(a)\) and \((GF)^\circ(a)\) are \(0_C\); if \(F(a)\) is
nonzero, both are \(G(F(a))\) when this object is nonzero and \(0_C\)
otherwise.  Thus the two composites agree on objects.  On a mapping spectrum,
if either final image is zero both maps are the unique map to a literal zero
spectrum; otherwise both are the map induced by \(GF\).  Consequently
\((GF)^\circ=G^\circ F^\circ\) on the nose.

The inclusion \(i_A\) is spectrally fully faithful.  Every omitted object is
a zero object and is therefore actually isomorphic to \(0_A\), so \(i_A\)
is a DK equivalence.  For \(F\colon A\to B\), the comparison between
\(i_BF^\circ\) and \(Fi_A\) is the identity when \(F(a)\) is nonzero and the
unique isomorphism \(0_B\cong F(a)\) otherwise.  Enriched naturality follows
from the same zero-spectrum calculation used above.  The uniqueness of maps
between zero objects supplies the unit and composition coherences, so these
comparisons form a pseudonatural transformation.  The last assertion follows
from Lemma~\ref{lem:pseudonatural-localization}.
\end{proof}

\begin{lemma}
\label{lem:compatible-enhancement-bar-models}
Let \(A\) have a unique chosen zero object \(0\), be pointwise cofibrant, and
assume that
\[
 \mathbb S\longrightarrow A(a,a)
\]
is a cofibration for every \(a\ne0\), and let \(M\) be a
pointwise-cofibrant reduced \(A\)-bimodule.  Then every ordinary iterated bar
power \(M^{\odot_A n}\) is pointwise cofibrant and reduced, and
\begin{equation}\label{eq:ordinary-derived-coefficient-power}
 M^{\odot_A n}
 \xrightarrow{\ \simeq\ }
 M^{\mathbb L\odot_A n}
\end{equation}
is an equivalence.  The cyclic and multicyclic bars are Reedy cofibrant, so
ordinary realization computes derived realization.
\end{lemma}

\begin{proof}
A summand indexed by a tuple containing \(0\) is
literally zero.  On all other summands, degeneracies insert unit cofibrations.
The latching maps are therefore coproducts of iterated pushout products of
cofibrations, smashed with cofibrant mapping and coefficient spectra.  Hence
the bars are Reedy cofibrant.  Induction on \(n\) gives
\eqref{eq:ordinary-derived-coefficient-power}.
\end{proof}

\begin{theorem}
\label{thm:reduced-bar-replacement}
Fix an ambient universe.  Let
\[
 Q^{\mathrm W}\colon\mathsf{SpWaldCat}^{\circ}\longrightarrow
 \mathsf{SpWaldCat}^{\circ},
 \qquad
 q\colon Q^{\mathrm W}\Longrightarrow\operatorname{id},
\]
be the functorial cofibrant replacement of
\cite[Remark~3.16]{CLMPZ}, restricted to spectral Waldhausen categories with
a unique chosen zero object and exact functors preserving it on the nose.
We write \(Q_{\mathrm{red}}=Q^{\mathrm W}\).  For every
\(A\in\mathsf{SpWaldCat}^{\circ}\), the map
\begin{equation}\label{eq:reduced-bar-replacement}
 q_A\colon Q_{\mathrm{red}}A\longrightarrow A
\end{equation}
has the following properties:
\begin{enumerate}
\item the object set and Waldhausen base are unchanged, and the chosen zero
      object remains the unique zero object;
\item \(q_A\) is a trivial fibration on every mapping spectrum, hence a DK
      and Morita equivalence;
\item the mapping spectra of \(Q_{\mathrm{red}}A\) are cofibrant, the unit
      maps at all nonzero objects are cofibrations, and mapping spectra to or
      from the chosen zero object are literally zero;
\item for a strictly zero-preserving exact endofunctor \(H\), the graph pair
\[
 \bigl(Q_{\mathrm{red}}A,\,
       Q_{\mathrm{red}}A(Q_{\mathrm{red}}H-,-)\bigr)
\]
      is an admissible spectral pair in the sense of
      \S\ref{subsec:twistings-graph-coefficients};
\item applying \(Q_{\mathrm{red}}\) to a strict diagram of zero-reduced
      twistings gives another strict diagram of twistings; literal
      endofunctor intertwinings and power identities are preserved.
\end{enumerate}
The augmentation of each graph pair is a coefficient Morita
equivalence.  If the construction is performed in two universes, the two
outputs become naturally equivalent after DK and coefficient-Morita
localization by comparison with their common enlarged input.
\end{theorem}

\begin{proof}
Functoriality and the natural augmentation are
\cite[Remark~3.16]{CLMPZ}.  On the nonzero object set, choose the functorial
factorization
\[
 \Sigma^\infty A_0^\times
 \longrightarrow (Q^{\mathrm W}A)^\times
 \xrightarrow{\ q_A\ }A^\times,
\]
where the first arrow is a cofibration and the second is a trivial fibration
in the fixed-object-set model structure.  Here \(A^\times\) and
\(A_0^\times\) denote the restrictions of the spectral enrichment and its
Waldhausen base to the nonzero objects.  The mapping spectra of
\(\Sigma^\infty A_0^\times\) are wedges of sphere spectra and are
cofibrant.  Hence \((Q^{\mathrm W}A)^\times\) is cofibrant under this base
enrichment.  By \cite[Proposition~6.3]{SchwedeShipleyMonoidal}, its mapping
spectra are cofibrant.  At a nonzero object \(a\), the unit is the composite
\[
 \mathbb S\longrightarrow
 \Sigma^\infty A_0^\times(a,a)\longrightarrow
 (Q^{\mathrm W}A)^\times(a,a).
\]
The first arrow includes the identity summand and the second is a pointwise
cofibration by the same proposition, so the unit is a cofibration.  The
second factor is, by construction, a trivial fibration on every mapping
spectrum.  Adjoining the chosen zero object with zero mapping spectra proves
the first three assertions.

The graph coefficient is made from mapping spectra of
\(Q_{\mathrm{red}}A\), hence is pointwise cofibrant.  It is reduced because
\(H(0)=0\) on the nose.  Therefore
Lemma~\ref{lem:compatible-enhancement-bar-models} proves the ordinary
bar and Reedy-cofibrancy assertions and admissibility.  The augmentation is
a coefficient Morita equivalence because both its category map and its
graph-coefficient map are pointwise weak equivalences.  Finally, an ordinary
functor carries a strict diagram to a strict diagram; the identities in
part~(5) follow from preservation of composition.
For an enlargement of universes, naturality of the augmentation gives the
asserted comparison after localization.
\end{proof}

\begin{proposition}
\label{prop:replacement-comparison-diagrams}
Let \(v\colon z\to z'\) be a strict morphism of zero-reduced twistings.
Assume that its underlying spectral functor is a DK equivalence, its graph
pair is a coefficient Morita equivalence, and it induces equivalences on the
cyclic source \(K\)-theory spectra compatibly with the fixed-cycle maps.
Then
\[
 Q_{\mathrm{red}}v\colon Q_{\mathrm{red}}z\longrightarrow
 Q_{\mathrm{red}}z'
\]
has the same properties and induces an equivalence of CLMPZ trace arrows in
\(\operatorname{Fun}(\Delta^1,\operatorname{RSys}^g_\infty)\).  The
conclusion is natural for strict diagrams of such maps.
\end{proposition}

\begin{proof}
In the natural square formed by \(Q_{\mathrm{red}}v\), \(v\), and the
augmentations, the horizontal maps are DK and coefficient Morita
equivalences.  The lower vertical map has these properties by hypothesis, so
two-out-of-three gives them for \(Q_{\mathrm{red}}v\).  On the source, the
horizontal maps are identities: Theorem~\ref{thm:reduced-bar-replacement}
leaves the Waldhausen bases, endofunctors, and fixed-cycle maps unchanged,
and genuine suspension and prolongation preserve these identities.  The
lower vertical map is an equivalence by hypothesis.  On the target,
Theorem~\ref{thm:admissible-pair-morita-invariance} applies at every
cyclic level and is compatible with all restriction arrows.  Finally,
Lemma~\ref{lem:clmpz-strict-functoriality} makes these maps a morphism
of trace arrows, which is an equivalence because both vertices are.  Since
\(Q_{\mathrm{red}}\) and the CLMPZ construction are functors on strict
morphisms, this argument is natural for strict diagrams.
\end{proof}

\begin{remark}
\label{rem:standing-replacement-convention}
In Sections~\ref{sec:morita-descent}
and~\ref{sec:iteration-frobenius}, the notation
\(Q_{\mathrm{red}}\) means the functorial replacement of the zero-reduced
input twisting.  The \(w_\bullet S_\bullet\)-objects, graph coefficients,
additivity maps, and restriction maps are then those of the standard CLMPZ
construction.  This convention makes strict maps of input twistings induce
maps of the associated CLMPZ diagrams.
\end{remark}

\subsection{Compatibility with iteration}
\label{app:power-comparison-details}

We prove Theorem~\ref{lem:clmpz-graph-power-comparison} on the CLMPZ
\(\Sigma_\Delta\)-diagram, whose objects \((q;k_1,\ldots,k_q)\) index
iterated \(S\)-constructions; \(w_{k_0}\) is a separate simplicial direction.

\begin{proof}[Proof of Theorem~\ref{lem:clmpz-graph-power-comparison}]
On morphisms define
\[
 \operatorname{Comp}_m(u_i)=u_0.
\]
Iterating the cyclic lacing equations proves functoriality.  Cofibrations,
weak equivalences, and pushouts are coordinatewise, and exactness of \(f\)
shows that ordered composition preserves them.

Apply \(\operatorname{Comp}_m\) levelwise to
\[
 w_{k_0}S^{(q)}_{k_1,\ldots,k_q}A.
\]
Since the \(w\)- and \(S\)-operators, permutations, insertions, and structural
isomorphisms act pointwise, these functors assemble to a map of the complete
left \(\Sigma_\Delta\)-diagrams.  If \(u\colon(A,f)\to(B,g)\) is strict, then
\(uf=gu\) implies
\[
 u\bigl(\alpha_{m-1}f(\alpha_{m-2})\cdots f^{m-1}(\alpha_0)\bigr)
 =u(\alpha_{m-1})g(u\alpha_{m-2})\cdots
   g^{m-1}(u\alpha_0),
\]
so ordered composition is natural in strict morphisms.  The equivalences
\(\omega\) of
Definition~\ref{def:marked-unwinding} and ordered bar composition give
the corresponding map of the right diagrams.  On zero-simplices it sends
\[
 \alpha_0\wedge\cdots\wedge\alpha_{m-1}
\]
to the ordered composite displayed in the statement.

The backwards additivity maps are wedges of subquotient inclusions.  Since the
endofunctor commutes strictly with those inclusions, both routes apply the same
spectral functor to every bar factor and then compose the same ordered arrows.
Thus we obtain a map of CLMPZ trace zigzags, preserved by realization,
genuine suspension, and derived prolongation.  Inverting the additivity
equivalence in the functor category gives
\eqref{eq:clmpz-graph-power-comparison}.  The duplication formula is
immediate.
\end{proof}

\subsection{Ordered block composition on the full CLMPZ diagram}
\label{app:block-comparison-details}

\begin{lemma}
\label{lem:ordered-multibar-fubini}
Let \(A\) be a pointwise-cofibrant spectral category and let
\(M_1,\ldots,M_r\) be pointwise-cofibrant \(A\)-bimodules.  Write
\(B^{(r)}(M_1,A,\ldots,A,M_r)\) for the fully expanded ordered multibar, with
one simplicial direction for each intervening copy of \(A\).  Every binary
parenthesization \(\mathfrak p\) of
\(M_1\odot_A\cdots\odot_A M_r\) gives an iterated two-sided bar
\(B_{\mathfrak p}\).  Fubini for the realization coends gives a natural
isomorphism
\begin{equation}\label{eq:ordered-multibar-fubini}
 \theta_{\mathfrak p}\colon
 \bigl\|B^{(r)}(M_1,A,\ldots,A,M_r)\bigr\|
 \xrightarrow{\ \cong\ } B_{\mathfrak p}.
\end{equation}
For two parenthesizations, the comparison
\(\theta_{\mathfrak q}\theta_{\mathfrak p}^{-1}\) represents the
relative-bar associator of \cite[Definition~4.10]{CLMPZ}.  These comparisons
are natural in \(A\), in the \(M_i\), and in degreewise maps of expanded
multibars.  Their composites are compatible with refinement of
parenthesizations; in particular, the two boundary composites for four
factors agree by the pentagon.  On a \(d\)-fold repeated word, the same
naturality makes the comparisons commute with the norm diagonal and with
the induced map on geometric \(C_d\)-fixed points.
\end{lemma}

\begin{proof}
At the point-set level, the source of
\eqref{eq:ordered-multibar-fubini} is the multiple coend in orthogonal
spectra
\[
 \int^{[q_1],\ldots,[q_{r-1}]}
 B^{(r)}_{q_1,\ldots,q_{r-1}}
 \wedge(\Delta^{q_1})_+\wedge\cdots\wedge(\Delta^{q_{r-1}})_+ .
\]
The iterated bar attached to any parenthesization evaluates this same coend
in the order prescribed by that parenthesization.  Associativity of smash
and Fubini for coends therefore give the displayed natural isomorphism.  It
is compatible with the object-indexed coproducts and is natural in every
additional CLMPZ simplicial direction, hence also after their realizations.
The coherence of these canonical coend isomorphisms identifies every
composite associated with a refinement with the direct Fubini isomorphism.
For three factors this is the relative-bar associator, and for four factors
the resulting boundary diagram is its pentagon.

Let \(\mathfrak p_0\) be the recursively left-associated parenthesization.
At each stage of \(B_{\mathfrak p_0}\), the newly attached factor is one of
the pointwise-cofibrant bimodules \(M_i\).  The argument of
Lemma~\ref{lem:pointwise-cofibrant-graph-multibar}, using
\cite[Definition~4.7]{CLMPZ} at each stage, therefore identifies
\(B_{\mathfrak p_0}\) with the corresponding iterated relative tensor
product in the Morita \(\infty\)-category.  For an arbitrary parenthesization
\(\mathfrak p\), the isomorphism
\(\theta_{\mathfrak p}\theta_{\mathfrak p_0}^{-1}\) transports this derived
model to \(B_{\mathfrak p}\).  Thus the Fubini comparisons represent the
canonical associativity equivalences for derived relative tensor products;
for three factors this is the associator of
\cite[Definition~4.10]{CLMPZ}.

Lemma~\ref{lem:norm-diagonal-block-interchange}, applied to the degreewise
Fubini maps, shows that these comparisons commute with all restriction maps.
For four factors, equality of the two Fubini composites gives the pentagon in
the restriction-system \(\infty\)-category.
\end{proof}

\begin{lemma}
\label{lem:admissible-graph-target-block}
Let \(A\) be a pointwise-cofibrant spectral category and let
\(f\colon A\to A\) be a strict spectral endofunctor.  Regard the graph pairs
as unpointed.  For \(m,n\geq1\), ordered bar composition and enriched
co-Yoneda define a natural genuine \(C_n\)-equivalence
\begin{equation}\label{eq:admissible-graph-target-block}
 b_{m,n}\colon
 \operatorname{Res}^{C_{mn}}_{C_n}\THH^{(mn)}(A;X_f)
 \xrightarrow{\ \simeq\ }
 \THH^{(n)}(A;X_{f^m}).
\end{equation}
As \(n\) varies, these maps commute with every restriction arrow and hence
define a natural equivalence of genuine restriction systems
\[
 b_m\colon
 \operatorname{sh}_m\mathbf E(A,f)
 \xrightarrow{\ \simeq\ }
 \mathbf E(A,f^m).
\]
The construction is natural for strict morphisms
\(u\colon(A,f)\to(B,g)\), where \(uf=gu\).  It is unital,
and for \(m,n,k\geq1\) the relative-bar associator gives a natural
comparison
\begin{equation}\label{eq:admissible-target-block-associativity}
 b_{n,k}^{f^m}\circ
 \operatorname{Res}^{C_{nk}}_{C_k}(b_{m,nk}^{f})
 \simeq b_{mn,k}^{f}.
\end{equation}
These comparisons are natural transformations in the functor
\(\infty\)-category of restriction systems and satisfy the associativity
coherence for fourfold products.
\end{lemma}

\begin{proof}
Lemma~\ref{lem:pointwise-cofibrant-graph-multibar} makes the graph
pairs admissible, identifies their ordinary bar powers with the corresponding
derived powers, and supplies the genuine restriction systems.  In the
multibar model of \cite[Proposition~7.6]{CLMPZ}, retain \(n\) outer cyclic
slots and group the \(mn\) graph coefficients into \(n\) consecutive rows of
length \(m\).
The ordinary bar augmentation in each row, followed by the enriched
co-Yoneda map of
\cite[Example~4.6 and the proof of Proposition~7.15]{CLMPZ}, gives
\[
 X_f^{\odot_A m}\longrightarrow X_{f^m}.
\]
Rotation by \(m\) input slots rotates the rows once, so the resulting map is
\(C_n\)-equivariant and gives
\eqref{eq:admissible-graph-target-block}.

Lemma~\ref{lem:norm-diagonal-block-interchange}, applied to the block maps
above, shows that \(b_{m,n}\) commutes with the restriction arrows of
\cite[Example~8.7]{CLMPZ}.

For \(d\mid n\), this norm-diagonal comparison and
\cite[Proposition~7.4]{CLMPZ} identify the geometric \(C_d\)-fixed-point map
of \(b_{m,n}\), with its residual action, with \(b_{m,n/d}\).  After forgetting
that action, Proposition~7.6 of \cite{CLMPZ} identifies this map with
coefficient \(\THH\) applied to the derived co-Yoneda equivalence
\[
 X_f^{\mathbb L\odot_A m}\simeq X_{f^m}.
\]
It is therefore an equivalence.  This holds for every subgroup of \(C_n\),
so \(b_{m,n}\) is a genuine \(C_n\)-equivalence.

Strict naturality follows by applying a morphism \(u\), with \(uf=gu\), to
every mapping and coefficient factor.  For three block lengths, both
composites arise from the same expanded ordered multibar.  By
Lemma~\ref{lem:ordered-multibar-fubini}, the comparison between their
parenthesizations is the relative-bar associator; this gives
\eqref{eq:admissible-target-block-associativity}.  The same lemma
identifies the two fourfold composites and supplies the pentagon.  Its
naturality with respect to the norm-diagonal square shows that these
comparisons commute with every restriction arrow before passage to the
restriction-system \(\infty\)-category.  We choose the one-fold block map to
be the identity, so \(b_1=\operatorname{id}\).
\end{proof}

\begin{lemma}
\label{lem:norm-diagonal-block-interchange}
Work with the reduced bar-cofibrant inputs of
Theorem~\ref{thm:reduced-bar-replacement}.  Fix a multisimplicial
degree of a cyclic-bar or twisted-additivity term, and let
\(\beta\colon Z\to W\) be any degreewise map built, without permuting
factors, from enriched composition, units, reassociation, and ordinary bar
augmentations.  This includes the Fubini comparisons and ordered block maps
used here.  If a word is the
\(d\)-fold repeat of a shorter word, then the square
\begin{equation}\label{eq:norm-diagonal-block-interchange}
\begin{tikzcd}[column sep=large]
 \Phi^{C_d}(Z^{\wedge d})
   \arrow[r,"{\Phi^{C_d}(\beta^{\wedge d})}"]
   \arrow[d,"{D_d^{-1}}"'] &
 \Phi^{C_d}(W^{\wedge d})
   \arrow[d,"{D_d^{-1}}"]\\
 Z \arrow[r,"\beta"'] & W
\end{tikzcd}
\end{equation}
commutes on the point-set level.  Here \(D_d\) is the norm diagonal and
is an isomorphism because every spectrum in the indicated bar summand is
cofibrant.  The same square commutes after taking the wedges indexed by
object strings and by the CLMPZ subquotient inclusions.  It is equivariant
for the residual cyclic action.  Consequently such maps commute with every
CLMPZ restriction-system structure map before prolongation.
\end{lemma}

\begin{proof}
For a cofibrant orthogonal spectrum \(Y\), the norm diagonal
\[
 D_d\colon Y\longrightarrow \Phi^{C_d}(Y^{\wedge d})
\]
is a natural point-set isomorphism by
\cite[Proposition~2.19]{CLMPZ}.  Naturality applied to
\(\beta\colon Z\to W\) is precisely the commutativity of
\eqref{eq:norm-diagonal-block-interchange} after using the inverses
of the vertical isomorphisms.  It applies to every object-indexed wedge
summand.  Geometric fixed points commute with the relevant wedges and
realizations, while reduced bar cofibrancy identifies the displayed raw fixed
points with the derived ones; see
Lemma~\ref{lem:compatible-enhancement-bar-models} and the proof of
\cite[Proposition~7.4]{CLMPZ}.

Because \(\beta\) does not permute factors, the same naturality applies at
each intermediate stage and after every subquotient inclusion.  Rotation of
the shorter word permutes the \(d\) repeated factors
on both sides, so the square is equivariant for the residual cyclic action.
This is the rotation calculation in the proof of
\cite[Proposition~7.4]{CLMPZ}.  Since the CLMPZ restriction maps are assembled
levelwise from these inverse norm diagonals, the squares assemble over all
degrees and divisors.  Their coherence follows from
\cite[Example~8.7 and Theorem~8.8]{CLMPZ}.
\end{proof}

Let \(A\) be a zero-reduced spectral Waldhausen category and let
\(H\colon A\to A\) be a strictly zero-preserving exact spectral
endofunctor.  The following construction gives the block square used in the
proof of Theorem~\ref{thm:restriction-trace-frobenius}.
Apply \(Q_{\mathrm{red}}\) to the input twisting \((A,H)\)
before forming the standard CLMPZ diagrams, as in
Remark~\ref{rem:standing-replacement-convention}.  For \(m,n\geq1\), ordered block
composition defines a \(C_n\)-equivariant exact functor
\begin{equation}\label{eq:general-block-composition-source}
 \operatorname{Comp}_{m\mid n}\colon
 \operatorname{Res}^{C_{mn}}_{C_n}
 \operatorname{End}^{(mn)}({}_HA/A_{\operatorname{id}})
 \longrightarrow
 \operatorname{End}^{(n)}({}_{H^m}A/A_{\operatorname{id}}).
\end{equation}
The functors at the various multisimplicial indices assemble to a morphism of
CLMPZ Dennis-trace diagrams, including the backwards twisted-additivity arrow.
After genuine suspension and derived prolongation, inverting that arrow in
the genuine \(C_n\)-equivariant functor category gives a natural
homotopy-commutative square of genuine \(C_n\)-orthogonal spectra
\begin{equation}\label{eq:general-block-composition-clmpz-square}
\begin{tikzcd}[column sep=large]
 \operatorname{Res}^{C_{mn}}_{C_n}\mathbf K_{mn}(H)
 \arrow[r,"\operatorname{Res}\operatorname{trc}^{(mn)}"]
 \arrow[d,"K(\operatorname{Comp}_{m\mid n})"'] &
 \operatorname{Res}^{C_{mn}}_{C_n}\mathbf E_{mn}(H)
 \arrow[d,"b_{m,n}"]\\
 \mathbf K_n(H^m)
 \arrow[r,"\operatorname{trc}^{(n)}"'] &
 \mathbf E_n(H^m).
\end{tikzcd}
\end{equation}
Here \(\mathbf K_r(H)\) and \(\mathbf E_r(H)\) denote the prolonged CLMPZ
source and target at level \(r\).  For every \(d\mid n\), geometric fixed
points identify the target map with
\begin{equation}\label{eq:general-block-comparison-geometric-fixed}
 \Phi^{C_d}b_{m,n}\simeq b_{m,n/d}
\end{equation}
under the residual actions and the identifications of cyclic groups induced
by the chosen generators.  Hence
\(b_{m,n}\) is a genuine \(C_n\)-equivalence.  For every \(r\mid n\), the
squares
\eqref{eq:general-block-composition-clmpz-square} at levels \(n\) and
\(r\) commute with the CLMPZ restriction arrows.

If \(\Delta_r\) denotes CLMPZ duplication and
\(P_m\colon\operatorname{End}({}_HA/A_{\operatorname{id}})\to
\operatorname{End}({}_{H^m}A/A_{\operatorname{id}})\) is the iteration
functor given by ordered composition, then
\begin{equation}\label{eq:general-block-composition-duplication}
 \operatorname{Comp}_{m\mid n}\Delta_{mn}=\Delta_nP_m
\end{equation}
on the nose.  For \(m,n,k\geq1\), lexicographic rectangular indexing gives
\begin{equation}\label{eq:general-block-composition-associativity}
 \operatorname{Comp}_{n\mid k}\operatorname{Comp}_{m\mid nk}
 =\operatorname{Comp}_{mn\mid k},
\end{equation}
and the associator for the enriched relative bar induces
\begin{equation}\label{eq:general-target-block-associativity}
 b_{n,k}^{H^m}\circ
 \operatorname{Res}^{C_{nk}}_{C_k}(b_{m,nk}^{H})
 \simeq b_{mn,k}^{H},
\end{equation}
with the evident subgroup identifications.  These constructions are natural
for strict morphisms \(uH=Ku\).

Indeed, for a cycle \((\alpha_i:Ha_i\to a_{i+1})_{i\in\mathbb Z/mn}\), the
\(j\)-th output arrow is the ordered composite
\[
 H^ma_{jm}\xrightarrow{H^{m-1}\alpha_{jm}}\cdots
 \xrightarrow{\alpha_{jm+m-1}}a_{(j+1)m}.
\]
Retaining morphism components indexed by \(jm\) defines an exact functor:
the cyclic morphism equations give functoriality, and coordinatewise
pushouts commute with the displayed composite because \(H\) is exact.
Rotation by \(m\) input places is one output rotation, so the functor is
\(C_n\)-equivariant.

Apply this formula at every
\(w_{k_0}S^{(q)}_{k_1,\ldots,k_q}A\).  The operators and subquotient
inclusions act coordinatewise and preserve the ordered blocks, so the
components form a morphism of \(\Sigma_\Delta\)-diagrams on the source side.
On the cyclic bar, use the lexicographic identification
\(\mathbb Z/n\times\{0,\ldots,m-1\}\cong\mathbb Z/mn\).
Bar associativity and enriched co-Yoneda compose each row of \(m\) graph
coefficients to \(X_{H^m}\).  Inner faces, the cyclic boundary face,
degeneracies, and rotation are preserved row by row.  On zero-simplices this
is the ordered lacing composite above, so the forward Dennis-trace square
commutes before realization.  The same calculation after each subquotient
inclusion treats every twisted-additivity summand.  After genuine suspension
and derived prolongation, inverting the backwards twisted-additivity
equivalence in the \(C_n\)-equivariant functor category yields
\eqref{eq:general-block-composition-clmpz-square}.

For \(d\mid n\), a fixed object word is a \(d\)-fold repetition of a shorter
word.  On its bar summand the block map is
\(\beta_{\mathbf a}^{\wedge d}\); the norm-diagonal square of
Lemma~\ref{lem:norm-diagonal-block-interchange} identifies its
geometric fixed map with the shorter block map \(b_{m,n/d}\), including the
residual action.  Nonfixed word orbits are induced from proper subgroups and
have trivial geometric \(C_d\)-fixed points, as in
\cite[Proposition~7.4]{CLMPZ}.  This proves
\eqref{eq:general-block-comparison-geometric-fixed}.

If \(r\mid n\), both source routes identify an \(n/r\)-fold repeated cycle
with the same shorter cycle.  On the target, the restriction map is the
inverse norm diagonal on the repeated word, so
Lemma~\ref{lem:norm-diagonal-block-interchange} gives compatibility
with block composition.  Thus the pre-prolongation zigzag commutes, and
functorial prolongation preserves it.
Finally, duplication repeats the ordered \(m\)-fold composite, proving
\eqref{eq:general-block-composition-duplication}.  Lexicographic
threefold indexing makes the two source groupings identical.  The canonical
relative-bar associator compares the target groupings and, by naturality,
commutes with the \(\Sigma_\Delta\)-, additivity-, norm-, and restriction
maps.  This proves the associativity statements and their naturality.

\section{Strict functorial models for perfect-module presentations}
\label{app:split-strictification}

The perfect-module assignment
\((A,f)\mapsto(\mathcal P_A,f_!)\) is naturally pseudofunctorial.  A split
Grothendieck construction replaces it by the strict functorial models used in
Propositions~\ref{thm:split-perfect-strictification}
and~\ref{lem:power-coherent-perfect-envelope}.  Their transition maps
commute strictly with the distinguished endofunctors, while Morita
localization recovers the original perfect-module diagrams.  Under this
localization, the strict identities recover the intrinsic iteration
equivalences of Lemma~\ref{lem:iteration-multiplicative-coherence}.
Homotopy-coherent \(B\mathbb N\)-diagrams are rigidified in
Proposition~\ref{prop:endofunctor-rigidification}.

\subsection{The perfect-module pseudofunctor}

For a presentation \(x=(A,f)\in\mathsf{Pres}_{\mathcal U}\), we use its
pointwise-cofibrant orthogonal realization from
Lemma~\ref{lem:orthogonal-morita-model}, again denoted by \((A,f)\).
The Waldhausen base of \(\mathcal P_x=\mathcal P_A\) consists of all
\(\mathcal U\)-small cofibrant perfect right \(A\)-modules and their ordinary
module maps.  Its spectral enrichment is specified below.  Put
\[
 T_x:=f_!\colon \mathcal P_x\longrightarrow\mathcal P_x.
\]

\begin{lemma}
\label{lem:many-object-perfect-model}
Let \(A\) be a \(\mathcal U\)-small pointwise-cofibrant orthogonal spectral
category.  Its cofibrant perfect modules admit a spectral Waldhausen
enrichment whose mapping spectra compute the derived module mapping spectra.
For every spectral functor \(u\colon A\to B\) between such categories,
ordinary extension of scalars restricts to an exact spectral functor
\[
 u_!\colon\mathcal P_A\longrightarrow\mathcal P_B
\]
preserving tensors with finite simplicial sets.  These functors have the
enriched co-Yoneda unit and composition isomorphisms.  The stable localization
of \(\mathcal P_A\) is \(\operatorname{Perf}(A)\), naturally in \(A\), and the
localized functor is \(\mathbf Lu_!\).
The Yoneda functor \(y_A\colon A\to\mathcal P_A\), sending \(a\) to
\(A(-,a)\), is a DK embedding and a Morita equivalence, with enriched
natural isomorphisms \(u_!y_A\cong y_Bu\).
\end{lemma}

\begin{proof}
Let
\[
 N\colon\operatorname{Sp}^{\mathcal O}
 \rightleftarrows\operatorname{Sp}^{\mathrm{EKMM}}:N^\#
\]
be the monoidal Quillen equivalence of
\cite[Chapter~I]{MandellMay}, with the positive stable model structure on
orthogonal spectra.  Write \(K=F_1S^1\), with its stable equivalence
\(K\to\mathbb S\).  Applying the strong symmetric-monoidal functor \(N\)
to the mapping spectra gives an EKMM-enriched category \(NA\).  Set
\begin{equation}\label{eq:perfect-module-derived-enrichment}
 J_AP=N(K\wedge P),\qquad
 \mathcal P_A(P,Q)=N^\#F_{NA}(J_AP,J_AQ).
\end{equation}
Here \(F_{NA}\) is the internal mapping spectrum of right \(NA\)-modules.
Composition comes from its composition and the lax symmetric-monoidal
structure of \(N^\#\); the base enrichment sends an ordinary module map to
its image under \(J_A\).  This is the many-object version of
\cite[Remark~3.7, Lemma~3.9, and Example~5.2]{CLMPZSW}.

The module model structures used here exist for arbitrary small enriched
categories by \cite[Theorem~7.2(1)]{SchwedeShipleyMonoidal}, including its
EKMM case.  Smashing a cofibrant \(A\)-module with \(K\) gives a positively
cofibrant \(A\)-module: this follows on the free generating cells
\(A(-,a)\wedge i\) from the positive cofibrancy of \(K\), and then by
pushouts and retracts.  The induced functor from positively cofibrant
\(A\)-modules to \(NA\)-modules is left Quillen, so \(J_AP\) is cofibrant.
Every \(NA\)-module is fibrant, since fibrations are detected objectwise and
every EKMM spectrum is fibrant.

For completeness, the change of enrichment gives a Quillen equivalence on
these module categories.  Use the positive projective structure on
\(A\)-modules and the adjunction obtained by applying \(N\) and \(N^\#\)
objectwise, with restriction along \(A\to N^\#NA\) on the right.  A
positively cofibrant module is objectwise cofibrant in the ordinary stable
structure, because \(A\) is pointwise cofibrant.  The unit is consequently
an objectwise stable equivalence by \cite[Chapter~I, Proposition~3.5]{MandellMay}.
The right adjoint preserves and detects weak equivalences by
\cite[Chapter~I, Lemma~3.3]{MandellMay}, so the adjunction
is a Quillen equivalence.  This is the argument for the positive and EKMM
models in \cite[Section~7, proof of Corollary~1.2]{SchwedeShipleyMonoidal}.
The ordinary and positive projective structures have the same weak
equivalences, and \(K\wedge P\to P\) is a weak equivalence for cofibrant
\(P\).  Thus \(J_A\) presents the resulting equivalence of module
\(\infty\)-categories.  It follows that
\eqref{eq:perfect-module-derived-enrichment} computes their mapping spectra
and restricts to the corresponding perfect objects.

Cofibrant perfect \(A\)-modules are closed under zero objects, retracts,
mapping cylinders, and pushouts along cofibrations, so their ordinary base
is a Waldhausen category.  The functor \(J_A\) preserves these pushouts and
cofibrations and sends weak equivalences between cofibrant modules to weak
equivalences.  In EKMM modules, the two-variable mapping spectrum between
cofibrant objects preserves weak equivalences in each variable.  Applied
contravariantly to a cofibration pushout it gives a pullback along a
fibration, hence a homotopy pullback.  Applied covariantly it also gives a
homotopy pullback, since the original homotopy pushout is a homotopy
pullback in this stable model category.  The right Quillen functor \(N^\#\)
preserves these homotopy pullbacks and weak equivalences between fibrant
objects.  Mapping spectra to and from zero are zero.  These observations
verify all three axioms of \cite[Definition~3.12]{CLMPZ}; smallness follows
on passing to the next universe.

Ordinary enriched extension of scalars is
\[
 (u_!P)(b)=\int^{a\in A}P(a)\wedge B(b,u(a)).
\]
It is a simplicially enriched left Quillen functor, preserves pushouts and
tensors, and carries representables to representables.  Hence it preserves
cofibrant perfect modules.  There is a natural isomorphism of \(NB\)-modules
\begin{equation}\label{eq:perfect-module-enrichment-base-change}
 \kappa_u\colon J_Bu_!\xrightarrow{\ \cong\ }(Nu)_!J_A,
\end{equation}
because \(N\) is strong symmetric monoidal and preserves the coends, and
extension of scalars commutes with the external factor \(K\).  The
EKMM-enriched functor \((Nu)_!\) induces
\[
 F_{NA}(J_AP,J_AQ)
 \longrightarrow
 F_{NB}((Nu)_!J_AP,(Nu)_!J_AQ).
\]
Conjugating by \(\kappa_u\) and applying \(N^\#\) defines the action of
\(u_!\) on \eqref{eq:perfect-module-derived-enrichment}.  It preserves
composition and the base enrichment by construction.

For composable \(u,v\), the two comparisons from \(J_Cv_!u_!\) to
\((N(vu))_!J_A\) are the same isomorphism of the iterated coend.  In
particular, the usual co-Yoneda compositor
\(c_{v,u}\colon v_!u_!\cong(vu)_!\) is enriched for the mapping spectra
\eqref{eq:perfect-module-derived-enrichment}.  The unit and pentagon
identities follow from those of the coend compositors.  This also proves
naturality of the comparison with \(\operatorname{Perf}(A)\).

The action on representable modules defines \(y_A\) for this enrichment.
On its mapping spectra, the comparison just proved is the derived Yoneda
equivalence
\[
 A(a,b)\xrightarrow{\ \sim\ }
 N^\#F_{NA}(J_AA(-,a),J_AA(-,b));
\]
equivalently, apply the proof of \cite[Lemma~3.10]{CLMPZSW} to each pair
\((a,b)\).  The representables generate the perfect modules under finite
cofibers and retracts, so \(y_A\) is Morita-dense.  Co-Yoneda gives
\(u_!A(-,a)\cong B(-,u(a))\), and
\eqref{eq:perfect-module-enrichment-base-change} makes these isomorphisms
enriched and compatible with composition.  The same comparison with
\(u=f^m\) identifies the graph coefficient on representables with
\(A(f^m-,-)\).
\end{proof}

\begin{lemma}
\label{lem:normalized-perfect-module-pseudofunctor}
The extension-of-scalars pseudofunctor on pointwise-cofibrant orthogonal
spectral categories is
pseudonaturally equivalent to a normal pseudofunctor for which
\[
 (\operatorname{id}_A)_!=\operatorname{id}_{\mathcal P_A}
\]
and all unit constraints are identities.  The nonidentity transition
functors are the usual enriched left Kan extensions, and their compositors
are the co-Yoneda isomorphisms
\(v_!u_!\simeq(vu)_!\).  This normalization is compatible with strict
endomorphism diagrams and with passage to derived functors.
\end{lemma}

\begin{proof}
Let \(\mathbb F\) be the enriched left Kan extension pseudofunctor, with
unit isomorphism
\(\lambda_A\colon\operatorname{id}_{\mathcal P_A}\xrightarrow{\sim}
\mathbb F(\operatorname{id}_A)\).  Define \(\mathbb F^{\mathrm n}(u)=
\mathbb F(u)\) for a nonidentity arrow and
\(\mathbb F^{\mathrm n}(\operatorname{id}_A)=
\operatorname{id}_{\mathcal P_A}\).  For every arrow \(u\), let
\[
 \alpha_u\colon\mathbb F^{\mathrm n}(u)\xrightarrow{\ \sim\ }\mathbb F(u)
\]
be the identity when \(u\) is nonidentity and \(\lambda_A\) when
\(u=\operatorname{id}_A\).  If \(u\) and \(v\) are composable, define the
normalized compositor by transport:
\[
 c^{\mathrm n}_{v,u}
 :=\alpha_{vu}^{-1}\,c_{v,u}\,
   (\alpha_v*\alpha_u)\colon
 \mathbb F^{\mathrm n}(v)\mathbb F^{\mathrm n}(u)
 \xrightarrow{\ \sim\ }\mathbb F^{\mathrm n}(vu).
\]
This formula also covers the case in which \(u,v\) are nonidentity but
\(vu\) is an identity.  The triangle identities imply
\(c^{\mathrm n}_{u,\operatorname{id}}=
 c^{\mathrm n}_{\operatorname{id},u}=\operatorname{id}\), and conjugating
the original pentagon by the \(\alpha_u\) proves the normalized pentagon.
The family \(\alpha\) is the required pseudonatural equivalence.  Since the
normalization changes no nonidentity functor and only transports structural
isomorphisms, it preserves strict endomorphism relations and induces the
same derived functors.
\end{proof}

\begin{lemma}
\label{lem:perfect-section-restriction}
Let \(I\) be either \(B\mathbb N\) or a finite directed cycle, and let
\(\mathcal G\colon I\to\operatorname{ModCat}_\Delta\) be the corresponding
diagram of orthogonal module model categories obtained from the symmetric
spectral presentations of Lemma~\ref{lem:orthogonal-morita-model}.
Assume that every transition functor preserves perfect objects.  Write
\(\operatorname{Sec}^{\operatorname{pc}}_{\operatorname{perf}}(\mathcal G)\)
for the full subcategory of projectively cofibrant strict sections with
perfect underlying objects, and write
\(\operatorname{End}_{\operatorname{perf}}(\mathcal G)\) for the Waldhausen
category of strict sections with cofibrant perfect underlying objects,
equipped with the objectwise weak equivalences and the Waldhausen structure
used by the CLMPZ twisted-endomorphism construction.  Then the inclusion
induces equivalences
\begin{align*}
 N\operatorname{Sec}^{\operatorname{pc}}_{\operatorname{perf}}(\mathcal G)
 [W^{-1}]
 &\simeq
 N\operatorname{End}_{\operatorname{perf}}(\mathcal G)[W^{-1}] \\
 &\simeq
 \operatorname{RLim}^{\operatorname{lax}}
   (\mathcal G_\infty)^{\operatorname{perf}}.
\end{align*}
It also induces natural equivalences of connective \(K\)-theory spectra
\[
 K\operatorname{End}_{\operatorname{perf}}(\mathcal G)
 \simeq
 K\operatorname{Sec}^{\operatorname{pc}}_{\operatorname{perf}}(\mathcal G)
 \simeq
 K\bigl(
   \operatorname{RLim}^{\operatorname{lax}}
   (\mathcal G_\infty)^{\operatorname{perf}}
  \bigr).
\]
After Dwyer--Kan localization, the same conclusions are natural for strict
maps of such diagrams whose vertex maps are simplicial left Quillen functors
preserving perfect objects and whose transition squares commute strictly.
\end{lemma}

\begin{proof}
First make the comparison on modules over the cofibrant symmetric spectral
categories \(Q_\Sigma A_i\) appearing before realization and prolongation
in Lemma~\ref{lem:orthogonal-morita-model}.  Denote this diagram by
\(\mathcal G^\Sigma\).  Its vertex model categories are simplicial and
combinatorial, so \cite[Proposition~3.2 and Theorem~1.1]{HarpazLax} give the
projective section model and identify its localization with the right-lax
limit.  We next transport this comparison to the orthogonal module models.

The change of enrichment in
Lemma~\ref{lem:orthogonal-morita-model} gives vertexwise simplicial Quillen
equivalences from \(\mathcal G^\Sigma\) to \(\mathcal G\).  The left
adjoints commute with the transition functors up to the co-Yoneda
isomorphisms, since realization and prolongation preserve the defining
coends.  These isomorphisms satisfy the composition constraints and induce
an adjunction on strict sections, whose right adjoint is given vertexwise
by the right adjoints and their mate transformations.  The orthogonal
section category has the projective model structure: its free section at
\(i\) has \(j\)-th value
\[
 \coprod_{\alpha:i\to j}\alpha_!P.
\]
Indeed, by extension--restriction adjunction, a section structure map
\(\alpha_!P_i\to P_j\) is a module map \(P_i\to\alpha^*P_j\).
Sections are therefore modules over the enriched path category obtained
by adjoining the arrows of \(I\) to the vertex spectral categories, with
the relations expressing these module maps.  Its free-module construction
is the displayed free-section construction.  The transferred module model
structure of \cite[Theorem~7.2(1)]{SchwedeShipleyMonoidal} thus has
objectwise fibrations and weak equivalences.  Projective
cofibrations are objectwise cofibrations, since every transition is left
Quillen.  Consequently a projectively cofibrant section is objectwise
cofibrant, and a projectively fibrant section is objectwise fibrant.  The
Quillen-equivalence criterion for the adjunction on sections is therefore
the criterion at each vertex.  This proves that it is a Quillen
equivalence, compatible with evaluation and with restriction of the
indexing category.  Harpaz's comparison now gives the asserted right-lax
limit for \(\mathcal G\) as well.  The subcategories of sections with
perfect underlying objects correspond, since the vertex equivalences
preserve compact objects.

Projective cofibrant replacement of a strict section is an objectwise weak
equivalence.  Its source has cofibrant perfect underlying objects whenever
the original section does; hence it belongs to
\(\operatorname{Sec}^{\operatorname{pc}}_{\operatorname{perf}}(\mathcal G)\).
It follows that the inclusion into
\(\operatorname{End}_{\operatorname{perf}}(\mathcal G)\) induces an
equivalence on Dwyer--Kan localizations.  Under Harpaz's equivalence, the
mapping space between two sections is the homotopy pullback imposing the
lacing equations.  For \(I=B\mathbb N\), this is precisely the
arrow-fibration model of \cite[Construction~2.5]{HNS}; the same argument for
the directed \(r\)-cycle gives the cyclic laced category.  Finite limits and
cofibers are computed vertexwise, so the resulting equivalence is exact.

To compare the Waldhausen \(K\)-theory spectrum used by CLMPZ with the
connective \(K\)-theory of the localized category, let
\(g=(g_i:P_i\to Q_i)\) be a morphism of strict sections.  Its
coordinatewise cylinder is
\[
 \operatorname{Cyl}(g_i)
 =Q_i\amalg_{P_i,g_i}(P_i\otimes\Delta^1).
\]
The transition functors are simplicial left Quillen functors; they preserve
tensors and pushouts on the cofibrant objects under consideration.  The
lacing equations for \(g\) therefore induce strict lacing maps on these
cylinders.  Thus every morphism factors functorially as a cofibration followed
by a weak equivalence, and the cylinder remains cofibrant and perfect in each
coordinate.  Weak equivalences are DHKS-saturated because they are detected
coordinatewise in module categories.  The hypotheses of
\cite[Lemma~7.11 and Corollary~7.12]{BGT} are consequently satisfied, and
Waldhausen \(K\)-theory agrees with the connective \(K\)-theory of the
localized stable \(\infty\)-category.  The same comparison applied to the
DK-equivalent projectively cofibrant subcategory gives the two displayed
equivalences of spectra.

The vertex functors preserve cofibrant perfect objects and weak equivalences
between them, so a strict map of diagrams as in the statement induces a
functor on the localized perfect section categories.  Harpaz's comparison
identifies it with the functor on right-lax limits, and naturality of the BGT
comparison for weakly exact functors gives naturality on \(K\)-theory.
\end{proof}

\begin{lemma}
\label{lem:strict-sections-iteration}
Under the equivalences of
Lemma~\ref{lem:perfect-section-restriction}, restriction of a strict
\(B\mathbb N\)-section along
\(\mu_m\colon B\mathbb N\to B\mathbb N\) corresponds to restriction of the
associated right-lax cone along \(\mu_m\).  In particular, the strict functor
that sends a structure arrow \(TP\to P\) to its ordered composite
\(T^mP\to P\) induces the intrinsic iteration functor
\[
 P_m\colon\operatorname{Lace}(\cC,\cM_F)
 \longrightarrow\operatorname{Lace}(\cC,\cM_{F^m}).
\]
This identification is natural in strict maps of endofunctor diagrams.  For
\(m,n\geq1\), the comparison for two successive restrictions is the canonical
equivalence \(P_n^{F^m}P_m^F\simeq P_{mn}^F\), and these comparisons are
unital and satisfy the associativity coherence.
\end{lemma}

\begin{proof}
Harpaz's equivalence between localized strict sections and right-lax limits is
natural both in the diagram and in the indexing category.  Precomposition by
\(\mu_m\) therefore gives the asserted commutative square of localized section
categories.  On a strict section it replaces \(TP\to P\) by the ordered
\(m\)-fold composite, so this functor is the one displayed in the statement.
The identities \(\mu_n\mu_m=\mu_{mn}\) identify two successive restrictions
with restriction along \(\mu_{mn}\); functoriality of precomposition supplies
the unit and associativity coherence.
\end{proof}

\begin{proposition}
\label{prop:perfect-module-pseudofunctor}
The assignment
\[
 x=(A,f)\longmapsto(\mathcal P_x,T_x)
\]
extends to a normal pseudofunctor from \(\mathsf{Pres}_{\mathcal U}\) to the
\(2\)-category of spectral Waldhausen categories with exact endofunctor.

For every strict morphism
\[
 u\colon x=(A,f)\longrightarrow y=(B,g),
\]
the induced exact spectral functor
\(u_!\colon\mathcal P_x\to\mathcal P_y\)
is equipped with a natural isomorphism
\begin{equation}\label{eq:perfect-module-beck-chevalley}
 \beta_u\colon u_!T_x\xRightarrow{\ \sim\ }T_yu_!.
\end{equation}
For composable \(u\) and \(v\), there are coherent compositors
\begin{equation}\label{eq:perfect-module-compositor}
 c_{v,u}\colon v_!u_!\xRightarrow{\ \sim\ }(vu)_!,
\end{equation}
with identity unit constraints.

After stable localization, this pseudofunctor induces the intrinsic
endofunctor diagram
\[
 x=(A,f)\longmapsto
 \bigl(\operatorname{Perf}(A),\mathbf Lf_!\bigr).
\]
Moreover, the twisted-endomorphism category of the chosen presentation
presents the corresponding
laced category:
\begin{equation}\label{eq:perfect-sections-lace}
 N\operatorname{End}
 \bigl({}_{T_x}\mathcal P_x/(\mathcal P_x)_{\operatorname{id}}\bigr)
 [W^{-1}]
 \simeq
 \operatorname{Lace}
 \bigl(\operatorname{Perf}(A),\cM_{\mathbf LT_x}\bigr),
\end{equation}
naturally in \(x\) in the pseudofunctorial sense.  The induced comparison on
connective algebraic \(K\)-theory is an equivalence
\begin{equation}\label{eq:perfect-sections-k-lace}
 K\operatorname{End}
 \bigl({}_{T_x}\mathcal P_x/(\mathcal P_x)_{\operatorname{id}}\bigr)
 \simeq
 K^{\operatorname{lace}}
 \bigl(\operatorname{Perf}(A),\cM_{\mathbf LT_x}\bigr).
\end{equation}
The analogous statements hold for every cyclic \(r\)-fold
twisted-endomorphism category, where the localized \(r\)-fold category is
the cyclic laced category
\(\mathcal L_r(\operatorname{Perf}(A),\mathbf LT_x)\) of
Definition~\ref{def:cyclic-laced-category}.
\end{proposition}

\begin{proof}
Use the exact extension-of-scalars functors of
Lemma~\ref{lem:many-object-perfect-model} and the normal pseudofunctor with
compositors \(c_{v,u}\) supplied by
Lemma~\ref{lem:normalized-perfect-module-pseudofunctor}.

For \(uf=gu\), define
\[
 \beta_u:=c_{g,u}^{-1}c_{u,f}\colon u_!f_!\xrightarrow{\ \sim\ }g_!u_!.
\]
The pseudofunctor coherence for the \(c_{v,u}\) implies the composition
coherence for the \(\beta_u\).

By Lemma~\ref{lem:many-object-perfect-model}, stable localization of
\((\mathcal P_A,f_!)\) is \((\operatorname{Perf}(A),\mathbf Lf_!)\).
Applying Lemma~\ref{lem:perfect-section-restriction} to
\(B\mathbb N\) gives
\eqref{eq:perfect-sections-lace} and
\eqref{eq:perfect-sections-k-lace}; applying it to the finite directed
\(r\)-cycle gives the cyclic statement.
\end{proof}

\subsection{Construction of
\texorpdfstring{\(\mathcal P_x^{\operatorname{st}}\)}{Pst}}

We now give the split model of
Proposition~\ref{thm:split-perfect-strictification}.  The transition maps on
mapping spectra use the pseudofunctor compositors, which make the transition
functors strictly associative.

\begin{proof}[Detailed proof of
Proposition~\ref{thm:split-perfect-strictification}]
For \(x\in\mathsf{Pres}_{\mathcal U}\), let the unreduced split category have
\[
 \operatorname{ob}\mathcal P_x^{\operatorname{spl}}
 =
 \left\{
  (a\colon z\to x,\;P\in\mathcal P_z)
 \right\},
\]
where \(a\) ranges over the morphisms of \(\mathsf{Pres}_{\mathcal U}\)
with target \(x\).  For objects \((a,P)\) and \((b,Q)\), put
\begin{equation}\label{eq:split-perfect-mapping-spectrum}
 \mathcal P_x^{\operatorname{spl}}
 \bigl((a,P),(b,Q)\bigr)
 :=
 \mathcal P_x(a_!P,b_!Q).
\end{equation}
Composition is inherited from \(\mathcal P_x\).  Define evaluation by
\[
 e_x^{\operatorname{spl}}\colon
 \mathcal P_x^{\operatorname{spl}}\longrightarrow\mathcal P_x,
 \qquad
 e_x^{\operatorname{spl}}(a,P)=a_!P,
\]
and declare cofibrations and weak equivalences to be those created by
\(e_x^{\operatorname{spl}}\).  This functor is spectrally fully faithful,
creates zero objects
and pushouts, and hits every object through the section
\(P\mapsto(\operatorname{id}_x,P)\), so it creates the
spectral-Waldhausen structure.  We choose
\((\operatorname{id}_x,0_{\mathcal P_x})\) as its zero object.

Applying Lemma~\ref{lem:zero-reduction}, set
\[
 \mathcal P_x^{\operatorname{st}}
 :=(\mathcal P_x^{\operatorname{spl}})^\circ .
\]
Thus \(\mathcal P_x^{\operatorname{st}}\) has a single chosen zero object
\(0_x\), and every mapping spectrum to or from \(0_x\) is literally zero.
Let \(\mathcal P_x^\circ\) be the zero reduction of
\(\mathcal P_x\), and let
\[
 e_x\colon\mathcal P_x^{\operatorname{st}}\longrightarrow\mathcal P_x^\circ
\]
be the composite of the inclusion with
\(e_x^{\operatorname{spl}}\).  Every retained object \((a,P)\) is actually
isomorphic to
\((\operatorname{id}_x,a_!P)\), represented by the identity of \(a_!P\);
every omitted object evaluates to a zero object and is represented by
\(0_x\).  Hence \(e_x\) is a DK equivalence and is essentially surjective by
actual isomorphisms.

Define \(j_x\colon\mathcal P_x^\circ\to
\mathcal P_x^{\operatorname{st}}\) by
\[
 j_x(P)=
 \begin{cases}
  (\operatorname{id}_x,P),&P\text{ is not a zero object},\\
  0_x,&P=0.
 \end{cases}
\]
On mapping spectra it is induced by the identity when both displayed
objects are nonzero and by the unique map to a zero mapping spectrum
otherwise.  It is spectrally fully faithful because the mapping spectra to
or from the chosen zero are literally zero.  Thus \(j_x\) is a DK and
Morita equivalence, proving part~(2).

The split categories are functorial in the presentation.  For
\(u\colon x\to y\), put
\(u_*^{\operatorname{spl}}(a,P)=(ua,P)\); on mapping spectra, send
\(h\colon a_!P\to b_!Q\) to the composite
\begin{equation}\label{eq:split-transition-map}
 (ua)_!P
 \xrightarrow{\,c_{u,a}^{-1}\,}
 u_!a_!P
 \xrightarrow{\,u_!(h)\,}
 u_!b_!Q
 \xrightarrow{\,c_{u,b}\,}
 (ub)_!Q.
\end{equation}
The pseudofunctor pentagon reduces the composite defining
\(v_*^{\operatorname{spl}}u_*^{\operatorname{spl}}\) to the same conjugate
as \((vu)_*^{\operatorname{spl}}\), including the identity case.  The functor
\(u_*^{\operatorname{spl}}\) preserves the class of zero objects, although it
need not preserve the chosen representative.  It is therefore a morphism in
\(\mathsf{SpWaldCat}^{\mathrm z}\), and we may put
\[
 u_*^{\operatorname{st}}
 :=(u_*^{\operatorname{spl}})^\circ.
\]
It sends \(0_x\) to \(0_y\); for a nonzero \((a,P)\), it is \((ua,P)\)
when the latter is nonzero and \(0_y\) otherwise.  Since zero reduction is
an ordinary functor, the equality
\((vu)_*^{\operatorname{spl}}
 =v_*^{\operatorname{spl}}u_*^{\operatorname{spl}}\) gives
\((vu)_*^{\operatorname{st}}=v_*^{\operatorname{st}}u_*^{\operatorname{st}}\)
literally.

The distinguished endofunctor on the split category is
\(T_x^{\operatorname{spl}}(a\colon z\to x,P)=(a,T_zP)\), where
\(z=(B,h)\) and \(T_z=h_!\).  For
\(b\colon w=(C,k)\to x\), put \(T_w=k_!\).  On a mapping-spectrum element
\(v\colon a_!P\to b_!Q\), set
\begin{equation}\label{eq:split-endofunctor-map}
 a_!T_zP
 \xrightarrow{\,\beta_a\,}
 T_xa_!P
 \xrightarrow{\,T_x(v)\,}
 T_xb_!Q
 \xrightarrow{\,\beta_b^{-1}\,}
 b_!T_wQ.
\end{equation}
The endofunctor \(T_x^{\operatorname{spl}}\) preserves the class of zero
objects, so it is a morphism in \(\mathsf{SpWaldCat}^{\mathrm z}\).  We may
therefore set
\[
 T_x^{\operatorname{st}}=(T_x^{\operatorname{spl}})^\circ.
\]
Thus \(T_x^{\operatorname{st}}(0_x)=0_x\), while the displayed formula
holds for a nonzero object unless its image is zero, in which case the image
is \(0_x\).
Compatibility with the transition functors is the pseudofunctor identity
\begin{equation}\label{eq:split-beta-coherence}
 \beta_{ua}(c_{u,a}T_z)
 =(T_yc_{u,a})(\beta_u a_!)u_!(\beta_a).
\end{equation}
Substituting \(\beta=c^{-1}c\) reduces both sides to the same composite of
compositors.  After zero reduction, naturality of \(\beta_u\) shows that
\(u_*^{\operatorname{st}}T_x^{\operatorname{st}}\) and
\(T_y^{\operatorname{st}}u_*^{\operatorname{st}}\) agree strictly.  This proves
the literal relations
\eqref{eq:split-perfect-literal-relations} and hence part~(1).

The exact endofunctor \(T_x\) induces
\(T_x^\circ\colon\mathcal P_x^\circ\to\mathcal P_x^\circ\).
The definition of \(j_x\), including its zero case, gives
\[
 j_xT_x^\circ=T_x^{\operatorname{st}}j_x
\]
on objects and mapping spectra; hence \(j_x\) is a strict morphism of
twistings \eqref{eq:standard-to-split-envelope}.  The evaluation
functors \(e_x\), the compositors \(c_{u,a}\), and the unique isomorphisms
between zero objects form a pseudonatural equivalence from the zero-reduced
split diagram to the zero reduction of the perfect-module pseudofunctor.
The inclusions \(\mathcal P_x^\circ\to\mathcal P_x\) recover the original
pseudofunctor after Morita localization.
Lemma~\ref{lem:pseudonatural-localization} turns it into the natural
equivalence
\eqref{eq:split-evaluation-localized}, proving part~(3).

For part~(4), evaluation on twisted-endomorphism categories is
\begin{equation}\label{eq:split-end-evaluation}
 \epsilon_x(a,P,h)=\bigl(a_!P,\,h\beta_a^{-1}\bigr),
\end{equation}
and sends the unique zero object to the unique zero section.  It is an exact
equivalence whose strict section is induced by \(j_x\).  The constraints
\(c_{u,a}\) and \eqref{eq:split-beta-coherence} make
the \(\epsilon_x\) pseudonatural in \(x\).  Composing with the
strict-section comparison of
Proposition~\ref{prop:perfect-module-pseudofunctor} gives the
natural equivalence \eqref{eq:split-source-lace-natural} for
\(r=1\); the same argument for strict sections of the free directed
\(r\)-cycle gives the cyclic comparison for every \(r\).
\end{proof}

\subsection{A strict functorial model compatible with iteration}

The same split category admits an endofunctor compatible with passage from
\(x=(A,f)\) to \(x^{[m]}=(A,f^m)\).  This gives the power-compatible model
\((\widetilde{\mathcal P}_x,S_x)\) of
Proposition~\ref{lem:power-coherent-perfect-envelope}.

\begin{proof}[Detailed proof of
Proposition~\ref{lem:power-coherent-perfect-envelope}]
The underlying spectral Waldhausen category is
\(\mathcal P_x^{\operatorname{st}}\):
\begin{align}
 \operatorname{ob}\widetilde{\mathcal P}_x
 &=\operatorname{ob}\mathcal P_x^{\operatorname{st}},          \label{eq:power-envelope-objects}\\
 \widetilde{\mathcal P}_x
 (X,Y)
 &=\mathcal P_x^{\operatorname{st}}(X,Y).                       \label{eq:power-envelope-mapping}
\end{align}
If \(z=(B,h)\), define the distinguished endofunctor on objects by
\begin{equation}\label{eq:power-envelope-endofunctor}
 S_x(a\colon z\to x,P)=(ah,P);
\end{equation}
this formula is interpreted through zero reduction: \(S_x(0_x)=0_x\), and
the right-hand side is replaced by \(0_x\) when \((ah,P)\) is a zero object.
It is well defined because \(ah\colon z\to x\) is again a morphism of
presentations and exact functors preserve zero objects.  On mapping spectra,
for nonzero source and target, take
\(v\colon a_!P\to b_!Q\) with \(a\colon(B,h)\to(A,f)\) and
\(b\colon(C,k)\to(A,f)\), put
\begin{equation}\label{eq:power-envelope-map-on-homs}
 (ah)_!P
 \xrightarrow{\,c_{f,a}^{-1}\,}
 f_!a_!P
 \xrightarrow{\,f_!(v)\,}
 f_!b_!Q
 \xrightarrow{\,c_{f,b}\,}
 (fb)_!Q=(bk)_!Q,
\end{equation}
using \(ah=fa\) and \(fb=bk\); use the unique map to a zero mapping spectrum
if either image is zero.  This is the zero reduction of an exact functor
conjugate to \(f_!\), so it defines an exact spectral endofunctor.  The
transition functors are the zero reductions of
\(\widetilde u_*^{\operatorname{spl}}(a,P)=(ua,P)\) and are defined on
nonzero mapping spectra by the same
conjugation formula \eqref{eq:split-transition-map} used for
\(\mathcal P_x^{\operatorname{st}}\); composition on the nose is proved in
the same way, and the pentagon
together with the equality \(uf=f'u\) gives
\(\widetilde u_*S_x=S_y\widetilde u_*\)
literally.  This proves the literal relations of part~(1) of
Proposition~\ref{lem:power-coherent-perfect-envelope}.

Iteration gives
\(S_x^m(a\colon z\to x,P)=(ah^m,P)\) on nonzero objects, with zero images
replaced by \(0_x\).  Since \(ah^m=f^ma\), the formula
\[
 q_{m,x}\colon
 0_x\longmapsto0_{x^{[m]}},\qquad
 (a\colon z\to x,P)\longmapsto
 (a\colon z^{[m]}\to x^{[m]},P)
\]
defines the functor \(q_{m,x}\) of
\eqref{eq:power-comparison-map}.  On mapping
spectra, the comparison between the iterated composite of the compositors
and the single compositor for \(f^m\) is the canonical associativity map
supplied by the normal pseudofunctor; naturality of that map shows that
\(q_{m,x}\) is a strict morphism of twistings.  Because \(q_{m,x}\) is the
identity on the mapping spectra
\eqref{eq:power-envelope-mapping}, the relations
\eqref{eq:power-comparison-coherence} are literal.  Normality and the
identities \(z^{[1]}=z\), \(x^{[1]}=x\) give
\(q_{1,x}=\operatorname{id}\).  Essential surjectivity holds because the
target zero object is in the image and every nonzero target object is
isomorphic, through the identity of its evaluated module, to an object with
first coordinate \(\operatorname{id}_{x^{[m]}}\).

The two endofunctors on the common underlying spectral Waldhausen category
are naturally isomorphic:
\begin{equation}\label{eq:ordinary-powered-action-comparison}
 \theta_x\colon T_x^{\operatorname{st}}\xRightarrow{\ \sim\ }S_x,
 \qquad
 \theta_{x,(a,P)}=c_{a,h}\colon a_!h_!P\xrightarrow{\ \sim\ }(ah)_!P.
\end{equation}
If the two displayed images are collapsed to \(0_x\), including at
\(0_x\) itself, the component of \(\theta_x\) is the identity of \(0_x\).
Iterating \(\theta_x\) gives a natural isomorphism
\(\theta_x^{(m)}\colon
 (T_x^{\operatorname{st}})^m\xRightarrow{\ \sim\ }S_x^m\).
To compare these powers by strict morphisms of twistings, let
\(\mathcal B_x^{(m)}\) have one zero object and two copies
\((X,1),(X,2)\) of every nonzero object
\(X\in\mathcal P_x^{\operatorname{st}}\).  Mapping spectra to or from the
zero object are zero, and the mapping spectrum between two nonzero copies is
\(\mathcal P_x^{\operatorname{st}}(X,Y)\), with composition inherited from
\(\mathcal P_x^{\operatorname{st}}\).  Equivalently, this is the zero
reduction of the unreduced two-copy category.  It is a spectral Waldhausen
category, and the two copy inclusions are strictly zero-preserving exact
functors.  Its endofunctor \(U_x^{(m)}\) is obtained by conjugating
\((T_x^{\operatorname{st}})^m\) by the identity on the first copy and by
\(\theta_x^{(m)}\) on the second.  Explicitly, put
\(\epsilon_{1,X}=\operatorname{id}\) and
\(\epsilon_{2,X}=\theta_{x,X}^{(m)}\).  For every mapping-spectrum element
\(v\colon(X,i)\to(Y,j)\), including the mixed-copy cases, define
\[
 U_x^{(m)}(v)=
 \epsilon_{j,Y}\,(T_x^{\operatorname{st}})^m(v)\,
 \epsilon_{i,X}^{-1}.
\]
Naturality of \(\theta_x^{(m)}\) makes this an enriched functor.  The
inclusions of the first and second copies give the two maps in
\eqref{eq:power-envelope-bridge}.  Each is fully faithful and
essentially surjective, with the identity maps providing isomorphisms between
the two copies.

For a morphism of presentations \(u\colon x\to y\), let
\(\mathcal B_u^{(m)}\) act on both copies by \(u_*^{\operatorname{st}}\) and on
all mapping spectra by the conjugation formula
\eqref{eq:split-transition-map}.  The pseudofunctor pentagon gives
\begin{equation}\label{eq:powered-bridge-presentation-naturality}
 u_*^{\operatorname{st}}\theta_x^{(m)}
 =\theta_y^{(m)}u_*^{\operatorname{st}},
\end{equation}
so \(\mathcal B_u^{(m)}U_x^{(m)}=U_y^{(m)}\mathcal B_u^{(m)}\) literally.
These functors preserve the Waldhausen structures and chosen zero objects,
so \(\mathcal B_u^{(m)}\) and the two copy inclusions are strict morphisms of
zero-reduced twistings.  Their defining squares commute on the nose,
naturally in the presentation.

The identity maps between the two copies give enriched natural isomorphisms on
every cyclic source category.  For \(q_{m,x}\), the isomorphism from each
target object to one in its image extends vertexwise to cyclic lacing strings.
Since \(q_{m,x}\) is the identity on mapping spectra and strictly intertwines
the endofunctors, it gives an equivalence on every cyclic source.  These
equivalences are compatible with rotation, repetition, and the fixed-cycle
maps, and hence induce \(K\)-equivalences.  The two inclusions and \(q_{m,x}\)
are also DK equivalences of the input spectral Waldhausen categories.  Since
they intertwine the distinguished endofunctors, their graph-coefficient maps are
coefficient Morita equivalences.  They therefore satisfy all the hypotheses
of
Proposition~\ref{prop:replacement-comparison-diagrams}.  After
\(Q_{\mathrm{red}}\), they induce equivalences of CLMPZ trace arrows, natural
in the presentation.

The compatibility of these equivalences with powers is recorded in
Lemma~\ref{lem:powered-bridge-multiplicative-coherence} below.
\end{proof}

\begin{lemma}
\label{lem:powered-bridge-multiplicative-coherence}
The two-copy spans \eqref{eq:power-envelope-bridge} for \((x,m)\) and
\((x^{[m]},1)\), together with \(q_{m,x}\), induce a natural equivalence in
the Morita localization of CLMPZ trace arrows,
\[
 \chi_{m,x}\colon
 \mathcal T_{\mathrm{CLMPZ}}
  (\mathcal P_x^{\operatorname{st}},(T_x^{\operatorname{st}})^m)
 \simeq
 \mathcal T_{\mathrm{CLMPZ}}
  (\mathcal P_{x^{[m]}}^{\operatorname{st}},
   T_{x^{[m]}}^{\operatorname{st}}).
\]
The equivalence \(\chi_{1,x}\) is the identity.  For \(m,n\geq1\), restrict
\(\chi_{m,x}\) along \(\mu_n\) and then compose it with
\(\chi_{n,x^{[m]}}\).  The resulting equivalence is naturally equivalent to
\(\chi_{mn,x}\).  These comparisons are natural in \(x\),
unital, and satisfy the associativity coherence for three factors.
On a fixed zero-reduced spectral Waldhausen category, the same two-copy
construction associates an equivalence of CLMPZ trace arrows to any natural
isomorphism between strictly zero-preserving exact spectral endofunctors.
\end{lemma}

\begin{proof}
Proposition~\ref{prop:replacement-comparison-diagrams} gives the
equivalences after replacement.  Restriction along \(\mu_n\) identifies their
left endpoints: it
carries \((T_x^{\operatorname{st}})^m\) to
\((T_x^{\operatorname{st}})^{mn}\) and the corresponding model over
\(x^{[m]}\) to its \(n\)-th power.  Both routes reassociate the same ordered
string of pseudofunctor compositors.
The normalized unit constraints give \(\chi_{1,x}=\operatorname{id}\), the
pentagon compares the binary routes, and the same pentagon coherence identifies
the two triple comparisons.  The literal identity
\(q_{n,x^{[m]}}q_{m,x}=q_{mn,x}\) identifies their right legs.
\end{proof}

\begingroup
\renewcommand{\bibliofont}{\footnotesize}
\raggedright
\raggedbottom
\bibliographystyle{alpha}
\bibliography{references}
\endgroup

\end{document}